\documentclass[a4paper,twoside,leqno]{article}

\usepackage[english]{babel}
\usepackage[utf8]{inputenc}
\usepackage{csquotes}
\usepackage{amsmath}
\usepackage{amsthm}
\usepackage{amsfonts}
\usepackage{amssymb}
\usepackage{fancyhdr}
\usepackage[mathscr]{eucal}
\usepackage{pdfpages}
\usepackage{multicol}
\usepackage[all]{xy}
\usepackage{pgf,tikz}
\usepackage[g]{esvect}
\usepackage{leftidx}
\usetikzlibrary{arrows}
\usepackage[pdftex,pdfborder={0 0 0},linktoc=all]{hyperref}

\theoremstyle{plain}
	\newtheorem{thm}{Theorem}[section]
	\newtheorem{cor}[thm]{Corollary}
	\newtheorem{lem}[thm]{Lemma}
	\newtheorem{prop}[thm]{Proposition}

\theoremstyle{definition}
	\newtheorem{dfn}[thm]{Definition}
	\newtheorem{ntn}[thm]{Notation}
	\newtheorem{dfns}[thm]{Definitions}

\theoremstyle{remark}
	\newtheorem{rem}[thm]{Remark}

\numberwithin{equation}{section}

\newcommand{\N}{\mathbb{N}}
\newcommand{\Z}{\mathbb{Z}}
\newcommand{\R}{\mathbb{R}}
\newcommand{\C}{\mathbb{C}}
\newcommand{\dx}{\dmesure\!}
\newcommand{\deron}[2]{\frac{\partial #1}{\partial #2}}

\newcommand{\esp}[2][]{\mathbb{E}_{#1}\!\left[ #2 \right]}
\newcommand{\espcond}[3][]{\mathbb{E}_{#1}\!\left[ #2 \rule{0pt}{4mm}\mvert #3 \right]}
\newcommand{\loi}[1]{\,\dmesure\! P_{#1}}
\newcommand{\mvert}{\mathrel{}\middle|\mathrel{}}
\newcommand{\norm}[1]{\left\lvert #1 \right\rvert}
\newcommand{\Norm}[1]{\left\lVert #1 \right\rVert}
\newcommand{\odet}[1]{\norm{\det\! ^\perp\left(#1\right)}}
\newcommand{\prsc}[2]{\left\langle #1\,, #2 \right\rangle}
\newcommand{\rmes}[1]{\norm{\dmesure\! V_{#1}}}
\newcommand{\trans}{\leftidx{^\text{t}}}
\newcommand{\vol}[1]{\Vol\left(#1\right)}

\renewcommand{\S}{\mathbb{S}}
\renewcommand{\P}{\mathbb{P}}
\renewcommand{\epsilon}{\varepsilon}
\renewcommand{\geq}{\geqslant}
\renewcommand{\leq}{\leqslant}

\DeclareMathOperator{\cov}{Cov}
\DeclareMathOperator{\dmesure}{d}
\DeclareMathOperator{\End}{End}
\DeclareMathOperator{\id}{id}
\DeclareMathOperator{\Id}{Id}
\DeclareMathOperator{\II}{II}
\DeclareMathOperator{\Sym}{Sym}
\DeclareMathOperator{\tr}{Tr}
\DeclareMathOperator{\var}{Var}
\DeclareMathOperator{\Vol}{Vol}

\author{Thomas Letendre \thanks{Thomas Letendre, École Normale Supérieure de Lyon, Unité de Mathématiques Pures et Appliquées, UMR CNRS 5669, 46 allée d'Italie, 69634 Lyon Cedex 07, France; \newline e-mail address: \url{thomas.letendre@ens-lyon.fr}.}}
\date{\today}
\title{Expected volume and Euler characteristic of random submanifolds}

\begin{document}

\maketitle

% Forematter
\begin{abstract}
In a closed manifold of positive dimension $n$, we estimate the expected volume and Euler characteristic for random submanifolds of codimension $r\in \{1,\dots,n\}$ in two different settings. On one hand, we consider a closed Riemannian manifold and some positive $\lambda$. Then we take $r$ independent random functions in the direct sum of the eigenspaces of the Laplace-Beltrami operator associated to eigenvalues less than $\lambda$ and consider the random submanifold defined as the common zero set of these $r$ functions. We compute asymptotics for the mean volume and Euler characteristic of this random submanifold as $\lambda$ goes to infinity. On the other hand, we consider a complex projective manifold defined over the reals, equipped with an ample line bundle $\mathcal{L}$ and a rank $r$ holomorphic vector bundle $\mathcal{E}$ that are also defined over the reals. Then we get asymptotics for the expected volume and Euler characteristic of the real vanishing locus of a random real holomorphic section of $\mathcal{E}\otimes\mathcal{L}^d$ as $d$ goes to infinity. The same techniques apply to both settings.
\end{abstract}



\section{Introduction}
\label{section intro}

Zeros of random polynomials were first studied by Bloch and P\`olya \cite{BP1932} in the early 30s. About ten years later, Kac \cite{Kac1943} obtained a sharp asymptotic for the expected number of real zeros of a polynomial of degree $d$ with independent standard Gaussian coefficients, as $d$ goes to infinity. This was later generalized to other distributions by Kostlan in \cite{Kos1993}. In particular, he introduced a normal distribution on the space of homogeneous polynomials of degree $d$ --- known as the Kostlan distribution --- which is more geometric, in the sense that it is invariant under isometries of $\C \P^1$. Bogomolny, Bohigas and Leboeuf \cite{BBL1996} showed that this distribution corresponds to the choice of $d$ independent roots, uniformly distributed in the Riemann sphere.

In higher dimension, the question of the number of zeros can be generalized in at least two ways. What is the expected volume of the zero set? And what is its expected Euler characteristic? More generally, one can ask what are the expected volume and Euler characteristic of a random submanifold obtained as the zero set of some Gaussian field on a Riemannian manifold. In this paper, we provide an asymptotic answer to these questions in the case of Riemannian random waves and in the case of real algebraic manifolds.

Let us describe our frameworks and state the main results of this paper. See section~\ref{section random submanifolds} for more details. Let $(M,g)$ be a closed (that is compact without boundary) smooth Riemannian manifold of positive dimension~$n$, equipped with the Riemannian measure $\rmes{M}$ associated to $g$ (defined below \eqref{equation riemannian measure}). This induces a $L^2$-inner product on $\mathcal{C}^\infty(M)$ defined by:
\begin{equation}
\label{equation prsc function}
\forall \phi, \psi \in \mathcal{C}^\infty(M), \qquad \prsc{\phi}{\psi} = \int_{x\in M} \phi(x)\psi(x) \rmes{M}.
\end{equation}
It is well-known that the subspace $V_\lambda \subset \mathcal{C}^\infty(M)$ spanned by the eigenfunctions of the Laplacian associated to eigenvalues smaller than $\lambda$ has finite dimension. Let $1 \leq r \leq n$ and let $f^{(1)},\dots,f^{(r)} \in V_\lambda$ be independent standard Gaussian vectors, then we denote by $Z_f$ the zero set of $f=(f^{(1)},\dots,f^{(r)})$. Then, for $\lambda$ large enough, $Z_f$ is almost surely a submanifold of $M$ of codimension $r$ (see section~\ref{section random submanifolds} below) and we denote by $\vol{Z_f}$ its Riemannian volume for the restriction of $g$ to $Z_f$. We also denote by $\chi(Z_f)$ its Euler characteristic.

\begin{thm}
\label{theorem expected volume harmonic case}
Let $(M,g)$ be a closed Riemannian manifold of dimension $n$. Let $V_\lambda$ be the direct sum of the eigenspaces of the Laplace-Beltrami operator associated to eigenvalues smaller than $\lambda$. Let $f^{(1)},\dots,f^{(r)}$ be $r$ independent standard Gaussian vectors in $V_\lambda$, with $1 \leq r \leq n$. Then the following holds as $\lambda$ goes to infinity:
\begin{equation*}
\esp{\vol{Z_f}} = \left( \frac{\lambda}{n+2} \right)^{\frac{r}{2}} \vol{M} \frac{\vol{\S^{n-r}}}{\vol{\S^n}} +O\!\left(\lambda^\frac{r-1}{2}\right).
\end{equation*}
\end{thm}

\noindent
Here and throughout this paper, $\esp{ \ \cdot \ }$ denotes the mathematical expectation of the quantity between the brackets and $\S^n$ is, as usual, the unit Euclidean sphere in $\R^{n+1}$.

If $n-r$ is odd, $Z_f$ is almost surely a smooth manifold of odd dimension. In this case, $\chi(Z_f)=0$ almost surely. If $n-r$ is even, we get the following result.

\begin{thm}
\label{theorem expected Euler characteristic harmonic case}
Let $(M,g)$ be a closed Riemannian manifold of dimension $n$. Let $V_\lambda$ be the direct sum of the eigenspaces of the Laplace-Beltrami operator associated to eigenvalues smaller than $\lambda$. Let $f^{(1)},\dots,f^{(r)}$ be $r$ independent standard Gaussian vectors in $V_\lambda$, with $1 \leq r \leq n$. Then, if $n-r$ is even, the following holds as $\lambda$ goes to infinity:
\begin{equation*}
\esp{\chi\left(Z_f\right)} = \left(-1\right)^{\frac{n-r}{2}} \left( \frac{\lambda}{n+2} \right)^\frac{n}{2} \vol{M} \frac{\vol{\S^{n-r+1}}\vol{\S^{r-1}}}{\pi \vol{\S^n}\vol{\S^{n-1}}} + O\!\left(\lambda^\frac{n-1}{2}\right).
\end{equation*}
\end{thm}

We also consider the framework of the papers \cite{GW2014b,GW2014c} by Gayet and Welschinger, see section~\ref{subsection real algebraic setting} for more details. Let $\mathcal{X}$ be a smooth complex projective manifold of complex dimension $n$. Let $\mathcal{L}$ be an ample holomorphic line bundle over $\mathcal{X}$ and $\mathcal{E}$ be a holomorphic vector bundle over $\mathcal{X}$ of rank $r$. We assume that $\mathcal{X}$, $\mathcal{L}$ and $\mathcal{E}$ are equipped with compatible real structures and that the real locus $\R \mathcal{X}$ of $\mathcal{X}$ is non-empty.

Let $h_\mathcal{L}$ denote a Hermitian metric on $\mathcal{L}$ with positive curvature $\omega$ and $h_\mathcal{E}$ denote a Hermitian metric on $\mathcal{E}$. Both metrics are assumed to be compatible with the real structures. Then $\omega$ is a Kähler form and it induces a Riemannian metric $g$ and a volume form $\dx V_\mathcal{X}=\frac{\omega^n}{n !}$ on $\mathcal{X}$. For any $d \in \N$, the space of smooth sections of $\mathcal{E}\otimes \mathcal{L}^d$ is equipped with a $L^2$-inner product similar to \eqref{equation prsc function} (see section~\ref{subsection real algebraic setting}).

Let $\R H^0(\mathcal{X},\mathcal{E}\otimes \mathcal{L}^d)$ denote the space of real global holomorphic sections of $\mathcal{E}\otimes \mathcal{L}^d$. This is a Euclidean space for the above inner product. Let $s$ be a standard Gaussian section in $\R H^0(\mathcal{X},\mathcal{E}\otimes \mathcal{L}^d)$, then we denote by $Z_s$ the real part of its zero set. Once again, for $d$ large enough, $Z_s$ is almost surely a smooth submanifold of $\R \mathcal{X}$ of codimension $r$. Let $\vol{Z_s}$ denote the Riemannian volume of $Z_s$ and $\chi(Z_s)$ denote its Euler characteristic. We get the analogues of Theorems~\ref{theorem expected volume harmonic case} and~\ref{theorem expected Euler characteristic harmonic case} in this setting.

\begin{thm}
\label{theorem expected volume algebraic case}
Let $\mathcal{X}$ be a complex projective manifold of dimension $n$ defined over the reals and $r \in \{1,\dots,n\}$. Let $\mathcal{L}$ be an ample holomorphic Hermitian line bundle over $\mathcal{X}$ and $\mathcal{E}$ be a rank $r$ holomorphic Hermitian vector bundle over $\mathcal{X}$, both equipped with real structures compatible with the one on $\mathcal{X}$. Let $s$ be a standard Gaussian vector in $\R H^0(\mathcal{X},\mathcal{E}\otimes \mathcal{L}^d)$. Then the following holds as $d$ goes to infinity:
\begin{equation*}
\esp{\vol{Z_s}} = d^\frac{r}{2} \vol{\R \mathcal{X}} \frac{\vol{\S^{n-r}}}{\vol{\S^n}}  + O\!\left(d^{\frac{r}{2}-1}\right).
\end{equation*}
\end{thm}

\begin{thm}
\label{theorem expected Euler characteristic algebraic case}
Let $\mathcal{X}$ be a complex projective manifold of dimension $n$ defined over the reals and $r \in \{1,\dots,n\}$. Let $\mathcal{L}$ be an ample holomorphic Hermitian line bundle over $\mathcal{X}$ and $\mathcal{E}$ be a rank $r$ holomorphic Hermitian vector bundle over $\mathcal{X}$, both equipped with real structures compatible with the one on $\mathcal{X}$. Let $s$ be a standard Gaussian vector in $\R H^0(\mathcal{X},\mathcal{E}\otimes \mathcal{L}^d)$. Then, if $n-r$ is even, the following holds as $d$ goes to infinity:
\begin{equation*}
\esp{\chi\left(Z_s\right)} = \left(-1\right)^{\frac{n-r}{2}} d^\frac{n}{2} \vol{\R \mathcal{X}} \frac{\vol{\S^{n-r+1}}\vol{\S^{r-1}}}{\pi \vol{\S^n}\vol{\S^{n-1}}} + O\!\left(d^{\frac{n}{2}-1}\right).
\end{equation*}
\end{thm}

In the case of random eigenfunctions of the Laplacian, Theorem~\ref{theorem expected volume harmonic case} was already known to Bérard \cite{Ber1985} for hypersurfaces. See also \cite[thm.~1]{Zel2009} where Zelditch shows that, in the case of hypersurfaces,
\begin{equation*}
\sqrt{\frac{n+2}{\lambda}}\esp{Z_f} \xrightarrow[\lambda \to +\infty]{} \frac{\vol{\S^{n-1}}}{\vol{\S^n}}\rmes{M}
\end{equation*}
in the sense of the weak convergence of measures. He also proves a similar result in the case of band limited eigenfunctions.

Let us discuss Theorems~\ref{theorem expected volume algebraic case} and~\ref{theorem expected Euler characteristic algebraic case} when $\mathcal{X}$ is $\C \P^n$ with the standard real structure induced by the conjugation in $\C^{n+1}$, $\mathcal{E}$ is the trivial bundle $\mathcal{X}\times \C^r$ and $\mathcal{L} = \mathcal{O}(1)$ is the hyperplane bundle with its usual metric. Then $\R \mathcal{X} = \R \P^n$ and $\omega$ is the Fubini-Study metric on $\C \P^n$, normalized so that it is the quotient of the Euclidean metric on the sphere $\S^{2n+1}$. Besides, $\R H^0(\mathcal{X},\mathcal{L}^d)$ is the space of real homogeneous polynomials of degree $d$ in $n+1$ variables, and $Z_s$ is the common real zero set of $r$ independent such polynomials.

In this setting, Kostlan \cite{Kos1993} proved that, for any $d \geq 1$,
\begin{equation*}
\esp{\vol{Z_s}} =d^\frac{r}{2} \vol{\R \P^{n-r}}.
\end{equation*}
See also the paper \cite{SS1993} by Shub and Smale, where they compute the expected number of common real roots for a system of $n$ polynomials in $n$ variables. The expected Euler characteristic of a random algebraic hypersurface of degree $d$ in $\R \P^n$ was computed by Podkorytov \cite{Pod2001}. Both Kostlan's and Podkorytov's results were generalized by Bürgisser. In \cite{Buer2006}, he computed the expected volume and Euler characteristic of a submanifold $Z_s$ of $\R \P^n$ defined as the common zero set of $r$ standard Gaussian polynomials $P_1, \dots, P_r$ of degree $d_1, \dots, d_r$ respectively. In particular, when these polynomials have the same degree $d$ and $n-r$ is even, he showed that:
\begin{equation}
\label{Burgisser}
\esp{\chi\left(Z_s\right)} = d^\frac{r}{2} \sum_{p=0}^{\frac{n-r}{2}} (1-d)^p \ \frac{\Gamma\!\left(p+\frac{r}{2}\right)}{p! \Gamma\!\left(\frac{r}{2}\right)},
\end{equation}
where $\Gamma$ denotes Euler's gamma function. Theorems~\ref{theorem expected volume algebraic case} and~\ref{theorem expected Euler characteristic algebraic case} agree with these previous results.

Recently, Gayet and Welschinger computed upper and lower bounds for the asymptotics of the expected Betti numbers of random real algebraic submanifolds of a projective manifold, see \cite{GW2014b,GW2014c}. This relies on sharp estimates for the expected number of critical points of index $i\in \{0,\dots,n-r\}$ of a fixed Morse function $p: \R \mathcal{X} \to \R$ restricted to the random $Z_s$. More precisely, let $N_i(Z_s)$ denote the number of critical points of index $i$ of $p_{/Z_s}$, let $\Sym(i,n-r-i)$ denote the open cone of symmetric matrices of size $n-r$ and signature $(i,n-r-i)$ and let $\dx \nu$ denote the standard Gaussian measure on the space of symmetric matrices. Gayet and Welschinger show \cite[thm.~3.1.2]{GW2014c} that:
\begin{equation}
\label{Gayet Welschinger}
\esp{N_i(Z_s)} \underset{d \to + \infty}{\sim} \left(\frac{d}{\pi}\right)^\frac{n}{2}\vol{\R \mathcal{X}}\frac{(n-1)!}{(n-r)!}\frac{e_\R(i,n-r-i)}{2^{r-1} \Gamma\left(\frac{r}{2}\right)},
\end{equation}
where $e_\R(i,n-r-i) = \displaystyle\int_{\Sym(i,n-r-i)} \norm{\det(A)} \dx \nu (A)$.

One can indirectly deduce Theorem~\ref{theorem expected Euler characteristic algebraic case} from this result and from \cite{Buer2006} in the following way. By Morse theory:
\[\esp{\chi(Z_s)} = \sum_{i=0}^{n-r} (-1)^i \esp{N_i(Z_s)} \underset{d \to + \infty}{\sim} C_{n,r} \ d^\frac{n}{2} \vol{\R\mathcal{X}}\]
where $C_{n,r}$ is a universal constant depending only on $n$ and $r$. Specifying to the case of $\R \P^n$, equation~\eqref{Burgisser} gives the value of $C_{n,r}$. Gayet and Welschinger also proved a result similar to \eqref{Gayet Welschinger} for hypersurfaces in the case of Riemannian random waves, see \cite{GW2014e}. It gives the order of growth of $\esp{\chi(Z_f)}$ in Theorem~\ref{theorem expected Euler characteristic harmonic case}, for $r=1$.

In their book \cite{AT2007}, Taylor and Adler compute the expected Euler characteristic of the excursion sets of a centered, unit-variance Gaussian field $f$ on a smooth manifold $M$. This expectation is given in terms of the Lipschitz--Killing curvatures of $M$ for the Riemannian metric $g_f$ induced by $f$, see \cite[thm.~12.4.1]{AT2007}. One can deduce from this result the expected Euler characteristic of the random hypersurface $f^{-1}(0)$, always in terms of the Lipschitz-Killing curvatures of $(M,g_f)$. It might be possible to deduce Theorems~\ref{theorem expected Euler characteristic harmonic case} and~\ref{theorem expected Euler characteristic algebraic case} from this result, in the case of hypersurfaces, when the Gaussian field $(f(x))_{x \in M}$ (resp.~$(s(x))_{x \in \R \mathcal{X}}$) has unit variance, but one would need to estimate the Lipschitz-Killing curvatures of $(M,g_f)$ (resp.~$(M,g_s)$) as $\lambda$ (resp.~$d$) goes to infinity.

In a related setting, Bleher, Shiffman and Zelditch \cite{BSZ2000a} computed the scaling limit of the $k$-points correlation function for a random complex submanifold of a complex projective manifold. See also \cite{BSZ2001} in a symplectic framework. Our proofs of Theorems~\ref{theorem expected Euler characteristic harmonic case} and~\ref{theorem expected Euler characteristic algebraic case} use the same formalism as these papers, adapted to our frameworks.

We now sketch the proofs of our main results in the Riemannian setting. The real algebraic case is similar. The first step is to express $\vol{Z_f}$ (resp.~$\chi(Z_f)$) as the integral of some function on $Z_f$. In the case of the volume this is trivial, and the answer is given by the Chern--Gauss--Bonnet theorem (see section~\ref{subsection Chern-Gauss-Bonnet theorem} below) in the case of the Euler characteristic. Then we use the Kac--Rice formula (see Theorem~\ref{theorem Kac-Rice}) which allows us to express $\esp{\vol{Z_f}}$ (resp.~$\esp{\chi(Z_f)}$) as the integral on $M$ of some explicit function that only depends on the geometry of $M$ and on the covariance function of the smooth Gaussian field defined by $f$.

It turns out that the covariance function of the field associated to the $r$ independent standard Gaussian functions $f^{(1)},\dots,f^{(r)}$ in $V_\lambda$ is given by the spectral function of the Laplacian. In the algebraic case, the covariance function is given by the Bergman kernel of $\mathcal{E}\otimes \mathcal{L}^d$. This was already used in \cite{BSZ2000a,Nic2014}.

Then, our results follow from estimates on the spectral function of the Laplacian (resp. the Bergman kernel) and their derivatives (see section~\ref{section estimates covariance kernels}). In the case of random waves, the estimates we need for the spectral function were proved by Bin \cite{Bin2004}, generalizing results of Hörmander \cite{Hoer1968}. In the algebraic case, much is known about the Bergman kernel \cite{BBS2008,BSZ2000a,MM2007,Zel1998} but we could not find the estimates we needed in codimension higher than $1$ in the literature. These estimates are established in section~\ref{subsection Bergman kernel} using Hörmander--Tian peak sections. Peak sections were already used in this context in \cite{GW2014b,GW2014c}, see also \cite{Tia1990}. The author was told by Steve Zelditch, after this paper was written, that one can deduce estimates for the Bergman kernel in higher codimension from the paper~\cite{BBS2008} by Berman, Berndtsson and Sjöstrand.

This paper is organized as follows. In section~\ref{section random submanifolds} we describe how our random submanifolds are generated and the setting of the main theorems. Section~\ref{section estimates covariance kernels} is dedicated to the estimates we need for the spectral function of the Laplacian and the Bergman kernel. In section~\ref{section riemannian geometry} we derive an integral formula for the Euler characteristic of a submanifold. The main theorems are proved in section~\ref{section proofs}, and we deal with two special cases in section~\ref{section special cases}: the flat torus and the real projective space. For these examples, it is possible to compute expectations for fixed $\lambda$ (resp. $d$) and we recover the results of Kostlan and Bürgisser. Three appendices deal respectively with: some standard results about Gaussian vectors, a rather technical proof we postponed until the end, and a derivation of the Kac--Rice formula using Federer's coarea formula.

\paragraph*{Acknowledgements.}
I am thankful to Damien Gayet for his help and support in the course of this work. I would also like to thank Benoît Laslier for answering many of my questions about probability theory.

\tableofcontents

% Mainmatter

\section{Random submanifolds}
\label{section random submanifolds}

This section is concerned with the precise definition of the random submanifolds we consider. The first two subsections explain how we produce them in a quite general setting. The third one introduces the covariance kernel, which characterizes their distribution. We also describe the distribution induced on the bundle of $2$-jets in terms of this kernel. Then we describe what we called Riemannian random waves, before explaining how to adapt all this in the real algebraic case. This kind of random submanifolds has already been considered by Bérard \cite{Ber1985}, Zelditch \cite{Zel2009} and Nicolaescu \cite{Nic2014} in the Riemannian case, and by Gayet and Welschinger in the real algebraic case, see \cite{GW2014b,GW2014c,GW2014a}. See also \cite{Buer2006,EK1995,LL2013} in special cases.

\subsection{General setting}
\label{subsection general setting}

Let $(M,g)$ be a smooth closed manifold of positive dimension $n$. We denote by $\rmes{M}$ the Riemannian measure on $M$ induced by $g$. That is, if $x=(x_1,\dots,x_n)$ are local coordinates in an open set $U \subset M$ and $\phi$ is a smooth function with compact support in $U$,
\begin{equation}
\label{equation riemannian measure}
\int_M \phi \rmes{M} = \int_{x\in U} \phi(x) \sqrt{\det(g(x))} \dx x_1\dots\dx x_n.
\end{equation}

From now on, we fix some $r \in \{1,\dots,n\}$ that we think of as the codimension of our random submanifolds. Let $\prsc{\cdot}{\cdot}$ denote the $L^2$-scalar product on $\mathcal{C}^\infty(M,\R^r)$ induced by $\rmes{M}$: for any $f_1$ and $f_2 \in \mathcal{C}^\infty(M,\R^r)$,
\begin{equation}
\label{L2 inner product}
\prsc{f_1}{f_2}= \int_{x \in M} \prsc{f_1(x)}{f_2(x)} \rmes{M},
\end{equation}
where the inner product on the right-hand side is the standard one on $\R^r$.

\begin{ntn}
Here and throughout this paper $\prsc{\cdot}{\cdot}$ will always denote the inner product on the concerned Euclidean or Hermitian space.
\end{ntn}

Let $V$ be a subspace of $\mathcal{C}^\infty(M,\R^r)$ of finite dimension $N$. For any $f \in V$, we denote by $Z_f$ the zero set of $f$. Let $D$ denote the \emph{discriminant locus} of $V$, that is the set of $f \in V$ that do not vanish transversally.

If $f$ vanishes transversally, $Z_f$ is a (possibly empty) smooth submanifold of $M$ of codimension $r$ and we denote by $\rmes{f}$ the Riemannian measure induced by the restriction of $g$ to $Z_f$. We also denote by $\vol{Z_f}$ the volume of $Z_f$ and by $\chi(Z_f)$ its Euler characteristic. In the case $r=n$, when $f\notin D$, $Z_f$ is a finite set and $\rmes{f}$ is the sum of the Dirac measures centered on points of $Z_f$.

We consider a random vector $f \in V$ with standard Gaussian distribution. That is the distribution of $f$ admits the density function:
\begin{equation}
\label{equation density function}
x \mapsto \frac{1}{(2\pi)^\frac{N}{2}} \exp\left(-\frac{1}{2}\Norm{x}^2\right)
\end{equation}
with respect to the Lebesgue measure of $V$. Under some further technical assumptions on $V$ (see section~\ref{subsection incidence manifold} below), $f$ vanishes transversally almost surely. Hence, the random variables $\vol{Z_f}$ and $\chi(Z_f)$ are well-defined almost everywhere, and it makes sense to compute their expectation.

For the convenience of the reader, we gathered the few results we need about Gaussian vectors in Appendix \ref{section gaussian vectors}. We introduce some notations here and refer to Appendix \ref{section gaussian vectors} for further details. In the sequel, we will denote by $X \sim \mathcal{N}(m,\Lambda)$ the fact that the random vector $X$ is distributed according to a Gaussian with mean $m$ and variance $\Lambda$. A standard Gaussian vector is $X \sim \mathcal{N}(0,\Id)$. We will denote by $\dx\nu_N$ the standard Gaussian measure on a Euclidean space of dimension $N$, that is the measure with density \eqref{equation density function} with respect to the Lebesgue measure.

\subsection{The incidence manifold}
\label{subsection incidence manifold}

Following \cite{Nic2014}, we say that $V$ is \emph{$0$-ample} if the map $j^0_x: f \mapsto f(x)$ is onto for every $x \in M$. From now on, we assume that this is the case and we introduce an \emph{incidence manifold} as Shub and Smale in \cite{SS1993} (see also \cite{GW2014b,GW2014c}).

Let $F : (f,x) \in V\times M \mapsto f(x) \in \R^r$ and let $\partial_1F$ and $\partial_2F$ denote the partial differentials of $F$ with respect to the first and second variable respectively. For any $(f,x) \in V \times M$, 
\begin{align}
\label{equation partial differentials of F}
\partial_1F(f,x) &= j^0_x & &\text{and} & \partial_2F(f,x) &= d_xf.
\end{align}
We assumed $j^0_x$ to be surjective for every $x \in M$, thus $F$ is a submersion. Then $\Sigma = F^{-1}(0)$ is a smooth submanifold of codimension $r$ of $V \times M$, called the incidence manifold, and for any $(f_0,x) \in V\times M$:
\[T_{(f_0,x)}\Sigma = \{(f,v) \in V \times T_x M \mid f(x) + d_xf_0\cdot v=0 \}.\]

We set $\pi_1:\Sigma \to V$ and $\pi_2:\Sigma\to M$ the projections from $\Sigma$ to each factor. A vector $f \in V$ is in the range of $d_{(f_0,x)}\pi_1$ if and only if there exists some $v \in T_xM$ such that $(f,v) \in T_{(f_0,x)}\Sigma$, that is $f(x)$ is in the range of $d_xf_0$. Since $V$ is $0$-ample, the map $j^0_x$ is onto, and $d_xf_0$ is surjective if and only if $d_{(f_0,x)}\pi_1$ is. Thus, the discriminant locus $D$ is exactly the set of critical values of $\pi_1$. By Sard's theorem, $D$ has measure $0$ in $V$, both for the Lebesgue measure and for $\dx\nu_{N}$, and $f$ vanishes transversally almost surely.

We equip $\Sigma$ with the restriction of the product metric on $V \times M$. Then, whenever $f \notin D$, $(\pi_1)^{-1}(f) = \{f\} \times Z_f$ is isometric to $Z_f$, hence we will identify these sets. Similarly, we will identify $(\pi_2)^{-1}(x)= \ker(j^0_x) \times \{x\}$ with the subspace $\ker(j^0_x)$ of $V$.

\subsection{The covariance kernel}
\label{subsection covariance kernel}

In this subsection we introduce the Schwarz kernel and covariance function associated to our space of random functions. It turns out (see Proposition~\ref{proposition covariance function} below) that these objects are equal. The first to use this fact were Bleher, Shiffman and Zelditch in the case of complex projective manifolds \cite{BSZ2000a} and in the case of symplectic manifolds \cite{BSZ2001}. In the setting of Riemannian random waves this was used by Zelditch \cite{Zel2009} and Nicolaescu \cite{Nic2014}.

In $\mathcal{C}^\infty(M,\R^r)$ equipped with the $L^2$-inner product \eqref{L2 inner product}, the orthogonal projection onto $V$ can be represented by its \emph{Schwartz kernel}, denoted by $E$. That is there exists a unique $E:M\times M \to \R^r \otimes \R^r$ such that for any smooth $f: M \to \R^r$, the projection of $f$ onto $V$ is given by:
\begin{equation}
\label{equation definition E}
x \mapsto \prsc{E(x,\cdot)}{f} = \int_{y\in M} \prsc{E(x,y)}{f(y)} \rmes{M}.
\end{equation}
In the previous formula, the inner product on the right-hand side is the usual one on $\R^r$, acting on the second factor of $\R^r \otimes \R^r$. The kernel $E$ has the following reproducing kernel property:
\begin{equation}
\label{equation reproducing kernel property E}
\forall f \in V,\ \forall x \in M, \quad f(x) = \prsc{E(x,\cdot)}{f}.
\end{equation}

If $(f_1,\dots,f_{N})$ is any orthonormal basis of $V$, one can check that $E$ is defined by:
\begin{equation}
\label{equation E in BON}
E : (x,y) \mapsto \sum_{i=1}^{N} f_i(x) \otimes f_i(y).
\end{equation}
This proves that $E$ is smooth. Besides, for all $x\in M$, $E(x,x)$ is in the span of $\{\zeta \otimes \zeta \mid \zeta \in \R^r\}$, and for all $\zeta \in \R^r$:
\begin{equation}
\label{equation positive kernel}
\prsc{E(x,x)}{\zeta \otimes \zeta} = \sum_{i=1}^N \prsc{f_i(x)\otimes f_i(x)}{\zeta \otimes \zeta} = \sum_{i=1}^N (\prsc{f_i(x)}{\zeta})^2 \geq 0.
\end{equation}
This last equation shows that we can check the $0$-amplitude condition for $V$ on its kernel.

\begin{lem}
\label{lemma checking 0 amplitude}
$V$ is $0$-ample if and only if, for all $x \in M$ and all $\zeta \in \R^r \setminus \{0\}$, we have: $\prsc{E(x,x)}{\zeta \otimes \zeta} > 0$. That is if and only if $E(x,x)$ is a positive-definite bilinear form on $(\R^r)^*$ for any $x \in M$.
\end{lem}

On the other hand, the standard Gaussian vector $f \in V$ defines a smooth centered Gaussian field $(f(x))_{x \in M}$ with values in $\R^r$. Its distribution is totally determined by its covariance function: $(x,y) \mapsto \cov(f(x),f(y))$ from $M \times M$ to $\R^r \otimes \R^r$, where $\cov(f(x),f(y))$ stands for the covariance form of the random vectors $f(x)$ and $f(y)$ (cf. Appendix~\ref{section gaussian vectors}).

\begin{prop}
\label{proposition covariance function}
Let $V$ be a finite-dimensional subspace of $\mathcal{C}^\infty(M,\R^r)$ and $E$ its Schwartz kernel. Let $f \sim \mathcal{N}(0,\Id)$ in $V$, then we have:
\begin{equation*}
\forall x, y \in M, \qquad \cov(f(x),f(y))= \esp{f(x)\otimes f(y)} = E(x,y).
\end{equation*}
\end{prop}

\begin{proof}
Let $x$ and $y\in M$, then the first equality is given by Lemma~\ref{lemma covariance as a tensor}. We will now show that $\esp{f(x) \otimes f(y)}$ satisfies condition~\eqref{equation definition E} to prove the second equality. Let $f_0 : M \to \R^r$ be a smooth function and $x \in M$,
\begin{align*}
\int_{y\in M} \prsc{\esp{f(x) \otimes f(y)}}{f_0(y)}\rmes{M} &= \int_{y \in M} \esp{\prsc{f(x)\otimes f(y)}{f_0(y)}} \rmes{M}\\
&= \int_{y \in M} \esp{f(x) \prsc{f(y)}{f_0(y)}}\rmes{M} = \esp{f(x) \prsc{f}{f_0}}.
\end{align*}
If $f_0 \in V^\perp$, this equals $0$. If $f_0 \in V$, we have:
\begin{equation*}
\esp{f(x)\prsc{f}{f_0}}= \esp{\prsc{E(x,\cdot)}{f}\prsc{f_0}{f}} = \prsc{E(x,\cdot)}{f_0}=f_0(x),
\end{equation*}
where we used the reproducing kernel property~\eqref{equation reproducing kernel property E} both for $f$ and $f_0$ and we applied Lemma~\ref{lemma random scalar product} to $f \sim \mathcal{N}(0,\Id)$. In both cases, $x \mapsto \esp{f(x) \prsc{f}{f_0}}$ is the projection of $f_0$ onto $V$, which shows the second equality in Proposition~\ref{proposition covariance function}.
\end{proof}

This tells us that the distribution of our Gaussian field is totally determined by the Schwartz kernel $E$. In our cases of interest, asymptotics are known for $E$ and its derivatives, see section~\ref{section estimates covariance kernels} below. This is what allows us to derive asymptotics for the expectation of the volume and Euler characteristic of $Z_f$.

\subsection{Random jets}
\label{subsection random jets}

Let $\nabla^M$ be the Levi-Civita connection on $(M,g)$. For any smooth $f:M\to \R^r$, we denote by $\nabla^2f =\nabla^M df$ the \emph{Hessian} of $f$.

Let $\partial_x$ (resp.~$\partial_y$) denote the partial derivative with respect to the first (resp.~second) variable for maps from $M \times M$ to $\R^r \otimes \R^r$. Likewise, we denote by $\partial_{x,x}$ (resp.~$\partial_{y,y}$) the second  partial derivative with respect to the first (resp.~second) variable twice. As for the Hessian above, all the higher order derivatives are induced by $\nabla^M$.

Now, let $f \sim \mathcal{N}(0,\Id)$ in $V$. We will describe the distribution induced by $f$ on the $2$-jets bundle of $M$. Let $x \in M$, then we denote by $\mathcal{J}^k_x(\R^r)$ the space of $k$-jets of smooth functions from $M$ to $\R^r$ at the point $x$ (we will only use $k\in\{0,1,2\}$). We already defined
\begin{align*}
j^0_x : \mathcal{C}^\infty(M,\R^r)& \longrightarrow \R^r .\\
f& \mapsto f(x)
\end{align*}
We define similarly,
\begin{align*}
j^1_x : \mathcal{C}^\infty(M,\R^r)& \to \R^r \otimes (\R \oplus T^*_xM)\\
f& \mapsto (f(x),d_xf)\\
\text{and} \qquad j^2_x : \mathcal{C}^\infty(M,\R^r)& \to \R^r \otimes (\R \oplus T^*_xM \oplus \Sym(T^*_xM)),\\
f& \mapsto (f(x),d_xf,\nabla^2_xf)
\end{align*}
where $\Sym(T^*_xM)$ denotes the space of symmetric bilinear forms on $T^*_xM$.

The map $j^2_x$ induces an isomorphism between $\mathcal{J}^2_x(\R^r)$ and the image of $j^2_x$. In the sequel we will identify these spaces through $j_x^2$. Likewise, $\mathcal{J}_x^1(\R^r)$ and the image of $j_x^1$ will be identified through $j_x^1$.

\begin{lem}
\label{lemma variance 2-jet}
Let $V$ be a finite-dimensional subspace of $\mathcal{C}^\infty(M,\R^r)$ and $E$ its Schwartz kernel. Let $f \sim \mathcal{N}(0,\Id)$ in $V$ and $x \in M$. Then $j^2_x(f)=(f(x),d_xf,\nabla^2_xf)$ is a centered Gaussian vector, and its variance form $\var(j^2_x(f))$ is characterized by:
\begin{align}
\label{equation var t}
\var\left(f(x)\right) &= \esp{f(x) \otimes f(x)}\; = E(x,x),\\
\label{equation var L}
\rule{0pt}{4mm}\var\left(d_xf\right) &= \esp{\nabla_x f \otimes \nabla_x f}\; = (\partial_x \partial_y E)(x,x),\\
\label{equation var S}
\rule{0pt}{4mm}\var\left(\nabla^2_xf\right) &= \esp{\nabla^2_x f \otimes \nabla^2_x f} = (\partial_{x,x} \partial_{y,y} E)(x,x),\\
\label{equation cov tL}
\rule{0pt}{4mm}\cov\left(f(x),d_xf\right) &= \esp{f(x) \otimes \nabla_x f}\; = (\partial_y E)(x,x),\\
\label{equation cov tS}
\rule{0pt}{4mm}\cov\left(f(x),\nabla^2_xf\right) &= \esp{f(x) \otimes \nabla^2_x f} = (\partial_{y,y} E)(x,x),\\
\label{equation cov LS}
\rule{0pt}{4mm}\cov\left(d_xf,\nabla^2_xf\right) &= \esp{\nabla_x f \otimes \nabla^2_x f} = (\partial_x \partial_{y,y}E)(x,x).
\end{align}
\end{lem}

\begin{proof}
The first equality on each line is given by Lemmas~\ref{lemma variance as a tensor} and~\ref{lemma covariance as a tensor}. Then Proposition~\ref{proposition covariance function} gives the second equality in~\eqref{equation var t}. The other equalities are obtained by taking partial derivatives of~\eqref{equation var t}
\end{proof}

With this lemma, we have described the distribution of $j_x^2(f)$ only in terms of $E$. Since $j_x^0(f)$ and $j_x^1(f)$ are the projections of $j_x^2(f)$ onto $\R^r$ and $\R^r \otimes (\R \oplus T_x^*M)$ respectively, their distributions are also characterized by Lemma~\ref{lemma variance 2-jet}.

\subsection{Riemannian random waves}
\label{subsection harmonic setting}

In this section, we describe what we called Riemannian random waves, that is random linear combinations of eigenfunctions of the Laplacian.

Let $\Delta$ denote the Laplace-Beltrami operator on the closed Riemannian manifold $(M,g)$. Recall the following classical facts from the theory of elliptical operators, see \cite[thm.~4.43]{GHL1990}.
\begin{thm}
\label{theorem eigenvalues eigenspaces}
\begin{enumerate}
\item The eigenvalues of $\Delta :\mathcal{C}^\infty(M) \to \mathcal{C}^\infty(M)$ can be arranged into a strictly increasing sequence of non-negative numbers $(\lambda_k)_{k \in \N}$ such that $\lambda_k \xrightarrow[k \to +\infty]{} + \infty$.
\item The associated eigenspaces are finite-dimensional, and they are pairwise orthogonal for the $L^2$-inner product \eqref{equation prsc function} on $\mathcal{C}^\infty(M)$ induced by $g$.
\end{enumerate}
\end{thm}

Let $\lambda \geq 0$, then we denote by $V_\lambda$ the subspace of $\mathcal{C}^\infty(M)$ spanned by the eigenfunctions of $\Delta$ associated to eigenvalues that are less or equal to $\lambda$. Each $f=(f^{(1)},\dots,f^{(r)}) \in (V_\lambda)^r$ defines naturally a map from $M$ to $\R^r$ so that we can see $(V_\lambda)^r$ as a subspace of $\mathcal{C}^\infty(M,\R^r)$. By Theorem~\ref{theorem eigenvalues eigenspaces}, $V_\lambda$ is finite-dimensional so we can apply the construction of sections~\ref{subsection general setting} to~\ref{subsection random jets} to $(V_\lambda)^r$.

Since we consider a product situation, it is possible to make several simplifications. First, the scalar product \eqref{L2 inner product} on $(V_\lambda)^r$ is induced by the one on $V_\lambda$. Thus $f=(f^{(1)},\dots,f^{(r)})$ is a standard Gaussian in $(V_\lambda)^r$ if and only if $f^{(1)},\dots,f^{(r)}$ are $r$ independent standard Gaussian in $V_\lambda$. For a $f \in (V_\lambda)^r$ satisfying this condition, $f^{(1)},\dots,f^{(r)}$ are independent, and so are their derivatives of any order. This means that, for every $x\in M$, the matrices of $\var(f(x))$, $\var(d_xf)$ and $\var(\nabla^2_xf)$ in the canonical basis of $\R^r$ are block diagonal.

Another way to say this is that the kernel of $(V_\lambda)^r$ is a product in the following sense. We denote by  $e_\lambda:M\times M \to \R$ the Schwartz kernel of the orthogonal projection onto $V_\lambda$ in $\mathcal{C}^\infty(M)$ and by $E_\lambda$ the Schwartz kernel of $(V_\lambda)^r$. The kernel $e_\lambda$ is the \emph{spectral function of the Laplacian} and precise asymptotics are known for $e_\lambda$ and its derivatives, see section~\ref{section estimates covariance kernels}.

Let $(\varphi_1,\dots,\varphi_{N})$ be an orthonormal basis of $V_\lambda$. By \eqref{equation E in BON},
\begin{equation}
\label{equation e lambda}
e_\lambda :  (x,y) \mapsto \sum_{i=1}^{N} \varphi_i(x)\varphi_i(y).
\end{equation}
Let $(\zeta_1,\dots,\zeta_r)$ denote the canonical basis of $\R^r$. The maps $\varphi_i \zeta_q : M \to \R^r$ with $1\leq i \leq N$ and $1\leq q \leq r$ give an orthonormal basis of $(V_\lambda)^r$ and, for all $x$ and $y\in M$,
\begin{equation*}
E_\lambda(x,y) = \sum_{q=1}^r \sum_{i=1}^{N} (\varphi_i(x)\zeta_q) \otimes (\varphi_i(y) \zeta_q)=  e_\lambda(x,y) \sum_{q=1}^r \zeta_q\otimes\zeta_q.
\end{equation*}
\begin{lem}
\label{lemma E in terms of e}
Let $(\zeta_1,\dots,\zeta_r)$ be any orthonormal basis of $\R^r$, for all $x$, $y\in M$,
\begin{equation*}
E_\lambda(x,y) = e_\lambda(x,y) \left( \sum_{q=1}^r \zeta_q \otimes \zeta_q\right).
\end{equation*}
\end{lem}

An immediate consequence of this and Lemma~\ref{lemma checking 0 amplitude} is that $(V_\lambda)^r$ is $0$-ample if and only if $e_\lambda(x,x) >0$ for all $x\in M$. By \eqref{equation e lambda}, this is equivalent to $V_\lambda$ being base-point-free, that is for every $x \in M$, there exists $f \in V_\lambda$ such that $f(x) \neq 0$.

\begin{lem}
\label{lemma 0 amplitude}
For all $\lambda \geq 0$, $(V_\lambda)^r$ is $0$-ample.
\end{lem}
\begin{proof}
The constant functions on~$M$ are eigenfunctions of $\Delta$ associated to the eigenvalue~$0$. Thus for every $\lambda \geq 0$, $V_\lambda$ contains all constant functions on $M$, hence is base-point-free. By the above remark, $(V_\lambda)^r$ is then $0$-ample.
\end{proof}

\subsection{The real algebraic setting}
\label{subsection real algebraic setting}

Let us now describe more precisely the real algebraic framework. The main difference with what we did previously is that we consider sections of a rank $r$ vector bundle instead of maps to $\R^r$. The local picture is the same as in sections~\ref{subsection general setting} to~\ref{subsection harmonic setting}, so that we can adapt everything to this setting. But the formalism is a bit heavier. A classical reference for most of this material is \cite{GH1994}.

Let $\mathcal{X}$ be a smooth complex projective manifold of complex dimension $n$. We equip $\mathcal{X}$ with a real structure, that is with an antiholomorphic involution $c_\mathcal{X}$. We assume that its real locus, the set of fixed points of $c_\mathcal{X}$, is not empty and we denote it by $\R \mathcal{X}$. Let $\mathcal{L}$ be an ample holomorphic line bundle over $\mathcal{X}$ equipped with a real structure $c_\mathcal{L}$ compatible with the one on $\mathcal{X}$. By this we mean that $c_\mathcal{X} \circ \pi = \pi \circ c_\mathcal{L}$, where $\pi : \mathcal{L} \to \mathcal{X}$ stands for the projection map onto the base space. Similarly, let $\mathcal{E}$ be a holomorphic vector bundle of rank $r$ over $\mathcal{X}$, with a compatible real structure $c_\mathcal{E}$.

Let $h_\mathcal{L}$ and $h_\mathcal{E}$ be real Hermitian metrics on $\mathcal{L}$ and $\mathcal{E}$ respectively, that is $c_\mathcal{L}^*(h_\mathcal{L})=\overline{h_\mathcal{L}}$ and $c_\mathcal{E}^*(h_\mathcal{E})=\overline{h_\mathcal{E}}$. We assume that $h_\mathcal{L}$ is positive in the sense that its \emph{curvature form} $\omega$ is Kähler. Locally we have:
\begin{equation*}
\omega_{/\Omega} = \frac{1}{2i}\partial \bar{\partial}\ln\left(h_\mathcal{L}(\zeta,\zeta)\right)
\end{equation*}
where $\zeta$ is any non-vanishing local holomorphic section of $\mathcal{L}$ on the open set $\Omega \subset \mathcal{X}$. This form corresponds to a Hermitian metric $g_\C=\omega(\cdot,i\cdot)$ on $\mathcal{X}$ whose real part is a Riemannian metric $g$. We denote by $\dx V_{\mathcal{X}}$ the volume form $\frac{\omega^n}{n!}$ on $\mathcal{X}$.

\begin{rem}
The normalization of $\omega$ is the one of \cite{BBS2008,BSZ2000a}, but differs from our references concerning peak sections \cite{GW2014c,Tia1990}. This will cause some discrepancies with the latter two in section~\ref{subsection peak sections}. With our convention, the Fubini-Study metric on $\R \P^n$ induced by the standard metric on the hyperplane line bundle $\mathcal{O}(1)$ is the quotient of the Euclidean metric on $\S^n$.
\end{rem}

Let $d \in \N$, then the vector bundle $\mathcal{E}\otimes \mathcal{L}^d$ comes with a real structure $c_d = c_\mathcal{E} \otimes c_\mathcal{L}^{d}$ compatible with $c_\mathcal{X}$ and a real Hermitian metric $h_d = h_\mathcal{E} \otimes h_\mathcal{L}^{d}$. We equip the space $\Gamma(\mathcal{E}\otimes \mathcal{L}^d)$ of smooth sections of $\mathcal{E}\otimes \mathcal{L}^d$ with the $L^2$ Hermitian product defined by:
\begin{equation}
\label{equation definition prsc algebraic case}
\forall s_1, s_2 \in \Gamma(\mathcal{E}\otimes \mathcal{L}^d), \qquad \prsc{s_1}{s_2} = \int_{\mathcal{X}} h_d(s_1,s_2) \dx V_\mathcal{X}.
\end{equation}

We know from the vanishing theorem of Kodaira and Serre that the space $H^0(\mathcal{X},\mathcal{E}\otimes \mathcal{L}^d)$ of global holomorphic sections of $\mathcal{E}\otimes \mathcal{L}^d$ has finite dimension $N_d$ and that $N_d$ grows as a polynomial of degree $n$ in $d$, when $d$ goes to infinity.  We denote by:
\begin{equation*}
\R H^0(\mathcal{X},\mathcal{E}\otimes \mathcal{L}^d) = \left\{s \in H^0\left(\mathcal{X},\mathcal{E}\otimes \mathcal{L}^d\right)\mvert c_d \circ s = s \circ c_\mathcal{X} \right\}
\end{equation*}
the space of real holomorphic sections of $\mathcal{E}\otimes \mathcal{L}^d$, which has real dimension $N_d$. The Hermitian product~\eqref{equation definition prsc algebraic case} induces a Euclidean inner product on $\R H^0(\mathcal{X},\mathcal{E}\otimes \mathcal{L}^d)$. Notice that we integrate on the whole of $\mathcal{X}$, not only on the real locus, even when we consider real sections.

If $s \in \R H^0(\mathcal{X},\mathcal{E}\otimes \mathcal{L}^d)$ is such that its restriction to $\R \mathcal{X}$ vanishes transversally, then its real zero set $Z_s=s^{-1}(0)\cap \R \mathcal{X}$ is a (possibly empty) submanifold of $\R \mathcal{X}$ of codimension~$r$. We denote by $\rmes{s}$ the Riemannian measure induced on this submanifold by the metric~$g$.

As in section~\ref{subsection incidence manifold}, we consider the incidence manifold:
\begin{equation}
\label{definition Sigma d}
\Sigma_d = \left\{(s,x)\in \R H^0(\mathcal{X},\mathcal{E}\otimes \mathcal{L}^d) \times \R \mathcal{X} \mvert s(x) = 0 \right\}.
\end{equation}
In this setting, $\Sigma_d$ is the zero set of the bundle map $F_d : \R H^0(\mathcal{X},\mathcal{E}\otimes \mathcal{L}^d) \times \R \mathcal{X} \to \R(\mathcal{E}\otimes \mathcal{L}^d)$ over $\R \mathcal{X}$ defined by $F_d:(s,x)\mapsto s(x)$. In a trivialization, the situation is similar to the one in section~\ref{subsection incidence manifold}. Thus, if $\R H^0(\mathcal{X},\mathcal{E}\otimes \mathcal{L}^d)$ is $0$-ample, $\Sigma_d$ is a smooth manifold equipped with two projection maps, $\pi_1$ and $\pi_2$, onto $\R H^0(\mathcal{X},\mathcal{E}\otimes \mathcal{L}^d)$ and $\R \mathcal{X}$ respectively. By Sard's theorem, the discriminant locus of $\R H^0(\mathcal{X},\mathcal{E}\otimes \mathcal{L}^d)$ then has measure $0$ for any non-singular Gaussian measure, since it is the set of critical values of $\pi_1$.

\begin{rem}
Here, by $\R H^0(\mathcal{X},\mathcal{E}\otimes \mathcal{L}^d)$ is $0$-ample we mean that, for every $x \in \R \mathcal{X}$, the evaluation map $j^{0,d}_x : s\in \R H^0(\mathcal{X},\mathcal{E}\otimes \mathcal{L}^d) \mapsto s(x) \in \R (\mathcal{E}\otimes \mathcal{L}^d)_x$ is onto.
\end{rem}

Let $\nabla^d$ denote any real connection on $\mathcal{E}\otimes \mathcal{L}^d$, that is such that for every smooth section~$s$, $\nabla^d \left(c_d \circ s \circ c_\mathcal{X}\right) = c_d \circ \left(\nabla^d s\right) \circ dc_\mathcal{X}$. For example one could choose the Chern connection. We consider the vertical component $\nabla^d F_d$ of the differential of $F_d$, whose kernel is the tangent space of $\Sigma_d$. For any $(s_0,x)\in \Sigma_d$ the partial derivatives of $F_d$ are given by:
\begin{align}
\label{equation differential incidence manifold alg}
\partial_1^dF_d(s_0,x)&=j^{0,d}_x & &\text{and} & \partial_2^dF_d(s_0,x)&=\nabla_x^ds_0.
\end{align}
Note that we only consider points of the zero section of $\mathcal{E}\otimes \mathcal{L}^d$, hence all this does not depend on the choice of $\nabla^d$.

Let $P_1$ (resp. $P_2$) denote the projection from $\mathcal{X} \times \mathcal{X}$ onto the first (resp. second) factor. Recall that $(\overline{\mathcal{E}\otimes \mathcal{L}^d}) \boxtimes (\mathcal{E}\otimes \mathcal{L}^d)$ stands for the bundle $P_1^*(\overline{\mathcal{E}\otimes \mathcal{L}^d}) \otimes P_2^* (\mathcal{E}\otimes \mathcal{L}^d)$ over $\mathcal{X}\times \mathcal{X}$. Let $E_d$ denote the Schwartz kernel of the orthogonal projection from the space of real smooth sections of $\mathcal{E}\otimes \mathcal{L}^d$ onto $\R H^0(\mathcal{X},\mathcal{E}\otimes \mathcal{L}^d)$. It is the unique section of $(\overline{\mathcal{E}\otimes \mathcal{L}^d}) \boxtimes (\mathcal{E}\otimes \mathcal{L}^d)$ such that, for every real smooth section $s$ of $\mathcal{E}\otimes \mathcal{L}^d$, the projection of $s$ on $\R H^0(\mathcal{X},\mathcal{E}\otimes \mathcal{L}^d)$ is given by:
\begin{equation*}
x \mapsto \prsc{E_d(x,\cdot)}{s} = \int_{y\in \mathcal{X}} h_d(E_d(x,y),s(y)) \dx V_\mathcal{X}.
\end{equation*}
Here, $h_d$ acts on the second factor of $(\overline{\mathcal{E}\otimes \mathcal{L}^d})_x \otimes (\mathcal{E}\otimes \mathcal{L}^d)_y$.

The kernel $E_d$ satisfies a reproducing kernel property similar to~\eqref{equation reproducing kernel property E}:
\begin{equation}
\label{equation reproducing kernel alg}
\forall s \in \R H^0(\mathcal{X},\mathcal{E}\otimes \mathcal{L}^d),\ \forall x \in \R \mathcal{X}, \qquad s(x) = \prsc{E_d(x,\cdot)}{s} \in \R (\mathcal{E}\otimes \mathcal{L}^d)_x.
\end{equation}

If $(s_1,\dots,s_{N_d})$ is an orthonormal basis of $\R H^0(\mathcal{X},\mathcal{E}\otimes \mathcal{L}^d)$ then for all $x$ and $y \in \mathcal{X}$,
\begin{equation}
\label{equation kernel in BON alg}
E_d(x,y) = \sum_{i=1}^{N_d} s_i(x)\otimes s_i(y).
\end{equation}
For any $x\in \R\mathcal{X}$, this shows  $E_d(x,x)$ is in the span $\{\zeta \otimes \zeta \mid \zeta \in \R(\mathcal{E}\otimes \mathcal{L}^d)_x\}$. We also get the analogue of Lemma~\ref{lemma 0 amplitude}.
\begin{lem}
\label{lemma 0 amplitude alg}
$\R H^0(\mathcal{X},\mathcal{E}\otimes \mathcal{L}^d)$ is $0$-ample if and only if for any $x \in \R \mathcal{X}$, $E_d(x,x)$ is a positive-definite bilinear form on $(\R(\mathcal{E}\otimes \mathcal{L}^d)_x)^*$.
\end{lem}

Let $s$ be a standard Gaussian vector in $\R H^0(\mathcal{X},\mathcal{E}\otimes \mathcal{L}^d)$, then $(s(x))_{x \in \R \mathcal{X}}$ defines a Gaussian field with values in $\mathcal{E}\otimes\mathcal{L}^d$ and its covariance function is a section of $(\overline{\mathcal{E}\otimes \mathcal{L}^d}) \boxtimes (\mathcal{E}\otimes \mathcal{L}^d)$. The same proof as for Proposition~\ref{proposition covariance function} gives the following.

\begin{prop}
\label{proposition covariance function alg}
Let $E_d$ be the Schwartz kernel of $\R H^0(\mathcal{X},\mathcal{E}\otimes \mathcal{L}^d)$. Let $s \sim \mathcal{N}(0,\Id)$ in $\R H^0(\mathcal{X},\mathcal{E}\otimes \mathcal{L}^d)$, then we have:
\begin{equation*}
\forall x, y \in \mathcal{X}, \qquad \cov(s(x),s(y))= \esp{s(x)\otimes s(y)} = E_d(x,y).
\end{equation*}
\end{prop}

\begin{rem}
The kernel $E_d$ is also the kernel of the orthogonal projection onto $H^0(\mathcal{X},\mathcal{E}\otimes \mathcal{L}^d)$ in $\Gamma\left(\mathcal{E}\otimes \mathcal{L}^d\right)$ for the Hermitian inner product~\eqref{equation definition prsc algebraic case}, that is the \emph{Bergman kernel} of $\mathcal{E}\otimes \mathcal{L}^d$.
\end{rem}

As in section~\ref{subsection random jets}, let $\nabla^{\R\mathcal{X}}$ denote the Levi-Civita connection  on $(\R \mathcal{X},g)$. This connection and $\nabla^d$ induce the connection $\nabla^{\R\mathcal{X}}\otimes \Id + \Id \otimes \nabla^d$ on $T^*(\R\mathcal{X}) \otimes (\mathcal{E}\otimes \mathcal{L}^d)$.  We denote by $\nabla^{2,d}$ the second covariant derivative $(\nabla^{\R\mathcal{X}}\otimes \Id + \Id \otimes \nabla^d) \circ \nabla^d$.

Let $x \in \R \mathcal{X}$ and let $\mathcal{J}^k_x(\mathcal{E}\otimes \mathcal{L}^d)$ denote the space of real $k$-jets of real smooth sections of $\mathcal{E}\otimes \mathcal{L}^d$ at $x$. On the space of smooth real sections of $\mathcal{E}\otimes \mathcal{L}^d$ we define
\begin{align*}
j_x^{1,d}&: s \mapsto (s(x),\nabla^d_xs) & &\text{and} & j_x^{2,d}&: s \mapsto (s(x),\nabla_x^ds,\nabla^{2,d}_x s),
\end{align*}
where $\nabla^d_xs$ and $\nabla^{2,d}_xs$ are implicitly restricted to $T_x\R \mathcal{X}$. These maps induce isomorphisms from $\mathcal{J}_x^1(\mathcal{E}\otimes \mathcal{L}^d)$ to the image of $j_x^{1,d}$ in $\R(\mathcal{E}\otimes \mathcal{L}^d)_x \otimes ( \R \oplus T_x^*\R\mathcal{X})$ and from $\mathcal{J}_x^2(\mathcal{E}\otimes \mathcal{L}^d)$ to the image of $j_x^{2,d}$ in $\R(\mathcal{E}\otimes \mathcal{L}^d)_x \otimes ( \R \oplus T_x^*\R\mathcal{X} \oplus \Sym(T_x^*\R\mathcal{X}))$ respectively. Note that the above isomorphisms are not canonical since they depend on the choice of $\nabla^d$.

We have the following equivalent of Lemma~\ref{lemma variance 2-jet}, with the same proof, and similar notations for the partial covariant derivatives.

\begin{lem}
\label{lemma variance 2-jet alg}
Let $E_d$ denote the Bergman kernel of $\R H^0(\mathcal{X},\mathcal{E}\otimes \mathcal{L}^d)$ and let $s$ be a standard Gaussian vector in $\R H^0(\mathcal{X},\mathcal{E}\otimes \mathcal{L}^d)$. Let $x \in \R \mathcal{X}$, then $j_x^{2,d}(s)$ is a centered Gaussian vector, and its variance form is characterized by:
\begin{align}
\label{equation var t alg}
\var\left(s(x)\right) &= \esp{s(x) \otimes s(x)} \quad \ = E_d(x,x),\\
\label{equation var L alg}
\rule{0pt}{4mm}\var\left(\nabla_x^d s\right) &= \esp{\nabla_x^d s \otimes \nabla_x^d s} \quad \, = (\partial_x \partial_y E_d)(x,x),\\
\label{equation var S alg}
\rule{0pt}{4mm}\var\left(\nabla^{2,d}_x s\right) &= \esp{\nabla^{2,d}_x s \otimes \nabla^{2,d}_x s} = (\partial_{x,x} \partial_{y,y} E_d)(x,x),\\
\label{equation cov tL alg}
\rule{0pt}{4mm}\cov\left(s(x),\nabla_x^d s\right) &= \esp{s(x) \otimes \nabla^d_x s} \quad \, = (\partial_y E_d)(x,x),\\
\label{equation cov tS alg}
\rule{0pt}{4mm}\cov\left(s(x),\nabla^{2,d}_x s\right) &= \esp{s(x) \otimes \nabla^{2,d}_x s} \; \ = (\partial_{y,y} E_d)(x,x),\\
\label{equation cov LS alg}
\rule{0pt}{4mm}\cov\left(\nabla_x^d s,\nabla^{2,d}_x s\right) &= \esp{\nabla^d_x s \otimes \nabla^{2,d}_x s}\; \; \, = (\partial_x \partial_{y,y}E_d)(x,x).
\end{align}
\end{lem}

\section{Estimates for the covariance kernels}
\label{section estimates covariance kernels}

We state in this section the estimates for the kernels described above and their first and second derivatives. These estimates will allow us to compute the limit distribution for the random $2$-jets induced by the Gaussian field $(f(x))_{x \in M}$ (resp. $(s(x))_{x \in \R\mathcal{X}}$).

In the case of the spectral function of the Laplacian $e_\lambda$, the asymptotics of section~\ref{subsection spectral function of the Laplacian} were established by Bin \cite{Bin2004}, extending results of Hörmander \cite{Hoer1968}. In the algebraic case, Bleher, Shiffman and Zelditch used estimates for the related Szegö kernel, see \cite[thm.~3.1]{BSZ2000a}. In terms of the Bergman kernel, a similar result was established in \cite{BBS2008}. Both these results concern line bundles. Here, we establish the estimates we need for the Bergman kernel in the case of a higher rank bundle using Hörmander--Tian peak sections (see sections~\ref{subsection peak sections} and~\ref{subsection Bergman kernel} below).

\subsection{The spectral function of the Laplacian}
\label{subsection spectral function of the Laplacian}

We consider the Riemannian setting of section~\ref{subsection harmonic setting}. Let $x \in M$ and let $(x_1,\dots,x_n)$ be normal coordinates centered at $x$. Let $(y_1,\dots,y_n)$ denote the same coordinates in a second copy of $M$, so that $(x_1,\dots,x_n,y_1,\dots,y_n)$ are normal coordinates around $(x,x) \in M \times M$. We denote by $\partial_{x_i}$ (resp.~$\partial_{y_i}$) the partial derivative with respect to $x_i$ (resp.~$y_i$), and similarly $\partial_{x_i,x_j}$ (resp.~$\partial_{y_i,y_j}$) denotes the second derivative with respect to $x_i$ and $x_j$ (resp.~$y_i$ and $y_j$). Let
\begin{align*}
\gamma_0 &= \frac{1}{(4\pi)^\frac{n}{2} \Gamma\left(1+\frac{n}{2}\right)},\\
\gamma_1 &= \frac{1}{2(4\pi)^\frac{n}{2} \Gamma\left(2+\frac{n}{2}\right)}\\
\intertext{and}
\gamma_2 &= \frac{1}{4(4\pi)^\frac{n}{2} \Gamma\left(3+\frac{n}{2}\right)},
\end{align*}
where $\Gamma$ is Euler's gamma function. Let us recall the main theorem of \cite{Bin2004}.

\begin{thm}[Bin]
\label{theorem estimates bin}
Let $V_\lambda$ be as in section~\ref{subsection harmonic setting} and let $e_\lambda$ denote its Schwartz kernel. The following asymptotics hold, uniformly in $x\in M$, as $\lambda \to +\infty$:
\begin{align}
\label{bin x0y0}
e_\lambda(x,x) &= \gamma_0 \lambda^\frac{n}{2} + O\!\left(\lambda^\frac{n-1}{2}\right),\\
\label{bin x1y0}
\partial_{x_i} e_\lambda(x,x) &= O\!\left(\lambda^\frac{n}{2}\right),\\
\label{bin x2y0}
\partial_{x_i,x_k} e_\lambda(x,x) &= \left\{\begin{aligned} -\gamma_1 \lambda^{\frac{n}{2}+1} +\, &O\!\left(\lambda^\frac{n+1}{2}\right) & &\text{if } i=k,\\ &O\!\left(\lambda^\frac{n+1}{2}\right) & &\text{if } i\neq k,\end{aligned}\right.\\
\label{bin x1y1}
\partial_{x_i} \partial_{y_j} e_\lambda(x,x) &= \left\{\begin{aligned} \gamma_1 \lambda^{\frac{n}{2}+1} +\, &O\!\left(\lambda^\frac{n+1}{2}\right) & &\text{if } i=j,\\ &O\!\left(\lambda^\frac{n+1}{2}\right) & &\text{if } i\neq j, \end{aligned}\right.\\
\label{bin x2y1}
\partial_{x_i,x_k}\partial_{y_j} e_\lambda(x,x) &= O\!\left(\lambda^{\frac{n}{2}+1}\right),\\
\label{bin x2y2}
\partial_{x_i,x_k} \partial_{y_j,y_l} e_\lambda(x,x) &= \left\{\begin{aligned} 3\gamma_2 \lambda^{\frac{n}{2}+2} +\, &O\!\left(\lambda^\frac{n+3}{2}\right) & &\text{if } i=j=k=l,\\ \gamma_2 \lambda^{\frac{n}{2}+2} +\, &O\!\left(\lambda^\frac{n+3}{2}\right) & &\text{if } i=j\neq k=l \text{ or } i=k\neq j=l,\\ &O\!\left(\lambda^\frac{n+3}{2}\right) & &\text{otherwise.} \end{aligned}\right.
\end{align}
\end{thm}

Since $e_\lambda$ is symmetric, this also gives the asymptotics for $\partial_{y_j} e_\lambda$, $\partial_{y_j,y_l} e_\lambda$ and $\partial_{x_i} \partial_{y_j,y_l} e_\lambda$ along the diagonal. This theorem, together with Lemma~\ref{lemma E in terms of e}, gives the estimates we need for the kernel $E_\lambda$ of~$(V_\lambda)^r$.

We will need the following relations:
\begin{align}
\label{equation ratio gamma}
\frac{\gamma_0}{\gamma_1} &= n+2 & &\text{and} & \frac{\gamma_1}{\gamma_2} &= n+4.
\end{align}

\subsection{Hörmander--Tian peak sections}
\label{subsection peak sections}

We now recall the construction of Hörmander--Tian peak sections in the framework of section~\ref{subsection real algebraic setting}. Let $\mathcal{X}$ be a complex projective manifold. Let $\mathcal{E}$ be a rank~$r$ holomorphic Hermitian vector bundle and $\mathcal{L}$ be an ample holomorphic Hermitian line bundle, both defined over $\mathcal{X}$. We assume that $\mathcal{X}$, $\mathcal{E}$ and $\mathcal{L}$ are endowed with compatible real structures, and that the Kähler metric $g_\C$ on $\mathcal{X}$ is induced by the curvature $\omega$ of $\mathcal{L}$.

The goal of this subsection is to build, for every $d$ large enough, a family of real sections of $\mathcal{E} \otimes \mathcal{L}^d$ with prescribed $2$-jets at some fixed point $x\in \R\mathcal{X}$. Moreover, we want this family to be orthonormal, up to an error that goes to $0$ as $d$ goes to infinity. Using these sections, we will compute (in section~\ref{subsection Bergman kernel} below) the asymptotics we need for the Bergman kernel.

Let $x \in \R \mathcal{X}$ and $(x_1,\dots,x_n)$ be real holomorphic coordinates centered at $x$ and such that $(\deron{\phantom{f}}{x_1},\dots,\deron{\phantom{f}}{x_n})$ is orthonormal at $x$. The next lemma is established in \cite[lemma~3.3]{GW2014a}, up to a factor $\pi$ coming from different normalizations of the metric.

\begin{lem}
\label{lemma nice frame for L}
There exists a real holomorphic frame $\zeta_0$ for $\mathcal{L}$, defined over some neighborhood of $x$, whose potential $-\ln(h_\mathcal{L}(\zeta_0,\zeta_0))$ vanishes at $x$, where it reaches a local minimum with Hessian $g_\C$.
\end{lem}

We choose such a frame $\zeta_0$. Let $(\zeta_1,\dots,\zeta_r)$ be a real holomorphic frame for $\mathcal{E}$ over a neighborhood of $x$, which is orthonormal at $x$. Since $\mathcal{X}$ is compact, we can find $\rho >0$, not depending on~$x$, such that local coordinates and frames as above are defined at least on the geodesic ball of radius $\rho$ centered at $x$. The following results are proved in \cite[section~2.3]{GW2014c}. See also \cite[section~2.2]{GW2014b} and the paper by Tian \cite[lemmas~1.2 and 2.3]{Tia1990}, without the higher rank bundle $\mathcal{E}$ but with more details.

\begin{prop}
\label{proposition peak sections}
Let $p=(p_1,\dots,p_n) \in \N^n$, $m >p_1+\cdots +p_n$ and $q \in \{1,\dots,r\}$.\\
There exists $d_0 \in \N$ such that, for any $d \geq d_0$, there exist $C_{d,p}>0$ and $s \in \R H^0(\mathcal{X},\mathcal{E}\otimes \mathcal{L}^d)$ such that $\Norm{s}=1$ and
\begin{equation*}
s(x_1,\dots,x_n)= C_{d,p} \left(x_1^{p_1}\cdots x_n^{p_n}+ O\!\left(\Norm{(x_1,\dots,x_n)}^{2m}\right)\right) \left(1+ O\!\left(d^{-2m}\right)\right) \zeta_q \otimes \zeta_0^d
\end{equation*}
in some neighborhood of $x$, where the estimate $O(d^{-2m})$ is uniform in $x \in \R \mathcal{X}$.

Moreover, $d_0$ depends on $m$ but does not depend on $x$, $p$, $q$ or our choices of local coordinates and frames. Finally, $C_{d,p}$ is given by:
\begin{equation*}
\left(C_{d,p}\right)^{-2} = \int_{\left\{\Norm{(x_1,\dots,x_n)}\leq \frac{\ln(d)}{\sqrt{d}}\right\}} \norm{x_1^{p_1}\cdots x_n^{p_n}}^2 h_\mathcal{L}^d\left(\zeta_0^d,\zeta_0^d\right) \dx V_\mathcal{X}.
\end{equation*}
\end{prop}

The following definitions use Proposition~\ref{proposition peak sections} with $m=3$ and the corresponding $d_0$.
\begin{dfns}
\label{definition peak sections}
For any $d \geq d_0$ and $q\in \{1,\dots,r\}$ we denote the sections of $\R H^0(\mathcal{X},\mathcal{E}\otimes \mathcal{L}^d)$ given by Proposition~\ref{proposition peak sections} by:
\begin{itemize}
\item $s_{0}^{d,q}$ for $p_1=\cdots = p_n = 0$,
\item $s_{i}^{d,q}$ for $p_i=1$ and $\forall\, k\neq i$, $p_k= 0$,
\item $s_{i,i}^{d,q}$ for $p_i=2$ and $\forall\, k\neq i$, $p_k= 0$,
\item $s_{i,j}^{d,q}$ for $p_i=p_j=1$ and $\forall\, k\notin \{i,j\}$, $p_k= 0$, when $i<j$.
\end{itemize}
\end{dfns}
\noindent
Computing the values of the corresponding $C_{d,p}$ (see \cite[lemma~2.5]{GW2014b}), we get the following asymptotics as $d$ goes to infinity. Once again, $O(d^{-1})$ is uniform in $x$.

\begin{lem}
\label{lemma asymptotics peak sections}
For every $q \in \{1,\dots,r\}$, we have:
\begin{align}
\label{peak section 0}
s_{0}^{d,q} &= \sqrt{\frac{d^n}{\pi^n}} \left(1+ O\!\left(\Norm{(x_1,\dots,x_n)}^6\right)\right)\left(1+ O\!\left(d^{-1}\right)\right) \zeta_q \otimes \zeta_0^d,\\
\label{peak section 1}
\forall i \in \{1,\dots,n\}, \quad s_{i}^{d,q} &= \sqrt{\frac{d^{n+1}}{\pi^n}} \left(x_i+ O\!\left(\Norm{(x_1,\dots,x_n)}^6\right)\right)\left(1+ O\!\left(d^{-1}\right)\right) \zeta_q \otimes \zeta_0^d,\\
\label{peak section 2}
\forall i \in \{1,\dots,n\}, \quad s_{i,i}^{d,q} &= \sqrt{\frac{d^{n+2}}{\pi^n}} \left(\frac{x_i^2}{\sqrt{2}}+ O\!\left(\Norm{(x_1,\dots,x_n)}^6\right)\right)\left(1+ O\!\left(d^{-1}\right)\right) \zeta_q \otimes \zeta_0^d,\\
\intertext{and finally, \  $\forall i, j \in \{1,\dots,n\}$ such that $i<j$,}
\label{peak section 11}
s_{i,j}^{d,q} &= \sqrt{\frac{d^{n+2}}{\pi^n}} \left(x_ix_j+ O\!\left(\Norm{(x_1,\dots,x_n)}^6\right)\right)\left(1+ O\!\left(d^{-1}\right)\right) \zeta_q \otimes \zeta_0^d.
\end{align}
\end{lem}

\begin{rem}
These values differ from the one given in \cite[lemma~2.3.5]{GW2014c} by a factor $\left(\int_\mathcal{X} \dx V_\mathcal{X}\right)^{\frac{1}{2}}$ and some power of $\sqrt{\pi}$ because we do not use the same normalization for the volume form. For the same reason they also differ from \cite[lemma~2.3]{Tia1990} by a factor $\pi^\frac{n}{2}$.
\end{rem}

The sections defined in Definition~\ref{definition peak sections} are linearly independent, at least for $d$ large enough. In fact, they are asymptotically orthonormal in the following sense. Let $H_{2,x}^d \subset \R H^0(\mathcal{X},\mathcal{E}\otimes\mathcal{L}^d)$ denote the subspace of sections that vanish up to order $2$ at $x$.

\begin{lem}
\label{lemma asymptotically orthonormal}
The sections $(s_{i}^{d,q})_{\substack{1\leq q \leq r\\ 0 \leq i \leq n}}$ and $(s_{i,j}^{d,q})_{\substack{1\leq q \leq r\\ 1\leq i\leq j \leq n}}$ defined in Definition~\ref{definition peak sections} have $L^2$-norm equal to $1$ and their pairwise scalar product is dominated by a $O(d^{-1})$ independent of $x$. Moreover, their scalar product with any unit element of $H_{2,x}^d$ is dominated by some $O(d^{-1})$ not depending on $x$.
\end{lem}

\subsection{The Bergman kernel}
\label{subsection Bergman kernel}

In this section we compute some asymptotics for the Bergman kernel and its derivatives. Let $x \in \R \mathcal{X}$ and let $(x_1,\dots,x_n)$ be real holomorphic coordinates around $x$ such that $(\deron{\phantom{f}}{x_1},\dots,\deron{\phantom{f}}{x_n})$ is orthonormal at $x$. We denote by $(y_1,\dots,y_n)$ the same coordinates as $(x_1,\dots,x_n)$ in a second copy of $\mathcal{X}$. Let $\zeta_0$ be a real holomorphic frame for $\mathcal{L}$ given by Lemma~\ref{lemma nice frame for L} and $(\zeta_1,\dots,\zeta_r)$ be a real holomorphic frame for $\mathcal{E}$ that is orthonormal at $x$. For simplicity, we set $\zeta_p^d = \zeta_p(x) \otimes \zeta_0^d(x)$ for every $p \in \{1,\dots,r\}$ and $d \in \N$, so that $(\zeta_1^d,\dots,\zeta_r^d)$ is an orthonormal basis of $\R(\mathcal{E}\otimes\mathcal{L}^d)_x$.

Let $\nabla^d$ be any real connection on $\mathcal{E} \otimes \mathcal{L}^d$ such that, for every $p \in \{1,\dots,r\}$, $\nabla^d(\zeta_p \otimes \zeta_0^d)$ vanishes in a neighborhood of $x$. In this neighborhood, we have for every function $f$:
\begin{align}
\label{connection in coordinates}
\nabla^d(f \zeta_p \otimes \zeta_0^d) &= df \otimes \zeta_p \otimes \zeta_0^d & &\text{and} & \nabla^{2,d}(f \zeta_p \otimes \zeta_0^d) &= \nabla^2f \otimes \zeta_p \otimes \zeta_0^d,
\end{align}
where $\nabla^{2,d}$ stands for the associated second covariant derivative, as in section~\ref{subsection real algebraic setting}. This choice of connection may seem restrictive but the quantity we want to compute do not depend on a choice of connection so that we can choose one that suits us.

As usual, $\nabla^d$ induces a connection on $(\overline{\mathcal{E}\otimes \mathcal{L}^d}) \boxtimes (\mathcal{E}\otimes \mathcal{L}^d)$. We denote by $\partial^d_{x_i}$ and $\partial^d_{y_i}$ the partial covariant derivatives with respect to $x_i$ and $y_i$ respectively. We also denote by $\partial^d_{x_i,x_j}$ (resp.~$\partial^d_{y_i,y_j}$) the second derivative with respect to $x_i$ and $x_j$ (resp.~$y_i$ and $y_j$).

\begin{prop}
\label{proposition estimates Bergman}
The following asymptotics hold as $d \to +\infty$. They are independent of $x \in \R \mathcal{X}$ and of the choice of the holomorphic frame $(\zeta_1,\dots,\zeta_r)$.
\begin{align}
\label{estimates bergman x0y0}
\prsc{E_d(x,x)}{\zeta_p^d \otimes \zeta_{p'}^d} &= \left\{\begin{aligned} \frac{d^n}{\pi^n} +\, &O\!\left(d^{n-1}\right) & &\text{if } p=p',\\ &O\!\left(d^{n-1}\right) & &\text{otherwise,}\end{aligned}\right.\\
\label{estimates bergman x1y0}
\prsc{\partial^d_{x_i}E_d(x,x)}{\zeta_p^d \otimes \zeta_{p'}^d} &= O\!\left(d^{n-\frac{1}{2}}\right),\\
\label{estimates bergman x2y0}
\prsc{\partial^d_{x_i,x_k}E_d(x,x)}{\zeta_p^d \otimes \zeta_{p'}^d} &= O\!\left(d^{n}\right),\\
\label{estimates bergman x1y1}
\prsc{\partial^d_{x_i}\partial^d_{y_j}E_d(x,x)}{\zeta_p^d \otimes \zeta_{p'}^d} &= \left\{\begin{aligned} \frac{d^{n+1}}{\pi^n} +\, &O\!\left(d^{n}\right) & &\text{if $p=p'$ and $i=j$,}\\ &O\!\left(d^{n}\right) & &\text{otherwise,} \end{aligned}\right.\\
\label{estimates bergman x2y1}
\prsc{\partial^d_{x_i,x_k}\partial^d_{y_j}E_d(x,x)}{\zeta_p^d \otimes \zeta_{p'}^d} &= O\!\left(d^{n+\frac{1}{2}}\right),
\end{align}
\begin{equation}
\label{estimates bergman x2y2}
\prsc{\partial^d_{x_i,x_k}\partial^d_{y_j,y_l}E_d(x,x)}{\zeta_p^d \otimes \zeta_{p'}^d} = \left\{\begin{aligned} 2 \frac{d^{n+2}}{\pi^n} +\, &O\!\left(d^{n+1}\right) & &\text{if $p=p'$ and $i=j=k=l$,}\\ \frac{d^{n+2}}{\pi^n} +\, &O\!\left(d^{n+1}\right) & &\text{if $p=p'$ and $i=j\neq k=l$,}\\ & & &\text{or if $p=p'$ and $i=l\neq k=j$,}\\ &O\!\left(d^{n+1}\right) & &\text{otherwise.} \end{aligned}\right.
\end{equation}
\end{prop}

\begin{rem}
Since $\R H^0(\mathcal{X},\mathcal{E}\otimes\mathcal{L}^d)$ is not a product space, the terms with $p\neq p'$ are usually not zero. However they are zero when $\mathcal{E}$ is trivial, for example.
\end{rem}

\begin{cor}
\label{cor 0 amplitude}
For every $d$ large enough, $\R H^0(\mathcal{X},\mathcal{E}\otimes \mathcal{L}^d)$ is $0$-ample.
\end{cor}

\begin{proof}[Proof of Corollary~\ref{cor 0 amplitude}]
Let $x \in \R \mathcal{X}$. By~\eqref{estimates bergman x0y0}, the matrix of $E_d(x,x)$ in any orthonormal basis of $\R(\mathcal{E}\otimes \mathcal{L})_x$ is $\left(\frac{d}{\pi}\right)^n I_r \left(1+ O\!\left(d^{-1}\right)\right)$, where $I_r$ stands for the identity matrix of size $r$. Then $E_d(x,x)$ is positive-definite for $d$ larger than some $d_0$ independent of $x$. By Lemma~\ref{lemma 0 amplitude alg}, $\R H^0(\mathcal{X},\mathcal{E}\otimes\mathcal{L}^d)$ is $0$-ample for $d\geq d_0$.
\end{proof}

\begin{proof}[Proof of Proposition~\ref{proposition estimates Bergman}]
First we build an orthonormal basis of $\R H^0(\mathcal{X},\mathcal{E}\otimes\mathcal{L}^d)$ by applying the Gram--Schmidt process to the family of peak sections. Then we use formula~\eqref{equation kernel in BON alg} and the asymptotics of Lemma~\ref{lemma asymptotics peak sections} to prove the proposition.

We order the sections of Definition~\ref{definition peak sections} as follows:
\begin{multline}
\label{eq order}
s_0^{d,1},\dots,s_0^{d,r},s_1^{d,1},\dots,s_1^{d,r},\dots,s_n^{d,1},\dots,s_n^{d,r},s_{1,1}^{d,1},\dots,s_{1,1}^{d,r},s_{2,2}^{d,1},\dots,s_{2,2}^{d,r},\dots,s_{n,n}^{d,1},\dots,s_{n,n}^{d,r},\\
s_{1,2}^{d,1},\dots,s_{1,2}^{d,r}, s_{1,3}^{d,1},\dots,s_{1,3}^{d,r},\dots,s_{1,n}^{d,1},\dots,s_{1,n}^{d,r},s_{2,3}^{d,1},\dots,s_{2,3}^{d,r},\dots,s_{n-1,n}^{d,1},\dots,s_{n-1,n}^{d,r}.
\end{multline}
This family is linearly independent for $d$ large enough and spans a space whose direct sum with $H_{2,x}^d$ is $\R H^0(\mathcal{X},\mathcal{E}\otimes\mathcal{L}^d)$. We complete it into a basis $B$ of $\R H^0(\mathcal{X},\mathcal{E}\otimes\mathcal{L}^d)$ by adding an orthonormal basis of $H_{2,x}^d$ at the end of the previous list.

We apply the Gram--Schmidt process to $B$, starting by the last elements and going backwards. Let $\tilde{B}$ denote the resulting orthonormal basis, and $\tilde{s}_0^{d,1},\dots,\tilde{s}_n^{d,r},\tilde{s}_{1,1}^{d,1},\dots,\tilde{s}_{n-1,n}^{d,r}$ denote its first elements. This way, $\tilde{s}_{n-1,n}^{d,r}$ is a linear combination of $s_{n-1,n}^{d,r}$ and elements of $H_{2,x}^d$, and $\tilde{s}_0^{d,1}$ is a linear combination of (possibly) all elements of $B$. We denote by $(b_i)_{1 \leq i \leq \frac{r(n+1)(n+2)}{2}}$ the first elements of $B$ listed above \eqref{eq order} and by $(\tilde{b}_i)$ the corresponding elements of $\tilde{B}$.

Let $i \in \left\{1,\dots,\frac{r(n+1)(n+2)}{2}\right\}$ and assume that for any $k\leq i$ and any $j>i$ we have $\prsc{b_k}{\tilde{b}_j}=O\!\left(d^{-1}\right)$. Note that this is the case for $i=\frac{r}{2}(n+1)(n+2)$. Then,
\begin{equation}
\label{equation Gram Schmidt}
\tilde{b}_i=\frac{b_i-\displaystyle\sum_{j>i}\prsc{b_i}{\tilde{b}_j}\tilde{b}_j-\pi_i}{\Norm{b_i-\displaystyle\sum_{j>i}\prsc{b_i}{\tilde{b}_j}\tilde{b}_j-\pi_i}}
\end{equation}
where $\pi_i$ stands for the projection of $b_i$ onto $H_{2,x}^d$. By Lemma~\ref{lemma asymptotically orthonormal}, $\Norm{\pi_i}^2 = \prsc{b_i}{\pi_i} = O\!\left(d^{-1}\right)$. Then by our hypothesis $\Norm{b_i-\sum_{j>i} \prsc{b_i}{\tilde{b}_j}\tilde{b}_j-\pi_i}^2=1+O\!\left(d^{-1}\right)$. Using Lemma~\ref{lemma asymptotically orthonormal} and the above hypothesis once again, we get:
\begin{equation*}
\prsc{b_k}{\tilde{b}_i}=\left(1+O\!\left(d^{-1}\right)\right)\left(\prsc{b_k}{b_i}-\sum_{j>i}\prsc{b_k}{\tilde{b}_j}\prsc{b_i}{\tilde{b}_j}-\prsc{b_k}{\pi_i}\right)=O\!\left(d^{-1}\right),
\end{equation*}
for any $k<i$. By induction, for any $1\leq i<j \leq \frac{r(n+1)(n+2)}{2}$, $\prsc{b_i}{\tilde{b}_j}=O\!\left(d^{-1}\right)$. Then, using~\eqref{equation Gram Schmidt}, for any $i \in \left\{1,\dots,\frac{r(n+1)(n+2)}{2}\right\}$, 
\begin{equation*}
\tilde{b}_i=\left(b_i + O\!\left(d^{-1}\right)\sum_{j>i}\tilde{b}_j - \pi_i\right)\left(1+O\!\left(d^{-1}\right)\right).
\end{equation*}
Another induction gives:
\begin{equation}
\label{equation orthonormalized}
\tilde{b}_{i}=\left(b_i + O(d^{-1})\sum_{j>i}b_j+\tilde{\pi}_i\right)\left(1+O\!\left(d^{-1}\right)\right),
\end{equation}
for any $i$, where $\tilde{\pi}_i \in H_{2,x}^d$ is such that $\Norm{\tilde{\pi}_i}^{2}=O\!\left(d^{-1}\right)$. Moreover, all the estimates are independent of $x$ and $i \in \left\{1,\dots,\frac{r(n+1)(n+2)}{2}\right\}$.

Among the elements of $\tilde{B}$, only $\tilde{s}_0^{d,1},\dots,\tilde{s}_0^{d,r}$ do not vanish at $x$. Using formula~\eqref{equation kernel in BON alg}, we get $E_d(x,x) = \displaystyle\sum_{1\leq q \leq r} \tilde{s}_0^{d,q}(x) \otimes \tilde{s}_0^{d,q}(x)$. Then,
\begin{equation*}
\prsc{E_d(x,x)}{\zeta_p^d \otimes \zeta_{p'}^d} = \sum_{q=1}^r \prsc{\tilde{s}_0^{d,q}(x)}{\zeta_p^d}\prsc{\tilde{s}_0^{d,q}(x)}{\zeta_{p'}^d}.
\end{equation*}
Recall that $\tilde{b}_0=\tilde{s}_0^{d,1} ,\dots,\tilde{b}_r=\tilde{s}_0^{d,r}$. Because of~\eqref{equation orthonormalized}, for all $q \in \{1,\dots,r\}$,
\begin{equation}
\label{equation prsc peak section}
\begin{aligned}
\prsc{\tilde{s}_0^{d,q}(x)}{\zeta_p^d} &= \prsc{s_0^{d,q}(x)}{\zeta_p^d}+O\!\left(d^{-1}\right)\left(\sum_{q'=1}^r\prsc{s_0^{d,q'}(x)}{\zeta_p^d}\right)\\
&= \left\{ \begin{aligned}&\left(\frac{d}{\pi}\right)^\frac{n}{2} \left(1+O\!\left(d^{-1}\right)\right) & &\text{if $p=q$,} \\ &O\!\left(d^{\frac{n}{2}-1}\right) & &\text{otherwise,}\end{aligned}\right.
\end{aligned}
\end{equation}
where the last equality comes from equation~\eqref{peak section 0}. This establishes~\eqref{estimates bergman x0y0}.

Likewise,
\begin{align*}
\prsc{\partial^d_{x_i}E_d(x,x)}{\zeta_p^d	\otimes \zeta_{p'}^d} &= \prsc{\sum_{1\leq q \leq r} \partial^d_{x_i} \tilde{s}_0^{d,q}(x) \otimes \tilde{s}_0^{d,q}(x)}{\zeta_p^d \otimes \zeta_{p'}^d}\\
&= \sum_{1\leq q \leq r} \prsc{\partial^d_{x_i}\tilde{s}_0^{d,q}(x)}{\zeta_p^d} \prsc{\tilde{s}_0^{d,q}(x)}{\zeta_{p'}^d}.
\end{align*}
The description~\eqref{equation orthonormalized} shows that $\partial^d_{x_i}\tilde{s}_0^{d,q}(x)$ does not necessarily vanish, but it equals:
\begin{equation*}
O\!\left(d^{-1}\right)\sum_{\substack{1 \leq q' \leq r\\ 1\leq j \leq n}} \partial^d_{x_i}s_j^{d,q'}(x).
\end{equation*}
By~\eqref{peak section 1}, one gets that $\prsc{\partial^d_{x_i}\tilde{s}_0^{d,q}(x)}{\zeta_p^d}=O\!\left(d^{\frac{n-1}{2}}\right)$, for all $p$ and $q$. Besides by~\eqref{equation prsc peak section}, $\prsc{\tilde{s}_0^{d,q}(x)}{\zeta_{p}^d}=O\!\left(d^\frac{n}{2}\right)$ for all $p$ and $q$. This proves~\eqref{estimates bergman x1y0}.

The remaining estimates can be proved in the same way, using Lemma~\ref{lemma asymptotics peak sections} and the fact that the estimates for corresponding elements of $\tilde{B}$ and $B$ are the same.
\end{proof}

\section{An integral formula for the Euler characteristic of a submanifold}
\label{section riemannian geometry}

The goal of this section is to derive an integral formula for the Euler characteristic of a submanifold defined as the zero set of some $f:M\to \R^r$, in terms of $f$ and its derivatives. This section is independent of the previous ones and the results it contains are only useful for computing expected Euler characteristics (Theorems~\ref{theorem expected Euler characteristic harmonic case} and~\ref{theorem expected Euler characteristic algebraic case}).

We start by recalling the formalism of double forms, which was already used in this context by Taylor and Adler, see \cite[section~7.2]{AT2007}. The Riemann curvature tensor and the second fundamental form of a submanifold being naturally double forms, this provides a useful way to formulate the Chern--Gauss--Bonnet theorem and the Gauss equation. This is done in sections~\ref{subsection Chern-Gauss-Bonnet theorem} and~\ref{subsection Gauss equation} respectively. Finally, we express the second fundamental form of a submanifold in terms of the derivatives of a defining function and prove the desired integral formula in section~\ref{subsection an expression for the second fundamental form}.

\subsection{The algebra of double forms}
\label{subsection algebra of double forms}

We follow the exposition of \cite[pp.~157--158]{AT2007}. Let $V$ be a real vector space of dimension $n$. For $p$ and $q \in \{0,\dots,n\}$ we denote by $\bigwedge^{p,q}(V^*)$ the space $\bigwedge^p(V^*) \otimes \bigwedge^q(V^*)$ of $(p+q)$-linear forms on $V$ that are skew-symmetric in the first $p$ and in the last $q$ variables. The \emph{space of double forms} on $V$ is:
\begin{equation}
\textstyle\bigwedge^\bullet(V^*) \otimes \textstyle\bigwedge^\bullet(V^*) = \displaystyle\bigoplus_{0\leq p,q \leq n} \textstyle\bigwedge^{p,q}(V^*).
\end{equation}
Elements of $\bigwedge^{p,q}(V^*)$ are called \emph{$(p,q)$-double forms}, or \emph{double forms of type $(p,q)$}. We set:
\begin{equation}
\textstyle\bigwedge^{\bullet,\bullet}(V^*)=\displaystyle\bigoplus_{p=0}^n \textstyle\bigwedge^{p,p}(V^*).
\end{equation}
Note that $\bigwedge^{1,1}V^*$ is the space of bilinear forms on $V$.

On $\bigwedge^\bullet(V^*) \otimes \bigwedge^\bullet(V^*)$ we can define a double wedge product. It extends the usual wedge product on $\bigwedge^\bullet(V^*)\simeq \bigoplus_{p=0}^{n} \bigwedge^{p,0}(V^*)$, so we simply denote it by $\wedge$. For pure tensors $\alpha \otimes \beta$ and $\alpha ' \otimes \beta ' \in \bigwedge^\bullet(V^*) \otimes \bigwedge^\bullet(V^*)$, we set:
\begin{equation}
(\alpha \otimes \beta) \wedge (\alpha ' \otimes \beta') = (\alpha \wedge \alpha ') \otimes (\beta \wedge \beta ')
\end{equation}
and we extend $\wedge$ to all double forms by bilinearity. This makes $\bigwedge^\bullet(V^*) \otimes \bigwedge^\bullet(V^*)$ into an algebra, of which $\bigwedge^{\bullet,\bullet}(V^*)$ is a commutative subalgebra. We denote by $\gamma^{\wedge k}$ the double wedge product of a double form $\gamma \in \bigwedge^{\bullet,\bullet}(V^*)$ with itself $k$ times.

\begin{lem}
\label{lemma square of a symmetric (1,1)-form}
Let $\alpha$ be a symmetric $(1,1)$-double form on $V$, then for every $x,y,z$ and $w \in V$,
\begin{equation*}
\alpha^{\wedge 2}((x,y),(z,w)) = 2 \left(\alpha(x,z)\alpha(y,w) - \alpha(x,w)\alpha(y,z)\right).
\end{equation*}
\end{lem}

\begin{proof}
Let $(e_1,\dots,e_n)$ be a basis of $V$ and $(e_1^*,\dots,e_n^*)$ its dual basis. We have:
\begin{align*}
\alpha &= \sum_{1\leq i,k\leq n} \alpha_{ik} e_i^*\otimes e_k^* & &\text{ and then } & \alpha^{\wedge 2}&= \sum_{1 \leq i,j,k,l \leq n} \alpha_{ik}\alpha_{jl} (e_i^*\wedge e_j^*) \otimes (e_k^*\wedge e_l^*).
\end{align*}
Note that we do not restrict ourselves to indices satisfying $i<j$ and $k<l$ as is usually the case with skew-symmetric forms. By multilinearity, it is sufficient to check the result on elements of the basis. Let $i,j,k$ and $l \in \{1,\dots,n\}$, then
\begin{align*}
\alpha^{\wedge 2}((e_i,e_j),(e_k,e_l)) &= \alpha_{ik}\alpha_{jl} - \alpha_{jk}\alpha_{il} - \alpha_{il}\alpha_{jk} + \alpha_{jl}\alpha_{ik}\\
&= 2(\alpha_{ik}\alpha_{jl}- \alpha_{il}\alpha_{jk}) \hspace{3cm} \text{(since } \alpha \text{ is symmetric)} \\
&= 2(\alpha(e_i,e_k)\alpha(e_j,e_l)-\alpha(e_i,e_l)\alpha(e_j,e_k)).\qedhere
\end{align*}
\end{proof}

We can consider random vectors in spaces of double forms. The following technical result will be useful in the proofs of Theorems~\ref{theorem expected Euler characteristic harmonic case} and~\ref{theorem expected Euler characteristic algebraic case}. See \cite[lemma 12.3.1]{AT2007} for a proof.
\begin{lem}
\label{lemma exchanging expectation and wedge product}
Let $V$ be a vector space of finite dimension $n$. Let $\alpha$ be a Gaussian vector in $\bigwedge^{1,1}V^*$. If $\alpha$ is centered, then for any $p \leq \frac{n}{2}$,
\begin{equation*}
\esp{\alpha^{\wedge 2p}} = \frac{(2p)!}{2^p p!} \left(\esp{\alpha^{\wedge 2}}\right)^{\wedge p}.
\end{equation*}
\end{lem}

Assume now that $V$ is endowed with an inner product. It induces a natural inner product on $\bigwedge^\bullet(V^*)$ such that, if $(e_1,\dots,e_n)$ is an orthonormal basis of $V$,
\[\left\{e^*_{i_1}\wedge \dots \wedge e^*_{i_p} \mvert 1\leq p \leq n \text{ and } 1\leq i_1<i_2< \dots <i_p\leq n \right\}\]
is an orthonormal basis of $\bigwedge^\bullet(V^*)$. We define the \emph{trace operator} $\tr$ on $\bigwedge^{\bullet,\bullet}(V^*)$ in the following way. If $\alpha \otimes \beta \in \bigwedge^{\bullet,\bullet}(V^*)$ is a pure tensor, then:
\begin{equation}
\label{eq trace}
\tr(\alpha \otimes \beta)= \prsc{\alpha}{\beta}
\end{equation}
and we extend $\tr$ to $\bigwedge^{\bullet,\bullet}(V^*)$ by linearity.

Let $M$ be a smooth manifold of dimension $n$. Applying the previous construction pointwise to $T_xM$, we define the vector bundle $\bigwedge^\bullet(T^*M) \otimes \bigwedge^\bullet(T^*M)$ on $M$. Sections of this bundle are called \emph{differential double forms} on $M$, and we can take the double wedge product of two such sections. Finally, if $M$ is equipped with a Riemannian metric, we have a trace operator $\tr$ which is defined pointwise by \eqref{eq trace}. This operator is $\mathcal{C}^{\infty}(M)$-linear and takes sections of the subbundle $\bigwedge^{\bullet,\bullet}(T^*M)=\bigoplus_{p=0}^n \bigwedge^{p,p}(T^*M)$ to smooth functions.

\subsection{The Chern--Gauss--Bonnet theorem}
\label{subsection Chern-Gauss-Bonnet theorem}

Let $(M,g)$ be a closed smooth Riemannian manifold of dimension $n$. We denote by $\nabla^M$ the Levi-Civita connection of $M$, and by $\kappa$ its \emph{curvature operator}. That is $\kappa$ is the $2$-form on $M$ with values in the bundle $\End(TM)= TM \otimes T^*M$ defined by:
\begin{equation*}
\kappa(X,Y)Z = \nabla^M_X \nabla^M_Y Z - \nabla^M_Y \nabla^M_X Z - \nabla^M_{[X,Y]}Z
\end{equation*}
for any vector fields $X$, $Y$ and $Z$. Here $[X,Y]$ is the Lie bracket of $X$ and $Y$.

We denote by $R$ the \emph{Riemann curvature tensor} of $M$, defined by:
\begin{equation*}
R(X,Y,Z,W) = g(\kappa(X,Y)W,Z),
\end{equation*}
for any vector fields $X$, $Y$, $Z$ and $W$ on $M$. This defines a four times covariant tensor on $M$ which is skew-symmetric in the first two and in the last two variables, hence $R$ can naturally be seen as a $(2,2)$-double form. All this is standard material, except for the very last point, see for example \cite[section~3.3]{Jos2008}.

We now state the Chern--Gauss--Bonnet theorem in terms of double forms. Recall that $\rmes{M}$ denotes the Riemannian measure on $M$ (see \eqref{equation riemannian measure}).
\begin{thm}[Chern--Gauss--Bonnet]
\label{theorem Chern Gauss Bonnet}
Let $M$ be a closed Riemannian manifold of even dimension $n$. Let $R$ denote its Riemann curvature tensor and $\chi(M)$ denote its Euler characteristic. We have:
\begin{equation*}
\chi\left( M \right) = \frac{1}{(2\pi)^\frac{n}{2} \left( \frac{n}{2}\right)!} \int_M \tr \left( R^{\wedge \frac{n}{2}} \right) \rmes{M}.
\end{equation*}
\end{thm}

\noindent
If $M$ is orientable, this can be deduced from Atiyah-Singer's index theorem. The general case is treated in \cite{Pal1978}. The above formula in terms of double forms can be found in \cite[thm.~12.6.18]{AT2007}, up to a sign coming from different sign conventions in the definition of~$R$.

\begin{rem}
\label{rem odd dimension}
If $M$ is a closed manifold of odd dimension then $\chi(M) = 0$, see \cite[cor.~3.37]{Hat2002}.
\end{rem}

\subsection{The Gauss equation}
\label{subsection Gauss equation}

Let $(M,g)$ be a smooth Riemannian manifold of dimension $n$ and $\tilde{M}$ be a smooth submanifold of $M$ of codimension $r\in \{1,\dots,n-1\}$. We denote by $\nabla^M$ and $\tilde{\nabla}$ the Levi-Civita connections on $M$ and $\tilde{M}$ respectively. Likewise, we denote by $R$ and $\tilde{R}$ their Riemann curvature tensor. We wish to relate $R$ and $\tilde{R}$. This is done by the Gauss equation, see Proposition~\ref{proposition Gauss equation} below.

We denote by $\II$ the \emph{second fundamental form} of $\tilde{M} \subset M$ which is defined as the section of $T^\perp \tilde{M} \otimes T^*\tilde{M} \otimes T^*\tilde{M}$ satisfying:
\begin{equation}
\label{equation definition second fundamental form}
\II(X,Y) = -\left(\nabla^M_X Y - \tilde{\nabla}_X Y\right) = - \left(\nabla^M_X Y\right)^\perp
\end{equation}
for any vector fields $X$ and $Y$ on $\tilde{M}$. Here, $\left(\nabla^M_X Y\right)^\perp$ stands for the orthogonal projection of $\nabla^M_X Y$ on $T^\perp \tilde{M}$. It is well-known that $\II$ is symmetric in $X$ and $Y$, see \cite[lemma~3.6.2]{Jos2008}.

The second fundamental form encodes the difference between $\tilde{R}$ and $R$ in the following sense, see \cite[thm.~3.6.2]{Jos2008}.
\begin{prop}[Gauss equation]
\label{proposition Gauss equation}
Let $X$, $Y$, $Z$ and $W$ be vector fields on $\tilde{M}$, then:
\[\tilde{R}(X,Y,Z,W) = R(X,Y,Z,W) + \prsc{\II(X,Z)}{\II(Y,W)} - \prsc{\II(X,W)}{\II(Y,Z)}.\]
\end{prop}

We want to write this Gauss equation in terms of double forms. Let $x \in \tilde{M}$ and $X$, $Y$, $Z$ and $W \in T_x\tilde{M}$. Let $U \sim \mathcal{N}(0,\Id)$ in $T_x M$, by Lemma~\ref{lemma random scalar product}:
\begin{multline*}
\prsc{\II(X,Z)}{\II(Y,W)} - \prsc{\II(X,W)}{\II(Y,Z)}\\
= \esp{\prsc{\II(X,Z)}{U}\prsc{\II(Y,W)}{U} - \prsc{\II(X,W)}{U}\prsc{\II(Y,Z)}{U}}.
\end{multline*}
Then we apply Lemma~\ref{lemma square of a symmetric (1,1)-form} to the symmetric $(1,1)$-double form $\prsc{\II}{U}$ for fixed $U$. This gives:
\[\prsc{\II(X,Z)}{\II(Y,W)} - \prsc{\II(X,W)}{\II(Y,Z)}= \frac{1}{2}\, \esp{\prsc{\II}{U}^{\wedge 2}((X,Y),(Z,W))}.\]

We proved the following version of the Gauss equation.
\begin{prop}[Gauss equation]
\label{lemma random Gauss equation}
Let $(M,g)$ be a Riemannian manifold and let $\tilde{M}$ be a smooth submanifold of $M$, such that $\dim(M) > \dim(\tilde{M}) \geq 1$. Let $R$ and $\tilde{R}$ denote the Riemann curvature of $M$ and $\tilde{M}$ respectively, and let $\II$ be the second fundamental form of $\tilde{M}\subset M$. Then, in the sense of double forms:
\begin{equation*}
\forall x \in \tilde{M}, \qquad \tilde{R}(x) = R(x) + \frac{1}{2}\, \esp{\prsc{\II(x)}{U}^{\wedge 2}},
\end{equation*}
where $U\sim \mathcal{N}(0,\Id)$ with values in $T_xM$, and $R(x)$ is implicitly restricted to $T_x\tilde{M}$.
\end{prop}

\subsection{An expression for the second fundamental form}
\label{subsection an expression for the second fundamental form}

Let us now express the second fundamental form $\II$ of a submanifold $\tilde{M}$ of $M$ defined as the zero set of a smooth map $f:M\to \R^r$. For this we need some further definitions.

Let $V$ and $V'$ be two Euclidean spaces of dimension $n$ and $r$ respectively. Let $L : V \to V'$ be a linear surjection, then the adjoint operator $L^*$ is injective and its image is $\ker(L)^\perp$, so that $LL^*$ is invertible.

\begin{dfn}
\label{definition pseudo-inverse}
Let $L:V\to V'$ be a surjection, the \emph{pseudo-inverse} (or \emph{Moore--Penrose inverse}) of $L$ is defined as $L^\dagger = L^* (LL^*)^{-1}$ from $V'$ to $V$.
\end{dfn}

\noindent
The map $L^\dagger$ is the inverse of the restriction of $L$ to $\ker(L)^\perp$. It is characterized by the fact that $LL^\dagger$ is the identity map of $V'$ and $L^\dagger L$ is the orthogonal projection onto $\ker(L)^\perp$.

Let $f: M \to \R^r$ be a smooth submersion and assume that $\tilde{M}=f^{-1}(0)$. Recall that $\nabla^2 f = \nabla^M df$.
\begin{lem}
\label{lemma second fundamental form and Hessian}
Let $M$ be a Riemannian manifold and let $\tilde{M} \subset M$ be a submanifold of $M$ defined as the zero set of the smooth submersion $f:M \to \R^r$. Let $\II$ denote the second fundamental form of $\tilde{M} \subset M$. Then, 
\[ \forall x \in \tilde{M}, \quad \II(x) = (d_xf)^\dagger \circ \nabla^2_xf,\]
where $\nabla^2_xf$ is implicitly restricted to $T_x \tilde{M}$.
\end{lem}

\begin{proof}
Let $x\in \tilde{M}$, since $\II(x)$ and $(d_xf)^\dagger$ take values in $T_x\tilde{M}^\perp = \ker(d_xf)^\perp$, we only need to prove that $d_xf \circ \II(x) = \nabla^2_xf$. Let $X$ and $Y$ be two vector fields on $\tilde{M}$. The map $df\cdot Y$ vanishes uniformly on $\tilde{M}$, hence:
\begin{equation*}
d_x( df\cdot Y) \cdot X = (\nabla^M_X df)_x \cdot Y + d_xf \cdot (\nabla^M_X Y) = 0.
\end{equation*}
Then, using equation~\eqref{equation definition second fundamental form} and $\ker(d_xf) = T_x \tilde{M}$,
\begin{equation*}
(d_xf\circ \II(x))(X,Y) = -d_xf\cdot(\nabla^M_X Y)^\perp = -d_xf\cdot(\nabla^M_X Y) = (\nabla^M_X df)_x \cdot Y = \nabla_x^2f(X,Y).
\qedhere
\end{equation*}
\end{proof}

\begin{prop}
\label{proposition Euler chi of a submanifold}
Let $(M,g)$ be a closed Riemannian manifold of dimension $n$ and $R$ its Riemann curvature. Let $f:M\to \R^r$ be a smooth submersion and $Z_f = f^{-1}(0)$. We denote by $\rmes{f}$ the Riemannian measure on $Z_f$ and by $R_f$ its Riemann curvature.

If $n-r$ is even, the Euler characteristic of $Z_f$ is:
\begin{equation}
\label{eq CGB}
\chi(Z_f) = \frac{1}{(2\pi)^m m!} \int_{x\in Z_f}\tr \left( R_f(x)^{\wedge m} \right)\rmes{f},
\end{equation}
where $m=\frac{n-r}{2}$. Furthermore, for all $x \in Z_f$,
\begin{equation}
\label{eq riemann}
R_f(x) = R(x)+\frac{1}{2}\esp{\prsc{\nabla^2_xf}{(d_xf)^{\dagger *}(U)}^{\wedge 2}},
\end{equation}
where $U$ is a standard Gaussian vector in $T_xM$.
\end{prop}

\begin{proof}
First, we apply the Chern--Gauss--Bonnet theorem~\ref{theorem Chern Gauss Bonnet} to $Z_f$, which gives~\eqref{eq CGB}. Then let $x\in Z_f$ and $U\sim \mathcal{N}(0,\Id)$ in $T_x M$, by Proposition~\ref{lemma random Gauss equation},
\[R_f(x) = R(x) +\frac{1}{2} \esp{\prsc{\II_f(x)}{U}^{\wedge 2}},\]
where $\II_f$ is the second fundamental form of $Z_f \subset M$. We conclude by Lemma~\ref{lemma second fundamental form and Hessian}.
\end{proof}

Proposition~\ref{proposition Euler chi of a submanifold} is also true for zero sets of sections. Let $s$ be a section of some rank $r$ vector bundle over $M$ that vanishes transversally and $Z_s$ be its zero set. Let $\rmes{s}$ denote the Riemannian measure on $Z_s$ and $R_s$ denote its Riemann tensor. As above, we can apply Theorem~\ref{theorem Chern Gauss Bonnet}, so that:
\begin{equation}
\label{CGB alg}
\chi(Z_s) = \frac{1}{(2\pi)^m m!} \int_{x\in Z_s}\tr \left( R_s(x)^{\wedge m} \right)\rmes{s}.
\end{equation}
The result of Proposition~\ref{lemma random Gauss equation} is still valid for $Z_s$. Besides, the same proof as in Lemma~\ref{lemma second fundamental form and Hessian} shows that, for any connection $\nabla^d$, the second fundamental form $\II_s$ of $Z_s$ satisfies:
\begin{equation*}
\forall x \in Z_s, \quad \II_s(x) = (\nabla^d_xs)^\dagger \circ \nabla^{2,d}_xs.
\end{equation*}

\begin{rem}
This is not surprising since the terms of this equality do not depend on a choice of connection and the result in a trivialization is given by Lemma~\ref{lemma second fundamental form and Hessian}.
\end{rem}

Finally, for every connection $\nabla^d$ and every $x\in Z_s$, we get:
\begin{equation}
\label{Riemann alg}
R_s(x) = R(x)+\frac{1}{2}\esp{\prsc{\nabla^{2,d}_xs}{(\nabla^d_xs)^{\dagger *}(U)}^{\wedge 2}},
\end{equation}
where $U$ is a standard Gaussian vector in $T_xM$, as in \eqref{eq riemann}.

\section{Proofs of the main theorems}
\label{section proofs}

We now set to prove the main theorems. The proofs will be detailed in the Riemannian case but only sketched in the real algebraic one, since they are essentially the same.

\subsection{The Kac--Rice formula}
\label{subsection Kac Rice formula}

First, we state the celebrated Kac--Rice formula, which is one of the key ingredients in our proofs. This formula is proved in \cite[thm.~4.2]{BSZ2001}, see also \cite[chap.~6]{AW2009}. For the reader's convenience, we include a proof in Appendix~\ref{section proof of Kac-Rice}.

\begin{dfn}
\label{definition odet}
Let $L:V \to V'$ be a linear map between Euclidean vector spaces. We denote by $\odet{L}$ the \emph{orthogonal determinant} of $L$: 
\[\odet{L}=\sqrt{\det(LL^*)},\]
where $L^*:V'\to V$ is the adjoint operator of $L$.
\end{dfn}

\begin{rem}
If $L$ is not onto then $\odet{L}=0$. Else, let $A$ be the matrix of the restriction of $L$ to $\ker(L)^\perp$ in any orthonormal basis of $\ker(L)^\perp$ and $V'$, then we have $\odet{L}= \norm{\det(A)}$.
\end{rem}

As in section~\ref{section random submanifolds}, we consider a closed Riemannian manifold $M$ of dimension $n$ and a subspace $V\subset\mathcal{C}^\infty(M,\R^r)$ of dimension $N$ (recall that $1 \leq r \leq n$). We assume that $V$ is $0$-ample, in the sense of section~\ref{subsection incidence manifold}, so that
\[\Sigma = \left\{(f,x) \in V \times M \mid f(x) = 0\right\}\]
is a submanifold of codimension $r$ of $V\times M$. Let $f$ be a standard Gaussian vector in $V$. Then $Z_f$ is almost surely a smooth submanifold of codimension $r$ of $M$ (see section~\ref{subsection incidence manifold}). Recall that $E$ denotes both the Schwartz kernel of $V$ and the covariance function of $(f(x))_{x\in M}$.

\begin{thm}[Kac--Rice formula]
\label{theorem Kac-Rice}
Let $\phi:\Sigma \to \R$ be a Borel measurable function, then
\begin{equation*}
\esp{\int_{x \in Z_f}\hspace{-2mm} \phi(f,x) \rmes{f}} = \frac{1}{(2\pi)^\frac{r}{2}}\!\int_{x \in M}\! \frac{1}{\sqrt{\det{E(x,x)}}} \espcond{\phi(f,x) \odet{d_xf}}{f(x)=0} \rmes{M},
\end{equation*}
whenever one of these integrals is well-defined.
\end{thm}
\noindent
The expectation on the right-hand side is to be understood as the conditional expectation of $\phi(f,x) \odet{d_xf}$ given $f(x)=0$. By $\det(E(x,x))$, we mean the determinant of the matrix of the bilinear form $E(x,x)$ in any orthonormal basis of $\R^r$.

\subsection{Proof of Theorem \ref{theorem expected volume harmonic case}}
\label{subsection proof volume harmonic}

We start with the expectation of the volume, which is the toy-model for this kind of computations. In this case, the proof is closely related to \cite{BSZ2000a,BSZ2001}, in a slightly different setting. The first step is to apply Kac--Rice formula above with $\phi:(f,x)\mapsto 1$. We get:
\begin{equation}
\label{equation expectation volume 2}
\esp{\vol{Z_f}}= \frac{1}{(2\pi)^\frac{r}{2}}\int_{x \in M} \frac{1}{\sqrt{\det(E(x,x))}} \espcond{\odet{d_xf}}{f(x)=0}\rmes{M}.
\end{equation}

Let $x \in M$, then $j_x^1(f)=(f(x),d_xf)$ is a Gaussian vector in $\R^r \otimes( \R \oplus T_x^*M)$ whose distribution only depends on the values of $E$ and its derivatives at $(x,x)$, see Lemma~\ref{lemma variance 2-jet}. Thus $\esp{\vol{Z_f}}$ will only depend on the values of $E$ and its derivatives along the diagonal as was expected from \cite[thm.~2.2]{BSZ2000a}.

The next step is to compute pointwise asymptotics for the integrand on the right-hand side of \eqref{equation expectation volume 2}. We will use Lemma~\ref{lemma variance 2-jet}, which describes the distribution of $j_x^1(f)$, and the estimates of section~\ref{section estimates covariance kernels}. Both in the Riemannian and the algebraic settings, the pointwise asymptotic turns out to be universal: it does not depend on $x$ or even on the ambient manifold. This is because the distribution of $j_x^1(f)$ is determined by the asymptotics of section~\ref{section estimates covariance kernels} which are universal.

We now specify to the case of Riemannian random waves (see section~\ref{subsection harmonic setting}), that is $V=(V_\lambda)^r$ for some non-negative $\lambda$. Recall that $(V_\lambda)^r$ is $0$-ample (Lemma~\ref{lemma 0 amplitude}) so that equation~\eqref{equation expectation volume 2} is valid in this case. Let $x \in M$, by Lemma~\ref{lemma E in terms of e} and \eqref{bin x0y0} we have: 
\begin{equation}
\label{eq estimate det variance}
\det\left(E_\lambda(x,x)\right) = \left(e_\lambda(x,x)\right)^r = \left(\gamma_0 \lambda^\frac{n}{2}\right)^r \left(1 + O\!\left(\lambda^{-\frac{1}{2}}\right)\right).
\end{equation}

Then we want to estimate the conditional expectation in \eqref{equation expectation volume 2}. Before going further, the asymptotics of section~\ref{subsection spectral function of the Laplacian} suggest to consider the scaled variables:
\begin{equation}
\label{eq scaled variables}
(t_\lambda,L_\lambda) = \left(\frac{1}{\sqrt{\gamma_0}\lambda^\frac{n}{4}}f(x),\frac{1}{\sqrt{\gamma_1}\lambda^\frac{n+2}{4}}d_xf\right)
\end{equation}
instead of $j_x^1(f)$. This is a centered Gaussian vector whose variance is determined by \eqref{eq scaled variables}. Besides, by Definition~\ref{definition odet}, the orthogonal determinant is homogeneous of degree $r$ for linear maps taking values in $\R^r$, so that:
\begin{equation}
\label{equation oubliee}
\espcond{\odet{d_xf}}{f(x)=0} = (\gamma_1)^\frac{r}{2}\lambda^\frac{r(n+2)}{4} \espcond{\odet{L_\lambda}}{t_\lambda = 0}.
\end{equation}

\begin{lem}
\label{lemma conditional expectation volume}
For every $x\in M$, we have:
\[\espcond{\odet{L_\lambda}}{t_\lambda = 0} = (2\pi)^\frac{r}{2} \frac{\vol{\S^{n-r}}}{\vol{\S^n}}\left(1 + O\!\left(\lambda^{-\frac{1}{2}}\right)\right),\]
where the error term does not depend on the point $x\in M$.
\end{lem}

We postpone the proof of this lemma for now and conclude the proof of Theorem~\ref{theorem expected volume harmonic case}. By~\eqref{eq estimate det variance}, \eqref{equation oubliee} and Lemma~\ref{lemma conditional expectation volume}, for every $x \in M$,
\begin{equation*}
\frac{1}{(2\pi)^\frac{r}{2} \sqrt{\det(E(x,x))}} \espcond{\odet{d_xf}}{f(x)=0} = \left(\frac{\gamma_1}{\gamma_0}\lambda\right)^\frac{r}{2}\frac{\vol{\S^{n-r}}}{\vol{\S^n}}  \left(1 + O\!\left(\lambda^{-\frac{1}{2}}\right)\right),
\end{equation*}
where the error term does not depend on the point $x\in M$. By~\eqref{equation ratio gamma} this equals:
\begin{equation*}
\left(\frac{\lambda}{n+2}\right)^\frac{r}{2} \frac{\vol{\S^{n-r}}}{\vol{\S^n}} \left(1 + O\!\left(\lambda^{-\frac{1}{2}}\right)\right).
\end{equation*}
Plugging this into equation~\eqref{equation expectation volume 2} gives Theorem~\ref{theorem expected volume harmonic case}.

\begin{rem}
The same proof shows that, for any continuous function $\phi:M\to \R$, we have:
\begin{equation*}
\esp{\int_{Z_f} \phi \rmes{f}} = \left(\frac{\lambda}{n+2}\right)^\frac{r}{2} \left(\int_M \phi \rmes{M}\right) \frac{\vol{\S^{n-r}}}{\vol{\S^n}} + O(\lambda^\frac{r-1}{2}).
\end{equation*}
Hence,
\begin{equation*}
\left(\frac{n+2}{\lambda}\right)^\frac{r}{2}\esp{\rmes{f}} \xrightarrow[\lambda \to +\infty]{} \frac{\vol{\S^{n-r}}}{\vol{\S^n}}\rmes{M}
\end{equation*}
in the sense of the weak convergence of measures.
\end{rem}

We still have to prove Lemma~\ref{lemma conditional expectation volume}. For this, we need to compute the variance of $L_\lambda$ given $t_\lambda =0$. Let $(x_1,\dots,x_n)$ be normal coordinates centered at $x$, and $(\zeta_1,\dots,\zeta_r)$ denote the canonical basis of $\R^r$. We equip $\R^r \otimes (\R \oplus T_x^*M)$ with the orthonormal basis:
\begin{equation}
\label{basis volume}
(\zeta_1,\dots,\zeta_r,\zeta_1 \otimes dx_1,\dots,\zeta_1 \otimes dx_n,\zeta_2\otimes dx_1,\dots,\zeta_2\otimes dx_n,\dots,\zeta_r\otimes dx_1,\dots,\zeta_r\otimes dx_n).
\end{equation}
Let $\Lambda(\lambda)$ denote the matrix of $\var(t_\lambda,L_\lambda)$ in this basis. This matrix splits as
\begin{equation}
\label{equation def Lambda}
\Lambda(\lambda)=\begin{pmatrix} \Lambda_{00}(\lambda) & \Lambda_{01}(\lambda) \\ \Lambda_{10}(\lambda) & \Lambda_{11}(\lambda) \end{pmatrix},
\end{equation}
where $\Lambda_{00}(\lambda)$ and $\Lambda_{11}(\lambda)$ are the matrices of $\var(t_\lambda)$ and $\var(L_\lambda)$ and $\Lambda_{01}(\lambda)= \trans\Lambda_{10}(\lambda)$ is the matrix of $\cov(t_\lambda,L_\lambda)$. We can decompose further $\Lambda_{10}$ and $\Lambda_{11}$ into blocks of size $r \times r$:
\begin{align}
\label{equation decomposition Lambda}
\Lambda_{10}(\lambda)&=\begin{pmatrix}\Lambda_{10}^i(\lambda) \end{pmatrix}_{1\leq i \leq n}, & \Lambda_{11}(\lambda) & = \begin{pmatrix} \Lambda_{11}^{i,j}(\lambda)\end{pmatrix}_{1\leq i,j \leq n}. 
\end{align}

By the Definition~\eqref{eq scaled variables} of $(t_\lambda,L_\lambda)$, these blocks are obtained by scaling the corresponding blocks in the matrix of $\var(j_x^1(f))$. Lemmas~\ref{lemma variance 2-jet} and~\ref{lemma E in terms of e} tell us that each one of these blocks is a scalar matrix. Then, using the estimates of section~\ref{subsection spectral function of the Laplacian}:
\begin{align}
\label{Lambda00}
\Lambda_{00}(\lambda)&= \frac{e_\lambda(x,x)}{\gamma_0 \lambda^\frac{n}{2}} \ I_r = I_r + O\!\left(\lambda^{-\frac{1}{2}}\right),\\
\label{Lambda10}
\forall i \in \{1,\dots,n\}, \qquad  \Lambda_{10}^i(\lambda)&= \frac{\partial_{x_i}e_\lambda(x,x)}{\sqrt{\gamma_0\gamma_1} \lambda^\frac{n+1}{2}} \ I_r = O\!\left(\lambda^{-\frac{1}{2}}\right),\\
\label{Lambda11}
\text{and} \ \forall i,j \in \{1,\dots,n\}, \qquad \Lambda_{11}^{i,j}(\lambda)&= \frac{\partial_{x_i}\partial_{y_j}e_\lambda (x,x)}{\gamma_1 \lambda^{\frac{n}{2}+1}} \ I_r = \left\{ \begin{aligned} &I_r + O\!\left(\lambda^{-\frac{1}{2}}\right) & &\text{if $i=j$,}\\ &O\!\left(\lambda^{-\frac{1}{2}}\right) & &\text{otherwise,} \end{aligned} \right.
\end{align}
where $I_r$ stands for the identity matrix of size $r$. Thus $\Lambda(\lambda)=I_{r(n+1)}+O\!\left(\lambda^{-\frac{1}{2}}\right)$ and, by Corollary~\ref{corollary conditional expectation}, the distribution of $L_\lambda$ conditioned on $t_\lambda=0$ is a centered Gaussian with variance operator $\tilde{\Lambda}(\lambda)=\Id + O\!\left(\lambda^{-\frac{1}{2}}\right)$. Note that these estimates do not depend on $x$ or our choices of coordinates.

\begin{proof}[Proof of Lemma~\ref{lemma conditional expectation volume}]
Let $\tilde{L}_\lambda \sim \mathcal{N}(0,\tilde{\Lambda}(\lambda))$ in $\R^r \otimes T^*_xM$. For $\lambda$ large enough, $\tilde{\Lambda}(\lambda)$ is non-singular, and we have:
\begin{multline}
\label{equation lemma volume}
\espcond{\odet{L_\lambda}}{t_\lambda=0}= \esp{\odet{\tilde{L}_\lambda}}\\
=\frac{1}{(2\pi)^\frac{nr}{2}\sqrt{\det(\tilde{\Lambda}(\lambda))}} \int \odet{L} \exp\left(-\frac{1}{2}\prsc{\tilde{\Lambda}(\lambda)^{-1}L}{L}\right)\dx L,
\end{multline}
where $\dx L$ stands for the Lebesgue measure on $\R^r \otimes T^*_xM$. Beware that we see $L$ as a linear map in the term $\odet{L}$ but as a vector in $\tilde{\Lambda}(\lambda)^{-1}L$. The latter is not a composition.

Then $\tilde{\Lambda}(\lambda)=\Id + O\!\left(\lambda^{-\frac{1}{2}}\right)$, so that $\Norm{\tilde{\Lambda}(\lambda)^{-1}-\Id}$ is bounded by $\frac{C}{\sqrt{\lambda}}$ for some positive $C$. Hence, for all $L \in \R^r \otimes T_x^*M$,
\begin{equation*}
\norm{\prsc{\left(\tilde{\Lambda}(\lambda)^{-1}-\Id\right)L}{L}} \leq \frac{C}{\sqrt{\lambda}} \Norm{L}^2,
\end{equation*}
and by the mean value theorem,
\begin{equation*}
\norm{\exp\left(-\frac{1}{2}\prsc{\left(\tilde{\Lambda}(\lambda)^{-1}-\Id\right)L}{L}\right)-1} \leq \frac{C}{2\sqrt{\lambda}} \Norm{L}^2 \exp \left(\frac{C}{2\sqrt{\lambda}} \Norm{L}^2\right).
\end{equation*}
Then,
\begin{multline}
\label{big O1}
\norm{\int \odet{L}\left(\exp\left(-\frac{1}{2}\prsc{\tilde{\Lambda}(\lambda)^{-1}L}{L}\right)-\exp\left(-\frac{\Norm{L}^2}{2}\right)\right) \dx L}\\
\leq \frac{C}{2\sqrt{\lambda}} \int \odet{L}\Norm{L}^2 \exp\left(-\frac{\Norm{L}^2}{2}\left(1-\frac{C}{\sqrt{\lambda}}\right)\right) \dx L.
\end{multline}
The integral on the right-hand side of \eqref{big O1} converges to some finite limit as $\lambda \to +\infty$ by Lebesgue's dominated convergence theorem, so that:
\begin{equation}
\label{big O2}
\int \odet{L}\exp\left(-\frac{1}{2}\prsc{\tilde{\Lambda}(\lambda)^{-1}L}{L}\right)\dx L = \int \odet{L} e^{-\frac{1}{2}\Norm{L}^2}\dx L + O\!\left(\lambda^{-\frac{1}{2}}\right).
\end{equation}
Since $\det\left(\tilde{\Lambda}(\lambda)\right) = 1 + O\!\left(\lambda^{-\frac{1}{2}}\right)$, by~\eqref{equation lemma volume}, \eqref{big O2} we have:
\begin{equation*}
\espcond{\odet{L_\lambda}}{t_\lambda=0} = \esp{\odet{L}} + O\!\left(\lambda^{-\frac{1}{2}}\right),
\end{equation*}
where $L$ is a standard Gaussian vector in $\R^r \otimes T^*_xM$. The result of the lemma is given by Lemma~\ref{lemma expectation odet of L}.
\end{proof}

\subsection{Proof of Theorem \ref{theorem expected volume algebraic case}}
\label{subsection proof volume algebraic}

We now consider the real algebraic setting described in section~\ref{subsection real algebraic setting}. The proof goes along the same lines as above. Recall that $\mathcal{X}$ is a complex projective manifold of dimension $n$, equipped with a rank $r$ holomorphic vector bundle $\mathcal{E}$ and an ample holomorphic line bundle~$\mathcal{L}$, and that $\mathcal{X}$, $\mathcal{E}$ and $\mathcal{L}$ are endowed with compatible real structures. We are interested in the volume of the real zero set $Z_s$ of a standard Gaussian section $s$ in $\R H^0(\mathcal{X},\mathcal{E}\otimes \mathcal{L}^d)$.

By Corollary~\ref{cor 0 amplitude}, $\R H^0(\mathcal{X},\mathcal{E}\otimes \mathcal{L}^d)$ is $0$-ample for $d$ large enough, so that we can apply Kac--Rice formula (Theorem~\ref{theorem Kac-Rice}) with $\phi:(s,x)\mapsto 1$, as in the Riemannian case. Note that we have to use the incidence manifold $\Sigma_d$ defined by \eqref{definition Sigma d} here. As in~\eqref{equation expectation volume 2}, we get:
\begin{equation}
\label{eq alg1}
\esp{\vol{Z_s}}= \frac{1}{(2\pi)^\frac{r}{2}}\int_{x \in \R\mathcal{X}} \frac{1}{\sqrt{\det(E_d(x,x))}} \espcond{\odet{\nabla^d_xs}}{s(x)=0}\rmes{\R\mathcal{X}},
\end{equation}
where $\nabla^d$ is any real connection on $\mathcal{E}\otimes \mathcal{L}^d$.

Let $x \in \R\mathcal{X}$ and $(x_1,\dots,x_n)$ be real holomorphic coordinates around $x$ such that $(\deron{\phantom{f}}{x_1},\dots,\deron{\phantom{f}}{x_n})$ is orthonormal at $x$. Let $(\zeta_1^d,\dots,\zeta_r^d)$ be an orthonormal basis of $\R(\mathcal{E}\otimes\mathcal{L}^d)_x$. This yields an orthonormal basis of $\mathcal{J}_x^2(\mathcal{E}\otimes \mathcal{L}^d)$ similar to \eqref{basis volume}.

The value of $\espcond{\odet{\nabla^d_xs}}{s(x)=0}$ does not depend on the choice of $\nabla^d$, since $\nabla_x^ds$ does not depend on $\nabla^d$ when $s (x)=0$. We choose a connection that satisfies the conditions of section~\ref{subsection Bergman kernel}, in order to compute the pointwise asymptotic of this quantity.

The estimates of Proposition~\ref{proposition estimates Bergman} suggest to consider the scaled variables:
\begin{equation}
\label{eq scaled variables alg}
(t_d,L_d) = \left(\sqrt{\frac{\pi^n}{d^n}} s(x),\sqrt{\frac{\pi^n}{d^{n+1}}}\nabla^d_xs\right).
\end{equation}
Then, $(t_d,L_d)$ is a centered Gaussian vector in $\R(\mathcal{E}\otimes\mathcal{L}^d)_x$, and the matrix of $\var(t_d,L_d)$ in the basis described above is $I_{r(n+1)}+O(d^{-1})$. This is proved by the same kind of computation as in the Riemannian case, using the estimates of Proposition~\ref{proposition estimates Bergman}. The distribution of $L_d$ given $t_d=0$ is then a centered Gaussian with variance operator $\tilde{\Lambda}(d)= \Id +O(d^{-1})$.

As in the previous section (cf.~Lemma~\ref{lemma conditional expectation volume}), for every $x \in \R \mathcal{X}$,
\begin{equation}
\label{eq alg2}
\begin{aligned}
\espcond{\odet{\nabla^d_xs}}{s(x)=0} &= \left(\frac{d^{n+1}}{\pi^n}\right)^\frac{r}{2} \espcond{\odet{L_d}}{t_d=0}\\
&=\left(\frac{d^{n+1}}{\pi^n}\right)^\frac{r}{2} (2\pi)^\frac{r}{2} \frac{\vol{\S^{n-r}}}{\vol{\S^n}}\left(1 + O\!\left(d^{-1}\right)\right).
\end{aligned}
\end{equation}
Besides, the estimate~\eqref{estimates bergman x0y0} shows that:
\begin{equation}
\label{eq alg3}
\det\left(E_d(x,x)\right) = \left(\frac{d}{\pi}\right)^{rn}\left(1 + O\!\left(d^{-1}\right)\right),
\end{equation}
Finally, by~\eqref{eq alg1}, \eqref{eq alg2} and \eqref{eq alg3}, we have proved Theorem~\ref{theorem expected volume algebraic case}.

As for Riemannian random waves, the same proof shows that for any continuous function $\phi:\R\mathcal{X}\to \R$, we have:
\begin{equation*}
\esp{\int_{Z_s} \phi \rmes{s}} = d^\frac{r}{2} \left(\int_{\R \mathcal{X}} \phi \rmes{\R \mathcal{X}}\right) \frac{\vol{\S^{n-r}}}{\vol{\S^n}} + O\!\left(d^{\frac{r}{2}-1}\right).
\end{equation*}
Hence,
\begin{equation*}
\left(d^{-\frac{r}{2}}\right)\esp{\rmes{s}} \xrightarrow[\lambda \to +\infty]{} \frac{\vol{\S^{n-r}}}{\vol{\S^n}}\rmes{M}
\end{equation*}
in the sense of the weak convergence of measures.

\subsection{Proof of Theorem \ref{theorem expected Euler characteristic harmonic case}}
\label{subsection proof Euler harmonic}

In this section we compute the expected Euler characteristic of our random submanifolds. The proof is basically the same as in the volume case: apply Kac--Rice formula then compute a pointwise asymptotic for the conditional expectation that appears in Theorem~\ref{theorem Kac-Rice}. Only, this time, we apply Kac--Rice formula to a quantity $\phi(f,x)$ that really depends on the couple $(f,x) \in \Sigma$. This makes the computations a bit more complex. Luckily, $\phi(f,x)$ only depends on the  $2$-jet of $f$ at $x$, so we can still make pointwise computations.

Consider first the general setting of sections~\ref{subsection general setting} to~\ref{subsection random jets}: $f$ is a standard Gaussian vector in the finite-dimensional subspace $V\subset \mathcal{C}^\infty(M,\R^r)$. We assume that $V$ is $0$-ample, so that for almost every $f \in V$, $Z_f$ is a closed submanifold of dimension $n-r$.

If $n-r$ is odd, then $\chi(Z_f)=0$ almost surely (see Remark~\ref{rem odd dimension}). From now on, we assume that $n-r$ is even and set $m=\frac{n-r}{2}$. If $n=r$, then $Z_f$ is almost surely a finite set, and $\chi(Z_f)=\vol{Z_f}$ is just the cardinal of $Z_f$. In this case Theorems~\ref{theorem expected Euler characteristic harmonic case} and~\ref{theorem expected volume harmonic case} coincide, so we need only consider the case $r<n$ in the sequel.

We denote by $R$ the Riemann curvature of the ambient manifold $M$ and, for any $f \in V$, we denote by $R_f$ the Riemann curvature of $Z_f$. By Proposition~\ref{proposition Euler chi of a submanifold} and Kac--Rice formula (Theorem~\ref{theorem Kac-Rice}),
\begin{multline}
\label{equation expectation Euler 1}
\esp{\chi(Z_f)} = \esp{\frac{1}{(2\pi)^m m!} \int_{Z_f} \tr\left((R_f)^{\wedge m}\right)\rmes{f}}\\
=\frac{1}{m! (2\pi)^\frac{n}{2}} \int_{x \in M} \frac{1}{\sqrt{\det(E(x,x))}} \espcond{\odet{d_xf}\tr(R_f(x)^{\wedge m})}{f(x)=0}\rmes{M}.
\end{multline}
Moreover, for all $x \in M$, let $U \sim \mathcal{N}(0,\Id)$ in $T_xM$ be independent of $f$. Then, Proposition~\ref{proposition Euler chi of a submanifold} gives:\begin{equation}
\label{equation curvature of Zf}
R_f(x) = R(x) + \frac{1}{2}\esp[U]{\prsc{\nabla_x^2f}{d_xf^{\dagger *}(U)}^{\wedge 2}},
\end{equation}
where the notation $\esp[U]{\ \cdot \ }$ means that we only take the expectation with respect to the variable $U$. Here and in everything that follows, $R(x)$ and $\nabla^2_xf$ are implicitly restricted to $\ker(d_xf)=T_xZ_f$.

As in the volume case, the next step is to compute the pointwise asymptotic for the integrand in the last term of~\eqref{equation expectation Euler 1}. By~\eqref{equation curvature of Zf}, it only depends on $R(x)$ and the distribution of $j_x^2(f)$ which is characterized by Lemma~\ref{lemma variance 2-jet}. This shows that the expected Euler characteristic only depends on $R$ and the values of $E$ and its derivatives (up to order $2$ in each variable) along the diagonal in $M \times M$. It turns out that, both in the Riemannian and the algebraic cases, the pointwise asymptotic is universal and no longer depends on $R$.

Focusing on random waves, $V = (V_\lambda)^r$ and we already know from Lemma~\ref{lemma 0 amplitude} that $(V_\lambda)^r$ is $0$-ample. Let $x \in M$, recall that $\det(E_\lambda(x,x)) = \left(\gamma_0 \lambda^\frac{n}{2}\right)^r \left(1 + O\!\left(\lambda^{-\frac{1}{2}}\right)\right)$ (see~\eqref{eq estimate det variance}). Then, the main task is to estimate the conditional expectation in \eqref{equation expectation Euler 1}.

We consider the scaled variables:
\begin{equation}
\label{scaled variables 2}
(t_\lambda,L_\lambda,S_\lambda) = \left(\frac{1}{\sqrt{\gamma_0}\lambda^\frac{n}{4}}f(x), \frac{1}{\sqrt{\gamma_1}\lambda^\frac{n+2}{4}}d_xf, \frac{1}{\sqrt{\gamma_2}\lambda^\frac{n+4}{4}}\nabla^2_xf\right)
\end{equation}
in $\R^r \otimes (\R \oplus T_x^*M \oplus \Sym(T_x^*M))$. By Lemma~\ref{lemma variance 2-jet}, $(t_\lambda,L_\lambda,S_\lambda)$ is a centered Gaussian vector. We denote by $(\tilde{L}_\lambda,\tilde{S}_\lambda)$ a random variable in $\R^r \otimes (T_x^*M \oplus \Sym(T_x^*M))$ distributed as $(L_\lambda,S_\lambda)$ given $t_\lambda = 0$. Again, $(\tilde{L}_\lambda,\tilde{S}_\lambda)$ is a centered Gaussian, see Corollary~\ref{corollary conditional expectation}.

Let $U \sim \mathcal{N}(0,\Id)$ in $T_xM$ be independent of $f$ (hence of $j_x^2(f)$) and $(\tilde{L}_\lambda,\tilde{S}_\lambda)$. Then, by~\eqref{equation curvature of Zf}, \eqref{scaled variables 2} and~\eqref{equation ratio gamma}:
\begin{align*}
R_f(x) &= R(x) + \frac{1}{2}\esp[U]{\prsc{\nabla_x^2f}{d_xf^{\dagger *}(U)}^{\wedge 2}}\\
&= R(x) + \frac{1}{2}\esp[U]{\prsc{\sqrt{\gamma_2}\lambda^\frac{n+4}{4} S_\lambda}{\frac{L_\lambda^{\dagger *}(U)}{\sqrt{\gamma_1}\lambda^\frac{n+2}{4}}}^{\wedge 2}}\\
&= R(x) + \frac{\lambda}{2(n+4)}\esp[U]{\prsc{S_\lambda}{L_\lambda^{\dagger *}(U)}^{\wedge 2}}.
\end{align*}
Besides, $\odet{d_xf} = (\gamma_1)^\frac{r}{2} \lambda^\frac{r(n+2)}{4} \odet{L_\lambda}$, so that,
\begin{multline}
\label{equation cond exp}
\espcond{\odet{d_xf}\tr(R_f(x)^{\wedge m})}{f(x)=0} \\
\begin{aligned}
&=(\gamma_1)^\frac{r}{2} \lambda^\frac{r(n+2)}{4} \espcond{\odet{L_\lambda} \tr\left(\left(R(x) + \frac{\lambda}{2(n+4)}\esp[U]{\prsc{S_\lambda}{L_\lambda^{\dagger *}(U)}^{\!\wedge 2}}\right)^{\!\wedge m}\right)\hspace{-2mm}}{t_\lambda = 0}\\
&=(\gamma_1)^\frac{r}{2} \lambda^\frac{r(n+2)}{4} \esp{\odet{\tilde{L}_\lambda} \tr\left(\left(R(x) + \frac{\lambda}{2(n+4)}\esp[U]{\prsc{\tilde{S}_\lambda}{\tilde{L}_\lambda^{\dagger *}(U)}^{\!\wedge 2}}\right)^{\!\wedge m}\right)}.
\end{aligned}
\end{multline}

To conclude the proof, we will use the following lemmas.
\begin{lem}
\label{lemma dominant term}
The random vectors $\left(\tilde{L}_\lambda,\tilde{S}_\lambda\right)$ converge in distribution to a Gaussian vector $(L,S)$ as $\lambda \to +\infty$. Let $U \sim \mathcal{N}(0,\Id)$ in $T_xM$ be independent of $\left(\tilde{L}_\lambda,\tilde{S}_\lambda\right)$ and $(L,S)$, then:
\begin{multline*}
\esp{\odet{\tilde{L}_\lambda} \tr\left(\left(R(x) + \frac{\lambda}{2(n+4)}\esp[U]{\prsc{\tilde{S}_\lambda}{\tilde{L}_\lambda^{\dagger *}(U)}^{\wedge 2}}\right)^{\wedge m}\right)}\\
= \left(\frac{\lambda}{n+4}\right)^{m} \frac{m!}{(2m)!} \esp{\odet{L} \tr\left(\prsc{S}{L^{\dagger *}(U)}^{\wedge 2m}\right)}\left(1 + O\!\left(\lambda^{-\frac{1}{2}}\right)\right),
\end{multline*}
where the error term is uniform in $x \in M$.
\end{lem}

\begin{lem}
\label{lemma computing expectation}
Let $(L,S)$ be distributed as the limit of $\left(\tilde{L}_\lambda,\tilde{S}_\lambda\right)$ and $U \sim \mathcal{N}(0,\Id)$ in $T_xM$ be independent of $(L,S)$. We have:
\begin{equation*}
\esp{\odet{L} \tr\left(\prsc{S}{L^{\dagger *}(U)}^{\wedge 2m}\right)} = \left(-\frac{n+4}{n+2}\right)^m (2\pi)^\frac{n}{2} (2m)! \frac{\vol{\S^{n-r+1}}\vol{\S^{r-1}}}{\pi \vol{\S^n}\vol{\S^{n-1}}}.
\end{equation*}
\end{lem}

Assuming these lemmas and recalling~\eqref{eq estimate det variance} and~\eqref{equation cond exp}, the integrand in the last term of~\eqref{equation expectation Euler 1} equals:
\begin{equation*}
(-1)^m m! (2\pi)^\frac{n}{2} \left(\frac{\gamma_1}{\gamma_0}\lambda\right)^\frac{r}{2} \left(\frac{\lambda}{n+2}\right)^m \frac{\vol{\S^{n-r+1}}\vol{\S^{r-1}}}{\pi \vol{\S^n}\vol{\S^{n-1}}} \left(1 + O\!\left(\lambda^{-\frac{1}{2}}\right)\right),
\end{equation*}
where the error term is uniform in $x \in M$. Since $r+2m = n$ and $\frac{\gamma_0}{\gamma_1}=n+2$, we finally get:
\begin{equation*}
\esp{\chi(Z_f)} = (-1)^m \left(\frac{\lambda}{n+2}\right)^\frac{n}{2} \frac{\vol{\S^{n-r+1}}\vol{\S^{r-1}}}{\pi \vol{\S^n}\vol{\S^{n-1}}} \int_{x \in M}\left(1 + O\!\left(\lambda^{-\frac{1}{2}}\right)\right) \rmes{M},
\end{equation*}
and this is Theorem~\ref{theorem expected Euler characteristic harmonic case}.

We now have to prove Lemmas~\ref{lemma dominant term} and~\ref{lemma computing expectation}. For this we will need the following technical result which is a reformulation of \cite[prop.~3.12]{Buer2006}. The proof of Proposition~\ref{proposition same distribution} is mostly tedious computations and we postpone it until Appendix~\ref{section a useful result}.
\begin{prop}
\label{proposition same distribution}
Let $V$ and $V'$ be two Euclidean spaces of dimension $n$ and $r$ respectively, with $1 \leq r \leq n$. Let $L \in V' \otimes V^*$ and $U \in V$ be independent standard Gaussian vectors. Then, $L^\dagger$ is well-defined almost surely and $\left(\odet{L},(L^\dagger)^*U\right)$ has the same distribution as
\[\left(\Norm{X_n}\Norm{X_{n-1}}\cdots\Norm{X_{n-r+1}}, \frac{U'}{\Norm{X_{n-r+1}}}\right),\] where $U' \in V'$, $X_p \in \R^p$ for all $p \in \{n-r+1,\dots,n\}$ and $U',X_n,\dots,X_{n-r+1}$ are globally independent standard Gaussian vectors.
\end{prop}

We start by computing the variance of $(t_\lambda,L_\lambda,S_\lambda)$. As in section~\ref{subsection proof volume harmonic}, we choose normal coordinates $(x_1,\dots,x_n)$ centered at $x$ and denote by $(\zeta_1,\dots,\zeta_r)$ the canonical basis of $\R^r$. For any $i$ and $j$ such that $1\leq i < j \leq n$, we set $dx_{ij} = (dx_i \otimes dx_j + dx_j \otimes dx_i)$ and $dx_{ii} = dx_i \otimes dx_i$. We complete the basis of $\mathcal{J}_x^1(\R^r)$ given in~\eqref{basis volume} into an orthonormal basis of $\mathcal{J}_x^2(\R^r)$ by adding the following elements (in this order) at the end of the list:
\begin{multline}
\label{basis Euler characteristic}
\zeta_1\otimes dx_{11},\dots,\zeta_r\otimes dx_{11},\zeta_1\otimes dx_{22},\dots,\zeta_r\otimes dx_{22},\dots,\zeta_1\otimes dx_{nn},\dots,\zeta_r\otimes dx_{nn},\\
\zeta_1 \otimes dx_{12},\dots,\zeta_r \otimes dx_{12},\zeta_1 \otimes dx_{13},\dots,\zeta_r \otimes dx_{13},\dots,\zeta_1 \otimes dx_{1n},\dots,\zeta_r \otimes dx_{1n},\\
\dots,\zeta_1 \otimes dx_{(n-1)n},\dots,\zeta_r \otimes dx_{(n-1)n}.
\end{multline}

The matrix of $\var(t_\lambda,L_\lambda,S_\lambda)$ with respect to this basis is:
\begin{equation*}
\Lambda(\lambda)=\begin{pmatrix} \Lambda_{00}(\lambda) & \Lambda_{01}(\lambda) & \Lambda_{02}(\lambda) \\ \Lambda_{10}(\lambda) & \Lambda_{11}(\lambda) & \Lambda_{12}(\lambda) \\ \Lambda_{20}(\lambda) & \Lambda_{21}(\lambda) & \Lambda_{22}(\lambda)\end{pmatrix},
\end{equation*}
where $\Lambda_{00}(\lambda)$, $\Lambda_{10}(\lambda)$ and $\Lambda_{11}(\lambda)$ are as in~\eqref{equation def Lambda} and, similarly, $\Lambda_{22}(\lambda)$ is the matrix of $\var(S_\lambda)$, $\Lambda_{02}(\lambda)=\trans\Lambda_{20}(\lambda)$ is the matrix of $\cov(t_\lambda,S_\lambda)$ and $\Lambda_{12}(\lambda)=\trans\Lambda_{21}(\lambda)$ is the matrix of $\cov(L_\lambda,S_\lambda)$. As in~\eqref{equation decomposition Lambda}, we can decompose further each of these matrices in blocks of size $r \times r$. That is, $\Lambda_{10}(\lambda)$ and $\Lambda_{11}(\lambda)$ satisfy~\eqref{equation decomposition Lambda} and,
\begin{equation}
\label{decomposition Lambda}
\begin{aligned}
\Lambda_{20}(\lambda) &= \begin{pmatrix} \Lambda_{20}^{ik}(\lambda) \end{pmatrix}_{1\leq i\leq k \leq n},& \Lambda_{21}(\lambda) &= \begin{pmatrix} \Lambda_{21}^{ik,j}(\lambda)\end{pmatrix}_{\substack{1\leq i\leq k \leq n \\ 1\leq j \leq n}}\\
&\text{and} &  \Lambda_{22}(\lambda) &= \begin{pmatrix} \Lambda_{22}^{ik,jl}(\lambda)\end{pmatrix}_{\substack{1\leq i\leq k \leq n \\ 1\leq j \leq l \leq n}}.
\end{aligned}
\end{equation}

By definition of $(t_\lambda,L_\lambda,S_\lambda)$, these blocks are obtained by scaling the corresponding blocks in the matrix of $\var(j_x^2(f))$. By Lemmas~\ref{lemma variance 2-jet} and~\ref{lemma E in terms of e} the matrices in \eqref{decomposition Lambda} are scalar matrices. Recalling~\eqref{equation ratio gamma}, we set: $\gamma = -\sqrt{\frac{\gamma_1^2}{\gamma_0\gamma_2}} = -\sqrt{\frac{n+4}{n+2}}$. Then, by~\eqref{scaled variables 2} and Theorem~\ref{theorem estimates bin} we have:
\begin{align}
\label{Lambda20}
\Lambda_{20}^{ik}(\lambda) &= \frac{\partial_{x_i,x_k}e_\lambda(x,x)}{\sqrt{\gamma_0\gamma_2}\lambda^\frac{n+2}{2}} I_r = \left\{\begin{aligned} \gamma I_r + &O\!\left(\lambda^{-\frac{1}{2}}\right) & &\text{if } i=k,\\ &O\!\left(\lambda^{-\frac{1}{2}}\right) & &\text{if } i\neq k,\end{aligned}\right.\\
\label{Lambda21}
\Lambda_{21}^{ik,j}(\lambda) &= \frac{\partial_{x_i,x_k}\partial_{y_j}e_\lambda(x,x)}{\sqrt{\gamma_1\gamma_2}\lambda^\frac{n+3}{2}} I_r = O\!\left(\lambda^{-\frac{1}{2}}\right),\\
\label{Lambda22}
\Lambda_{22}^{ik,jl}(\lambda) &= \frac{\partial_{x_i,x_k} \partial_{y_j,y_l} e_\lambda(x,x)}{\gamma_2\lambda^\frac{n+4}{2}} I_r = \left\{\begin{aligned} 3I_r + &O\!\left(\lambda^{-\frac{1}{2}}\right) & &\text{if } i=j=k=l,\\ I_r + &O\!\left(\lambda^{-\frac{1}{2}}\right) & &\text{if } i=j\neq k=l\\ & & &\text{or } i=k \neq j=l,\\ &O\!\left(\lambda^{-\frac{1}{2}}\right) & &\text{otherwise,} \end{aligned}\right.
\end{align}
where $I_r$ denotes the identity matrix of size $r$. Similar estimates for $\Lambda_{00}(\lambda)$, $\Lambda_{10}(\lambda)$ and $\Lambda_{11}(\lambda)$ are given by~\eqref{Lambda00}, \eqref{Lambda10} and~\eqref{Lambda11} respectively. Then $\Lambda(\lambda)$ writes by blocks:
\begin{equation*}
\Lambda(\lambda) = 
\left(\begin{array}{c|c|cccc|c}
I_r &  & \gamma I_r & \gamma I_r & \cdots & \gamma I_r & \\
\hline
 & I_{nr} & & & & & \\
\hline
\gamma I_r & & 3I_r & I_r & \cdots & I_r & \\
\gamma I_r & & I_r & 3I_r & \ddots & \vdots & \\
\vdots & & \vdots & \ddots & \ddots & I_r & \\
\gamma I_r & & I_r & \cdots & I_r & 3I_r & \\
\hline
 & & & & & & I_{\frac{rn(n-1)}{2}} \\
\end{array}\right)
+ O\!\left(\lambda^{-\frac{1}{2}}\right),
\end{equation*}
where the empty blocks are zeros. The distribution of $(\tilde{L}_\lambda,\tilde{S}_\lambda)$,
that is the distribution of $(L_\lambda,S_\lambda)$ given $t_\lambda = 0$, is a centered Gaussian whose variance matrix is:
\begin{equation}
\label{equation big Lambda tilde}
\tilde{\Lambda}(\lambda) =
\left(\begin{array}{c|cccc|c}
I_{nr} & & & & & \\
\hline
& \beta_0 I_r & \beta I_r & \cdots & \beta I_r &\\
& \beta I_r & \beta_0 I_r & \ddots & \vdots &\\
& \vdots & \ddots & \ddots & \beta I_r &\\
& \beta I_r & \cdots & \beta I_r & \beta_0 I_r &\\
\hline
& & & & & I_{\frac{rn(n-1)}{2}}\\
\end{array}\right)
+ O\!\left(\lambda^{-\frac{1}{2}}\right),
\end{equation}
where $\beta = 1 - \gamma^2 = -\dfrac{2}{n+2}$ and $\beta_0 = 3 -\gamma^2 = \dfrac{2n+2}{n+2}$ (see Corollary~\ref{corollary conditional expectation}).

Let $\Lambda$ denote the leading term in~\eqref{equation big Lambda tilde}. Equation~\eqref{equation big Lambda tilde} shows that the random vectors $\left(\tilde{L}_\lambda,\tilde{S}_\lambda\right)$ converge in distribution to a random vector $(L,S)\sim \mathcal{N}(0,\Lambda)$ (see Lemma~\ref{lemma convergence in distribution}). We have:
\begin{equation*}
\det\left(\Lambda\right) = \left(\beta_0 + (n-1)\beta\right)^r \left(\beta_0-\beta\right)^{r(n-1)} = \left(\frac{4(2n+4)^{n-1}}{(n+2)^n}\right)^r,
\end{equation*}
so that $\Lambda$ is non-singular, and $\tilde{\Lambda}(\lambda)$ is non-singular for $\lambda$ large enough.

\begin{proof}[Proof of Lemma~\ref{lemma dominant term}]
We have already seen that $\left(\tilde{L}_\lambda,\tilde{S}_\lambda\right)$ converges in distribution, as $\lambda$ goes to infinity, to $(L,S)\sim \mathcal{N}(0,\Lambda)$. We still have to prove the estimate in the lemma. We have:
\begin{multline*}
\left(R(x) + \frac{\lambda}{2(n+4)}\esp[U]{\prsc{\tilde{S}_\lambda}{\tilde{L}_\lambda^{\dagger *}(U)}^{\wedge 2}}\right)^{\wedge m}\\
= \sum_{q=0}^m \binom{m}{q}\left(\frac{\lambda}{2(n+4)}\right)^q R(x)^{\wedge (m-q)}\wedge \esp[U]{\prsc{\tilde{S}_\lambda}{\tilde{L}_\lambda^{\dagger *}(U)}^{\wedge 2}}^{\wedge q}.
\end{multline*}
We can apply Lemma~\ref{lemma exchanging expectation and wedge product} to each term in this sum. This yields:
\begin{multline}
\label{equation big sum}
\esp{\odet{\tilde{L}_\lambda} \tr\left(\left(R(x) + \frac{\lambda}{2(n+4)}\esp[U]{\prsc{\tilde{S}_\lambda}{\tilde{L}_\lambda^{\dagger *}(U)}^{\wedge 2}}\right)^{\wedge m}\right)}\\
= \sum_{q=0}^m \binom{m}{q}\frac{q !}{(2q)!} \left(\frac{\lambda}{2(n+4)}\right)^q \esp{\odet{\tilde{L}_\lambda}\tr\left(R(x)^{\wedge (m-q)} \wedge \prsc{\tilde{S}_\lambda}{\tilde{L}_\lambda^{\dagger *}(U)}^{\wedge 2q}\right)}.
\end{multline}
Then it is sufficient to show that, for all $q \in \{0,\dots,m\}$,
\begin{multline}
\label{eq inter}
\esp{\odet{\tilde{L}_\lambda}\tr\left(R(x)^{\wedge (m-q)} \wedge \prsc{\tilde{S}_\lambda}{\tilde{L}_\lambda^{\dagger *}(U)}^{\wedge 2q}\right)}\\
= \esp{\odet{L}\tr\left(R(x)^{\wedge (m-q)} \wedge \prsc{S}{L^{\dagger *}(U)}^{\wedge 2q}\right)} + O\!\left(\lambda^{-\frac{1}{2}}\right),
\end{multline}
and that these terms are finite. Then~\eqref{equation big sum} and~\eqref{eq inter} yield the estimate in the lemma.

Let $q\in \{0,\dots,m\}$, then we first show that the principal part on the right-hand side of \eqref{eq inter} is finite. Let $\zeta\in \R^r$, then
\begin{equation}
\label{expectation S}
\esp[S]{\tr \left( R(x)^{\wedge (m-q)} \wedge \prsc{S}{\zeta}^{\wedge 2q} \right)}
\end{equation}
is finite since it is the expectation of some polynomial in the coefficients of $S$. Thus, \eqref{expectation S} only depends on $\zeta$, and it is an homogeneous polynomial in $\zeta$ of degree $2q$.

We assumed $U$ to be independent of $(L,S)$, and the expression~\eqref{equation big Lambda tilde} of $\Lambda$ shows that $L$ and $S$ are independent. Let $U' \sim \mathcal{N}(0,\Id)$ in $\R^r$ and $X_{n-r+1},\dots,X_n$ be standard Gaussian vectors, with $X_p \in \R^p$, such that $U',S,X_{n-r+1},\dots,X_n$ are globally independent. Applying Proposition~\ref{proposition same distribution}, we have:
\begin{multline}
\label{equation integrability}
\esp{\odet{L}\tr\left(R(x)^{\wedge (m-q)} \wedge \prsc{S}{L^{\dagger *}(U)}^{\wedge 2q}\right)}\\
\begin{aligned}
&= \esp{\frac{\Norm{X_n}\cdots\Norm{X_{n-r+2}}}{\Norm{X_{n-r+1}}^{2q-1}}\tr\left(R(x)^{\wedge (m-q)} \wedge \prsc{S}{U'}^{\wedge 2q}\right)}\\
&= \esp{\frac{1}{\Norm{X_{n-r+1}}^{2q-1}}} \esp{\tr\left(R(x)^{\wedge (m-q)} \wedge \prsc{S}{U'}^{\wedge 2q}\right)} \prod_{p=n-r+2}^n\esp{\Norm{X_p}}.
\end{aligned}
\end{multline}
Since $2q-1\leq 2m-1 \leq n-r-1$ and $X_{n-r+1}$ is a standard Gaussian in $\R^{n-r+1}$, $\esp{\frac{1}{\Norm{X_{n-r+1}}^{2q-1}}} < +\infty$ (see Lemma~\ref{lemma expectation norm Gaussian}). The other factors on the right-hand side of~\eqref{equation integrability} are expectations of polynomials is some standard Gaussian variables, so they are finite.

Then, $\tilde{\Lambda}(\lambda)=\Lambda + O\!\left(\lambda^{-\frac{1}{2}}\right)$, and the same kind of computations as in the proof of Lemma~\ref{lemma conditional expectation volume} gives~\eqref{eq inter}, which concludes the proof of Lemma~\ref{lemma dominant term}. Note that, since $M$ is compact, $R(x)$ is bounded, independently of $x$. We need this fact to ensure that the error term in~\eqref{eq inter} is independent of $x$.
\end{proof}

\begin{proof}[Proof of Lemma~\ref{lemma computing expectation}]
Let $(L,S) \sim \mathcal{N}(0,\Lambda)$ in $\R^r \otimes (T_x^*M \oplus \Sym(T_x^*M))$ and $U \sim \mathcal{N}(0,\Id)$ in $T_xM$ be independent of $(L,S)$. By~\eqref{equation integrability}:
\begin{equation}
\label{eq inter computation}
\esp{\odet{L}\tr\left(\prsc{S}{L^{\dagger *}(U)}^{\wedge 2m}\right)} = \esp{\frac{\Norm{X_{n-r+2}}\dots \Norm{X_n}}{\Norm{X_{n-r+1}}^{2m-1}}} \esp{\tr\left(\prsc{S}{U'}^{\wedge 2m}\right)},
\end{equation}
where $U'\sim \mathcal{N}(0,\Id)$ in $\R^r$, $X_p \sim \mathcal{N}(0,\Id)$ in $\R^p$ for all $p$, and $U',S,X_{n-r+1},\dots,X_n$ are globally independent.

Recall that we are only interested in the restriction of $S$ to $\ker(L)\subset T_xM$. But $L$ and $S$ are independent, as one can see on the expression of $\Lambda$ \eqref{equation big Lambda tilde}, and the distribution of $S$ is invariant under orthogonal transformations of $T_xM$. Thus, we can consider $S$ restricted to any $2m$-dimensional subspace of $T_xM$ in our computations. For simplicity, we restrict $S$ to $V$, the span of $\left(\deron{\phantom{f}}{x_1},\dots,\deron{\phantom{f}}{x_{2m}}\right)$.

We now compute the term $\displaystyle\esp{\tr\left(\prsc{S_{/V}}{U'}^{\wedge 2m}\right)}$. By Lemma~\ref{lemma exchanging expectation and wedge product},
\begin{equation*}
\esp[S]{\prsc{S_{/V}}{U'}^{\wedge 2m}} = \frac{(2m)!}{2^m m!} \esp[S]{\prsc{S_{/V}}{U'}^{\wedge 2}}^{\wedge m}.
\end{equation*}
Assuming that $S_{/V}= \displaystyle\sum_{1\leq i, k \leq 2m} S_{ik} dx_i \otimes dx_k$, with $S_{ik} \in \R^r$, we have:
\begin{equation*}
\prsc{S_{/V}}{U'} = \sum_{1\leq i, k \leq 2m} \prsc{S_{ik}}{U'} dx_i \otimes dx_k.
\end{equation*}
By Lemma~\ref{lemma random scalar product},
\begin{align*}
\esp[S]{\prsc{S_{/V}}{U'}^{\wedge 2}} &= \sum_{1\leq i, j, k, l \leq 2m} \esp[S]{\prsc{S_{ik}}{U'}\prsc{S_{jl}}{U'}} (dx_i \wedge dx_j) \otimes (dx_k \wedge dx_l)\\
&= \sum_{1\leq i, j, k, l \leq 2m} \prsc{U'}{\left( \Lambda^{ik,jl}\right)U'} (dx_i \wedge dx_j) \otimes (dx_k \wedge dx_l),
\end{align*}
where we denoted by $\Lambda^{ik,jl}$ the covariance operator of $S_{ik}$ and $S_{jl}$. Then,
\begin{multline*}
\esp[S]{\prsc{S}{U'}^{\wedge 2m}} = \frac{(2m)!}{2^m m!} \sum_{\substack{1\leq i_1,\dots, i_m\leq 2m \\ 1\leq j_1,\dots, j_m\leq 2m \\ 1\leq k_1,\dots, k_m\leq 2m \\ 1\leq l_1,\dots, l_m\leq 2m}} \left( \prod_{p=1}^m \prsc{U'}{\left(\Lambda^{i_pk_p,j_pl_p}\right)U'} \right)\times \\
\begin{aligned}
(dx_{i_1} \wedge  dx_{j_1} \wedge \dots \wedge dx_{i_m} \wedge dx_{j_m}) \otimes (dx_{k_1} \wedge dx_{l_1} \wedge \dots \wedge dx_{k_m} \wedge dx_{l_m})&\\
= \frac{(2m)!}{2^m m!} \sum_{\sigma, \sigma' \in \mathfrak{S}_{2m}}  \epsilon(\sigma)\epsilon(\sigma')\prod_{p=1}^m \prsc{U'}{\left(\Lambda^{\sigma(2p-1)\sigma'(2p-1),\sigma(2p)\sigma'(2p)}\right)U'}(dx \otimes dx)&,
\end{aligned}
\end{multline*}
where $\mathfrak{S}_{2m}$ is the set of permutations of $\{1,\dots,2m\}$, $\epsilon : \mathfrak{S}_{2m}\to \{-1,1\}$ denotes the signature morphism and $dx = dx_1 \wedge \dots \wedge dx_{2m}$. We get the last line by setting $\sigma(2p-1)=i_p$, $\sigma(2p)=j_p$, $\sigma'(2p-1)=k_p$ and $\sigma'(2p)=l_p$ and reordering the wedge products.

Since our local coordinates are such that $\left(\deron{\phantom{f}}{x_1},\dots,\deron{\phantom{f}}{x_{2m}}\right)$ is orthonormal at $x$, we have $\tr(dx \otimes dx) = \Norm{dx}^2=1$. Thus,
\begin{equation}
\label{equation trace of expectation}
\tr\left(\!\esp[S]{\prsc{S}{U'}^{\wedge 2m}}\right) =  \frac{(2m)!}{2^m m!}\! \sum_{\sigma, \sigma' \in \mathfrak{S}_{2m}} \epsilon(\sigma\sigma')\prod_{p=1}^m \prsc{U'}{\left(\Lambda^{\sigma(2p-1)\sigma'(2p-1),\sigma(2p)\sigma'(2p)}\right)U'}.
\end{equation}

By equation~\eqref{equation big Lambda tilde}, for any $\sigma$, $\sigma' \in \mathfrak{S}_{2m}$ and for any $p \in \{1,\dots,m\}$,
\begin{equation*}
\Lambda^{\sigma(2p-1)\sigma'(2p-1),\sigma(2p)\sigma'(2p)} = K(p,\sigma,\sigma') \Id,
\end{equation*}
where
\begin{equation}
\label{definition K}
K(p,\sigma,\sigma') = \left\{\begin{aligned} &\beta & &\text{if } \sigma(2p-1) = \sigma'(2p-1) \text{ and } \sigma(2p) = \sigma'(2p),\\
&1 & &\text{if } \sigma(2p-1) = \sigma'(2p) \text{ and } \sigma(2p) = \sigma'(2p-1),\\
&0 & &\text{otherwise.} \end{aligned} \right.
\end{equation}
Note that $K(p,\sigma,\sigma') = K(p,\id,\sigma^{-1}\circ \sigma')$, where $\id$ stands for the identity permutation. Then, setting $\tau = \sigma^{-1} \circ \sigma'$,
\begin{multline}
\label{sum on sigma}
\sum_{\sigma, \sigma' \in \mathfrak{S}_{2m}} \epsilon(\sigma\sigma') \prod_{p=1}^m \prsc{U'}{\left(\tilde{\Lambda}^{\sigma(2p-1)\sigma'(2p-1),\sigma(2p)\sigma'(2p)}\right)U'}\\
\begin{aligned}
&= \sum_{\sigma, \sigma' \in \mathfrak{S}_{2m}} \epsilon(\sigma\sigma') \Norm{U'}^{2m}\prod_{p=1}^m K(p,\sigma,\sigma')\\
&= (2m)!\Norm{U'}^{2m} \sum_{\tau \in \mathfrak{S}_{2m}} \epsilon(\tau) \prod_{p=1}^m K(p,\id,\tau).
\end{aligned}
\end{multline}
From the Definition~\eqref{definition K} of $K(p,\id,\tau)$, we get that $\prod_{p=1}^m K(p,\id,\tau) \neq 0$ if and only if~$\tau$ is a product of transpositions of the type $((2p-1) \ (2p))$. Now, if $I \subset \{1,\dots,m\}$ and $\tau = \prod_{p \in I} ((2p-1) \ (2p))$, we have:
\begin{align*}
\prod_{p=1}^m K(p,\id,\tau) &= \beta^{m-\norm{I}} & &\text{and} & \epsilon(\tau) &= (-1)^{\norm{I}},
\end{align*}
where $\norm{I}$ stands for the cardinal of $I$. Thus,
\begin{equation}
\label{equation sum on tau}
\begin{aligned}
\sum_{\tau \in \mathfrak{S}_{2m}} \epsilon(\tau) \prod_{p=1}^m K(p,\id,\tau) &= \sum_{I \subset \{1,\dots,m\}} (-1)^{\norm{I}} \beta^{m-\norm{I}} = \sum_{p=1}^m \binom{m}{p} (-1)^p \beta^{m-p}\\
&= (\beta -1)^m = (-1)^m \left(\frac{n+4}{n+2}\right)^m.
\end{aligned}
\end{equation}
Finally, by equations~\eqref{equation trace of expectation}, \eqref{sum on sigma} and~\eqref{equation sum on tau},
\begin{equation}
\label{eq bilan}
\esp{\tr\left(\prsc{S}{U'}^{\wedge 2m}\right)} = (-1)^m \left(\frac{n+4}{n+2}\right)^m \frac{((2m)!)^2}{2^m m!} \esp{\Norm{U'}^{2m}}.
\end{equation}
Then, by~\eqref{eq inter computation}, \eqref{eq bilan} and Lemma~\ref{lemma expectation norm Gaussian},
\begin{multline*}
\esp{\odet{L}\tr\left(\prsc{S}{L^{\dagger *}(U)}^{\wedge 2m}\right)}\\
=(-1)^m \left(\frac{n+4}{n+2}\right)^m \frac{((2m)!)^2}{2^m m!} (2\pi)^\frac{r}{2} \frac{\vol{\S^{n-r}}}{\vol{\S^1}}\frac{\vol{\S^{n-r+1}}}{\vol{\S^{n}}}\frac{\vol{\S^{r-1}}}{\vol{\S^{n-1}}}.
\end{multline*}
We conclude the proof of Lemma~\ref{lemma computing expectation} by computing:
\begin{equation*}
(2\pi)^\frac{r}{2} \frac{(2m)!}{2^m m!}\frac{\vol{\S^{n-r}}}{2} = (2\pi)^\frac{r}{2} \frac{2^m \Gamma\left(m+\frac{1}{2}\right)}{\sqrt{\pi}} \frac{\pi^{m+\frac{1}{2}}}{\Gamma\left(m+\frac{1}{2}\right)} = (2\pi)^\frac{n}{2}. \qedhere
\end{equation*}
\end{proof}

\subsection{Proof of Theorem \ref{theorem expected Euler characteristic algebraic case}}
\label{subsection proof Euler algebraic}

Let us now adapt the proof of the previous section to the case of real algebraic submanifolds. Once again, we need only consider the case where $r<n$ and $n-r$ is even. The framework is the same as in sections~\ref{subsection real algebraic setting} and~\ref{subsection proof volume algebraic}.

We already know that $\R H^0(\mathcal{X},\mathcal{E}\otimes \mathcal{L}^d)$ is $0$-ample by Corollary~\ref{cor 0 amplitude}. We also have the estimate \eqref{eq alg3} for $\det(E_d)$ along the diagonal, where $E_d$ is the Bergman kernel of $\mathcal{E}\otimes \mathcal{L}^d$. As in the Riemannian case, we use~\eqref{CGB alg} and Kac--Rice formula:
\begin{equation}
\label{equation expectation Euler 1 alg}
\esp{\chi(Z_s)} = \frac{1}{m! (2\pi)^\frac{n}{2}} \int_{x \in \R \mathcal{X}} \frac{ \espcond{\odet{\nabla^d_xs}\tr\left(R_s(x)^{\wedge m}\right)}{s(x)=0}}{\sqrt{\det(E_d(x,x))}}\rmes{\R \mathcal{X}},
\end{equation}
where $\nabla^d$ is any real connection on $(\mathcal{E}\otimes \mathcal{L}^d)$ and $R_s$ denotes the Riemann tensor of $Z_s$.

We need to compute the conditional expectation in~\eqref{equation expectation Euler 1 alg} for some fixed $x \in \R \mathcal{X}$. Since it does not depend on our choice of connection, we will use one that is adapted to $x$ as in sections~\ref{subsection Bergman kernel} and~\ref{subsection proof volume algebraic}. Let $x \in \R \mathcal{X}$, then we consider the scaled variables:
\begin{equation}
\label{scaled alg}
(t_d,L_d,S_d) = \left(\sqrt{\frac{\pi^n}{d^n}} s(x), \sqrt{\frac{\pi^n}{d^{n+1}}} \nabla^d_xs,\sqrt{\frac{\pi^n}{d^{n+2}}}\nabla^{2,d}_xs \right),
\end{equation}
and $(\tilde{L}_d,\tilde{S}_d) \in \R(\mathcal{E}\otimes\mathcal{L}^d)_x \otimes (T_x^*\R\mathcal{X} \oplus \Sym(T_x^*\R\mathcal{X}))$ distributed as $(L_d,S_d)$ given $t_d=0$. 

As in section~\ref{subsection proof volume algebraic}, let $(x_1,\dots,x_n)$ be real holomorphic coordinates centered at $x$ and such that $(\deron{\phantom{f}}{x_1},\dots,\deron{\phantom{f}}{x_n})$ is orthonormal at $x$. Let $(\zeta_1^d,\dots,\zeta_r^d)$ be an orthonormal basis of $\R(\mathcal{E}\otimes \mathcal{L}^d)_x$. We get an orthonormal basis of $\mathcal{J}_x^2(\mathcal{E}\otimes \mathcal{L}^d)$ similar to the one defined by~\eqref{basis volume} and~\eqref{basis Euler characteristic}.

We compute the matrix $\Lambda(d)$ of $\var(t_d,L_d,S_d)$ in this basis. For this we use the estimates for the blocks of $\Lambda(d)$ given by Proposition~\ref{proposition estimates Bergman} and \eqref{scaled alg}. Namely, for all $i$, $j$, $k$ and $l \in \{1,\dots,n\}$, with $i\leq k$ and $j\leq l$,
\begin{align}
\label{Lambda20 21 alg}
\Lambda_{20}^{ik}(d) &= O\!\left(d^{-1}\right), \hspace{15mm} \Lambda_{21}^{ik,j}(d) = O\!\left(d^{-1}\right),\\
\label{Lambda22 alg}
\Lambda_{22}^{ik,jl}(d) &= \left\{\begin{aligned} 2I_r + &O\!\left(d^{-1}\right) & &\text{if } i=j=k=l,\\ I_r + &O\!\left(d^{-1}\right) & &\text{if } i=j\neq k=l,\\ &O\!\left(d^{-1}\right) & &\text{otherwise.} \end{aligned}\right.
\end{align}
Recall from section~\ref{subsection proof volume algebraic} that the matrix of $\var(t_d,L_d)$ in this basis is $I_{r(n+1)}+O\!\left(d^{-1}\right)$. By Corollary~\ref{corollary conditional expectation}, the distribution of $(L_d,S_d)$ given $t_d=0$ is then a centered Gaussian whose variance matrix is, by blocks,
\begin{equation}
\label{equation big Lambda tilde alg}
\tilde{\Lambda}(d) =
\left(\begin{array}{c|c|c}
I_{nr} & & \\
\hline
& 2 I_{nr} & \\
\hline
& & I_{\frac{rn(n-1)}{2}}\\
\end{array}\right)
+ O\!\left(d^{-1}\right),
\end{equation}
where the empty blocks are zeros. Let $\Lambda$ denote the leading term in equation~\eqref{equation big Lambda tilde alg} and let $(L,S)\sim \mathcal{N}(0,\Lambda)$ in $\R(\mathcal{E}\otimes\mathcal{L}^d)_x \otimes (T_x^*\R\mathcal{X} \oplus \Sym(T_x^*\R\mathcal{X}))$. By Lemma~\ref{lemma convergence in distribution}, $(\tilde{L}_d,\tilde{S}_d)$ converges in distribution to $(L,S)$.

Let $U \sim \mathcal{N}(0,\Id)$ in $T_x\R \mathcal{X}$ be independent of all the other variables. By \eqref{Riemann alg},
\begin{equation}
\label{equation curvature of Zs}
R_s(x) = R(x)+\frac{1}{2}\esp{\prsc{\nabla^{2,d}_xs}{(\nabla^d_xs)^{\dagger *}(U)}^{\wedge 2}}.
\end{equation}

As in the case of Riemannian random waves~\eqref{equation cond exp}, we have:
\begin{multline}
\label{equation cond exp alg}
\espcond{\odet{\nabla^d_xs}\tr(R_s(x)^{\wedge m})}{s(x)=0} \\
= \left(\frac{d^{n+1}}{\pi^n}\right)^\frac{r}{2} \esp{\odet{\tilde{L}_d} \tr\left(\left(R(x) + d\ \esp[U]{\prsc{\tilde{S}_d}{\tilde{L}_d^{\dagger *}(U)}^{\wedge 2}}\right)^{\wedge m}\right)}.
\end{multline}
The proof of Lemma~\ref{lemma dominant term} adapts immediately to this setting, so that:
\begin{multline}
\label{equation dominant term}
\esp{\odet{\tilde{L}_d} \tr\left(\left(R(x) + d\ \esp[U]{\prsc{\tilde{S}_d}{\tilde{L}_d^{\dagger *}(U)}^{\wedge 2}}\right)^{\wedge m}\right)}\\
= d^{m} \frac{m!}{(2m)!} \esp{\odet{L} \tr\left(\prsc{S}{L^{\dagger *}(U)}^{\wedge 2m}\right)}\left(1 + O\!\left(d^{-1}\right)\right),
\end{multline}
where the error term is uniform in $x \in M$.

\begin{lem}
\label{lemma computing expectation alg}
Let $(L,S)\sim\mathcal{N}(0,\Lambda)$ in $\R(\mathcal{E}\otimes\mathcal{L}^d)_x \otimes (T_x^*\R\mathcal{X} \oplus \Sym(T_x^*\R\mathcal{X}))$ and $U \sim \mathcal{N}(0,\Id)$ in $T_xM$ be independent of $(L,S)$. We have:
\begin{equation*}
\esp{\odet{L} \tr\left(\prsc{S}{L^{\dagger *}(U)}^{\wedge 2m}\right)} = (-1)^m (2\pi)^\frac{n}{2} (2m)! \frac{\vol{\S^{n-r+1}}\vol{\S^{r-1}}}{\pi \vol{\S^n}\vol{\S^{n-1}}}.
\end{equation*}
\end{lem}

Once this lemma is proved, we get Theorem~\ref{theorem expected Euler characteristic algebraic case} immediately by~\eqref{eq alg3}, \eqref{equation expectation Euler 1 alg}, \eqref{equation cond exp alg}, \eqref{equation dominant term} and Lemma~\ref{lemma computing expectation alg}. We sketch the proof of Lemma~\ref{lemma computing expectation alg} which is, unsurprisingly, adapted from the proof of Lemma~\ref{lemma computing expectation}.

\begin{proof}
The only difference between Lemmas~\ref{lemma computing expectation} and~\ref{lemma computing expectation alg} comes from the definition of $\Lambda$ which is not the same in the algebraic case. The proof is exactly the same as the proof of Lemma~\ref{lemma computing expectation} until the definition of $K$~\eqref{definition K}. The $\Lambda^{ik,jl}$ are now given by~\eqref{equation big Lambda tilde alg}, hence we have to change the definition of $K$. In this setting,
\begin{equation*}
K(p,\sigma,\sigma') = \left\{\begin{aligned} &1 & &\text{if } \sigma(2p-1) = \sigma'(2p) \text{ and } \sigma(2p) = \sigma'(2p-1),\\
&0 & &\text{otherwise,} \end{aligned} \right.
\end{equation*}
so that $\displaystyle\prod_{p=1}^m K(p,\id,\tau)$ is $0$, unless $\tau = \tau_0 = \displaystyle\prod_{p=1}^m ((2p-1) (2p))$. Then~\eqref{equation sum on tau} becomes:
\[\sum_{\tau \in \mathfrak{S}_{2m}} \epsilon(\tau) \prod_{p=1}^m K(p,\id,\tau) = \epsilon(\tau_0) \prod_{p=1}^m K(p,\id,\tau_0) = (-1)^m.\]
This explains why the factor $\left(-\frac{n+4}{n+2}\right)^m$ becomes $(-1)^m$ in the algebraic case. What remains of the proof is as in Lemma~\ref{lemma computing expectation}.
\end{proof}

\section{Two special cases}
\label{section special cases}

In some special cases, the covariance kernel is known explicitly. It is then possible to prove more precise results. In this section, we sketch what happens on the flat torus and in the real projective space. In these cases, we get the expectation of the volume and the Euler characteristic of our random submanifolds for fixed $\lambda$ (resp. fixed $d$).

\subsection{The flat torus}

Let $\mathbb{T}^n = \R^n\diagup (2\pi \Z)^n$ denote the torus of dimension $n$ that we equip with the quotient of the Euclidean metric on $\R^n$. We have $\vol{\mathbb{T}^n}=(2\pi)^n$. We identify functions on $\mathbb{T}^n$ and $(2\pi \Z)^n$-periodic functions on $\R^n$. Then, the Laplacian is $\Delta = - \sum_{i=1}^n \deron{^2\phantom{f}}{x_i^2}$, and it is known that its eigenvalues are the integers of the form $\Norm{p}^2=\sum (p_i)^2$, with $p=(p_1,\dots,p_n) \in \N^n$. The eigenspace associated to $0$ is spanned by the constant function~$x \mapsto (2\pi)^{-\frac{n}{2}}$. For $\lambda >0$, the eigenspace associated to $\lambda$ is spanned by the normalized functions of the form:
\begin{align*}
x &\mapsto \frac{\sqrt{2}}{(2\pi)^\frac{n}{2}} \sin (\prsc{p}{x}) & &\text{and} & x &\mapsto \frac{\sqrt{2}}{(2\pi)^\frac{n}{2}} \cos (\prsc{p}{x}),
\end{align*}
where $p \in \Z^n$ is such that $\Norm{p}^2=\lambda$, and $\prsc{\cdot}{\cdot}$ is the canonical scalar product on $\R^n$.

We set $\mathbb{B}_\lambda = \{ p \in \Z^n \mid \Norm{p}^2 \leq \lambda \}$. After some computations, we get $e_\lambda$, the spectral function of the Laplacian on $\mathbb{T}^n$:
\begin{equation}
\label{equation kernel torus}
\forall \lambda \geq 0, \ \forall x, y \in \mathbb{T}^n, \qquad e_\lambda(x,y) = \frac{1}{(2\pi)^n} \sum_{p\in \mathbb{B}_\lambda} \cos(\prsc{p}{x-y}).
\end{equation}

Let $\lambda \geq 0$ and $r \in \{1,\dots,n\}$, let $V_\lambda$ be spanned by the eigenfunctions of $\Delta$ associated to eigenvalues smaller than $\lambda$, and let $f\sim \mathcal{N}(0,\Id)$ in $(V_\lambda)^r$. For all $x \in \mathbb{T}^n$, $j_x^2(f)$ is a centered Gaussian variable whose variance is determined by Lemma~\ref{lemma variance 2-jet} , Lemma~\ref{lemma E in terms of e} and the above formula~\eqref{equation kernel torus}. Note that $e_\lambda$ and its derivatives are constant along the diagonal.

We can compute explicitly the variance of $j_x^2(f)$ (which is independent of $x$) and follow the same steps as in sections~\ref{subsection proof volume harmonic} and~\ref{subsection proof Euler harmonic}. Only, this time, we can make exact computations with $\lambda$ fixed, instead of deriving asymptotics. These computations are not difficult, and similar to what we already did, so we simply state the final results. Note that the scaling of the variables has to be adapted.

\begin{prop}
\label{proposition flat torus volume}
On the flat torus $\mathbb{T}^n$, let $\lambda\geq 0$ and let $f_1,\dots,f_r$ be independent standard Gaussian functions in $V_\lambda$, with $1 \leq r\leq n$. We have:
\[\esp{\vol{Z_f}} = \left(\frac{1}{\norm{\mathbb{B}_\lambda}}\sum_{(p_1,\dots,p_n) \in \mathbb{B}_\lambda} (p_1)^2 \right)^\frac{r}{2} (2\pi)^n \frac{\vol{\S^{n-r}}}{\vol{\S^n}},\]
where $\mathbb{B}_\lambda = \{ p \in \Z^n \mid \Norm{p}^2 \leq \lambda \}$ and $\norm{\mathbb{B}_\lambda}$ denotes the cardinal of $\mathbb{B}_\lambda$.
\end{prop}

\noindent
This result was already known, see \cite{RW2008}, where Rudnick and Wigman compute the variance of $\vol{Z_f}$ when $r=1$, and the references therein.

\begin{prop}
\label{proposition flat torus Euler}
On the flat torus $\mathbb{T}^n$, let $\lambda\geq 0$ and let $f_1,\dots,f_r$ be independent standard Gaussian functions in $V_\lambda$, with $1 \leq r\leq n$. If $n-r$ is even, we have:
\[\esp{\chi(Z_f)} = (-1)^\frac{n-r}{2}\left(\frac{1}{\norm{\mathbb{B}_\lambda}}\sum_{(p_1,\dots,p_n) \in \mathbb{B}_\lambda} (p_1)^2 \right)^\frac{n}{2} (2\pi)^n \frac{\vol{\S^{n-r+1}}\vol{\S^{r-1}}}{\pi \vol{\S^n}\vol{\S^{n-1}}},\]
where $\mathbb{B}_\lambda = \{ p \in \Z^n \mid \Norm{p}^2 \leq \lambda \}$ and $\norm{\mathbb{B}_\lambda}$ denotes the cardinal of $\mathbb{B}_\lambda$.
\end{prop}

\begin{rem}
For this last result, one of the points that make the computations tractable is that the Riemann tensor of the ambient manifold is zero.
\end{rem}

\subsection{The projective space}

We consider the algebraic case with $\mathcal{X}=\C \P^n$, $\mathcal{E}= \C^r \times \C \P^n$ with the standard Hermitian metric on each fiber, and $\mathcal{L}=\mathcal{O}(1)$ the hyperplane bundle with its usual metric. Then, $\omega$ is the standard Fubini-Study form. We consider the real structures induced by the standard conjugation in $\C$.

Let $e_d$ denote the Bergman kernel of $\mathcal{L}^d$ and $E_d$ denote the Bergman kernel of $\mathcal{E}\otimes \mathcal{L}^d$. Since $\mathcal{E}$ is trivial we have a product situation, as in section~\ref{subsection harmonic setting}, and $E_d$ and $e_d$ are related as in Lemma~\ref{lemma E in terms of e}. Let $(\zeta_1,\dots,\zeta_r)$ be any orthonormal basis of $\C^r$, for all $x$, $y\in \C \P^n$,
\begin{equation}
\label{E in terms of e projective}
E_d(x,y) = \left( \sum_{q=1}^r \zeta_q \otimes \zeta_q\right) \otimes e_d(x,y).
\end{equation}

In this case, $\R \mathcal{X}=\R \P^n$ and the elements of $\R H^0(\mathcal{X},\mathcal{E}\otimes \mathcal{L}^d)$ are $r$-tuples of real homogeneous polynomials of degree $d$ in $n+1$ variables. We denote by $\R^{hom}_d[X_0,\dots,X_n]$ the space of real homogeneous polynomials of degree $d$ in $X_0,\dots,X_n$. Let $\alpha = (\alpha_0,\dots,\alpha_n) \in \N^{n+1}$, then we set $X^\alpha=X_0^{\alpha_0} \cdots X_n^{\alpha_n}$, $\norm{\alpha}= \alpha_0+\cdots+\alpha_n$ and, if $\norm{\alpha}= d$, we also set $\binom{d}{\alpha}= \frac{d!}{\alpha_0! \cdots \alpha_n !}$.

It is well-known that an orthonormal basis of $\R^{hom}_d[X_0,\dots,X_n]$ for the inner product~\eqref{equation definition prsc algebraic case} is given by the sections:
\begin{equation}
s_\alpha = \sqrt{\frac{(n+d)!}{\pi^n d!}\binom{d}{\alpha}} X^\alpha, \qquad \text{with } \norm{\alpha}=d,
\end{equation}
see \cite{BSZ2000a,BBL1996,Buer2006,Kos1993}. Then, formally,
\begin{equation}
e_d = \frac{(n+d)!}{\pi^n d!} \sum_{\norm{\alpha}=d} \binom{d}{\alpha} X^\alpha Y^\alpha = \frac{(n+d)!}{\pi^n d!} \prsc{X}{Y}^d.
\end{equation}

More precisely, we consider the local coordinates $(x_1,\dots,x_n) \mapsto [1:x_1:\dots:x_n]$, defined on a neighborhood of $[1:0:\dots:0]$, and the real holomorphic frame $s_{(d,0,\dots,0)}$ for~$\mathcal{O}(d)$ on this neighborhood. In these coordinates,
\begin{equation}
\label{equation kernel projective}
e_d(x,y) = (1+\prsc{x}{y})^d \left(s_{(d,0,\dots,0)}(x) \otimes s_{(d,0,\dots,0)}(y)\right),
\end{equation}
with $x=(x_1,\dots,x_n)$ and $y=(y_1,\dots,y_n)$. Since everything is invariant under orthogonal transformations of $\R \P^n$, this totally describes $e_d$. In particular, the covariances that appear in the proofs of Theorems~\ref{theorem expected volume algebraic case} and~\ref{theorem expected Euler characteristic algebraic case} are given, in this case, by the values of the function
\begin{equation*}
(x,y) \mapsto \frac{(n+d)!}{\pi^n d!}(1+\prsc{x}{y})^d
\end{equation*}
and its derivatives at the point $(0,0)$ in $\R^n \times \R^n$.

Let $d \in \N$, $r \in \{1,\dots,n\}$ and $s \sim \mathcal{N}(0,\Id)$ in $(\R^{hom}_d[X_0,\dots,X_n])^r\!$. For every $x\in \R \P^n$, $j_x^{2,d}(s)$ is a centered Gaussian variable whose variance is determined by Lemma~\ref{lemma variance 2-jet alg}, \eqref{E in terms of e projective} and~\eqref{equation kernel projective}. Once again, we can compute the variance of $j_x^{2,d}(s)$ explicitly, and it does not depend on $x\in \R \P^n$. Then, we can follow the steps of the proof of Theorem~\ref{theorem expected volume algebraic case}, and make exact computations for fixed degree $d$. This yields Kostlan's result \cite{Kos1993}. Note that we have to adapt the scaling of the variables.
\begin{thm}[Kostlan]
\label{theorem Kostlan}
In the real projective space $\R \P^n$, let $d \in \N$ and let $P_1,\dots,P_r$ be independent standard Gaussian polynomials in $\R^{hom}_d[X_0,\dots,X_n]$, with $1 \leq r\leq n$. Let $Z_s$ denote the common zero set of $P_1,\dots,P_r$. We have:
\[\esp{\vol{Z_s}} = d^\frac{r}{2} \vol{\R \P^{n-r}}.\]
\end{thm}

Assuming that $n-r$ is even and equals $2m$, we can also adapt the proof of Theorem~\ref{theorem expected Euler characteristic algebraic case} to get the expected Euler characteristic for fixed $d$. This computation is a bit more complex than on the torus since the Riemann curvature of $\R \P^n$ is not zero.

By Kac--Rice formula and equations~\eqref{E in terms of e projective} and \eqref{Riemann alg},
\begin{multline}
\label{equation expectation Euler proj}
\esp{\chi(Z_s)} = \frac{1}{m! (2\pi)^\frac{n}{2}} \int_{x \in \R \P^n} \frac{1}{\left(e_d(x,x)\right)^\frac{n}{2}}\times \cdots\\
\esp{\odet{\nabla^d_xs}\tr\left(\left(R(x)+\frac{1}{2}\esp{\prsc{\nabla^{2,d}_xs}{(\nabla^d_xs)^{\dagger *}(U)}^{\wedge 2}}\right)^{\wedge m}\right)}\rmes{\R \P^n},
\end{multline}
where $R$ is the Riemann curvature of $\R \P^n$ and $U \sim \mathcal{N}(0,\Id)$ in $T_x\R \P^n$. We used the fact that $s(x)$, $\nabla^d_xs$ and $\nabla^{2,d}_xs$ are independent in this particular case. This is why the expectation is not conditioned on $s(x)=0$.

Since everything is invariant under orthogonal transformations of $\R \P^n$, we only need to compute the expectation in the integrand at the point $x=[1:0:\dots:0]$. We do this in the same chart as for~\eqref{equation kernel projective} above. We get:
\[R(x) = \frac{1}{2} \sum_{1\leq i,j \leq n} dx_i \wedge dx_j \otimes dx_i \wedge dx_j.\]

Following the computations of section~\ref{subsection proof Euler harmonic}, we get (after a suitable rescaling) that the expectation in the integrand of \eqref{equation expectation Euler proj} equals:
\begin{equation}
\label{eq last}
d^\frac{r}{2} \sum_{q=0}^m \frac{(d-1)^q}{2^q} \binom{m}{q} \esp{\frac{\Norm{X_{n-r+2}}\dots \Norm{X_n}}{\Norm{X_{n-r+1}}^{2q-1}}} \esp{\tr\left(R(x)^{\wedge (m-q)}\wedge \prsc{S}{U'}^{\wedge 2q}\right)},
\end{equation}
where $U'\in \R^r$, $S \in (\R\mathcal{O}(d)_x)^r \otimes \Sym(T_x^*\R \P^n )$ and $X_p \in \R^p$ (for $n-r+1 \leq p \leq n$) are globally independent standard Gaussian vectors. By the same kind of computations as in the proof of Lemma~\ref{lemma computing expectation}, we get:
\begin{equation}
\esp[S]{\prsc{S}{U'}^{\wedge 2}} = -\Norm{U'}^2 \sum_{1\leq i,j \leq n} dx_i \wedge dx_j \otimes dx_i \wedge dx_j.
\end{equation}
Restricting $R$ and $S$ to the span of $\left(\deron{\phantom{f}}{x_1},\dots,\deron{\phantom{f}}{x_{2m}}\right)$, we have:
\[\forall q \in \{0,\dots,m\}, \qquad \tr\left(\R(x)^{\wedge m-q} \wedge \esp[S]{\prsc{S}{U'}^{\wedge 2}}^{\wedge q}\right) = (-1)^q\frac{1}{2^{m-q}} (2m)!\Norm{U'}^{2q}.\]
As in the proof of Lemma~\ref{lemma computing expectation}, we compute the expectation of this term with respect to the variable $U'$ by Lemma~\ref{lemma expectation norm Gaussian}. The same lemma allows us to compute the value of~\eqref{eq last}. Finally, we recover Bürgisser's result \cite{Buer2006}.

\begin{thm}[Bürgisser]
\label{theorem Burgisser}
In the real projective space $\R \P^n$, let $d \in \N$ and let $P_1,\dots,P_r$ be independent standard Gaussian polynomials in $\R^{hom}_d[X_0,\dots,X_n]$, with $1 \leq r\leq n$. Let $Z_s$ denote the common zero set of $P_1,\dots,P_r$. If $n-r$ is even, we have:
\[\esp{\chi(Z_s)} = d^\frac{r}{2} \sum_{p=0}^{\frac{n-r}{2}} (1-d)^p \ \frac{\Gamma\!\left(p+\frac{r}{2}\right)}{p! \ \Gamma\!\left(\frac{r}{2}\right)}.\]
\end{thm}

\appendix

\section{Concerning Gaussian vectors}
\label{section gaussian vectors}

In this appendix we survey the few facts we need concerning random vectors, especially Gaussian ones. It is essentially borrowed from \cite[Appendix~A]{Nic2014}. We include it here for the reader's convenience.

\subsection{Variance and covariance as tensors}
\label{subsection basics of random vectors}

Let $V$ be a real vector space of finite dimension and $X$ a random vector with values in $V$. For any $\xi \in V^*$, $\xi(X)$ is a real random variable. From now on, we assume that these variables are square integrable.
\begin{dfn}
The \emph{expectation} (or \emph{mean}) of $X$ is the linear form on $V^*$ defined by:
\begin{equation}
\label{equation definition expectation}
m_X : \xi \mapsto \esp{\xi(X)}.
\end{equation}
If $m_X = 0$, we say that $X$ (resp. its distribution $\loi{X}$) is \emph{centered}.
\end{dfn}
\noindent
Under the canonical isomorphism $V^{**}\simeq V$, we have $m_X = \displaystyle\int_V x \loi{X}$.

\begin{dfn}
The \emph{variance} of $X$ is the non-negative symmetric bilinear form on $V^*$ defined by:
\begin{equation}
\label{equation definition covariance}
\var(X) : (\xi,\eta) \mapsto \esp{\xi(X-m_X)\eta(X-m_X)}.
\end{equation}
\end{dfn}
\begin{rem}
Traditionally, the term ``variance'' is only used when $V$ has dimension $1$ and one speaks of ``covariance'' when $\dim(V) \geq 2$. We chose to use the term ``covariance'' for couples of distinct random vectors (see below) and ``variance'' otherwise. This is the convention of \cite{AW2009}, for example.
\end{rem}

As a bilinear form on $V^* \times V^*$, $\var(X)$ is naturally an element of $V \otimes V$ and we have the following lemma.
\begin{lem}
\label{lemma variance as a tensor}
Let $X$ be a random vector in $V$, then we have:
\begin{equation}
\label{equation variance as a tensor}
\var(X) = \esp{(X-m_X) \otimes (X-m_X)}.
\end{equation}
\end{lem}

\begin{proof}
For any $\xi$ and $\eta \in V^*$, we have:
\begin{align*}
\var(X)(\xi,\eta)&=\esp{\xi(X-m_X)\eta(X-m_X)}=\esp{(\xi \otimes \eta)((X-m_X) \otimes (X-m_X))}\\
&= (\xi \otimes \eta) \esp{(X-m_X) \otimes (X-m_X)}.
\qedhere
\end{align*}
\end{proof}

\begin{dfn}
The \emph{variance operator} of $X$ is the linear map $\Lambda_X : V^* \to V$ such that, for any $\xi$ and $\eta \in V^*$,
\begin{equation}
\label{equation definition covariance operator}
\xi\left(\Lambda_X \eta \right) = \var(X)(\xi,\eta).
\end{equation}
\end{dfn}

By Lemma~\ref{lemma variance as a tensor} we have:
\begin{equation}
\label{equation variance operator as a tensor}
\Lambda_X:\eta \mapsto \esp{(X-m_X)\otimes \eta(X-m_X)}.
\end{equation}

If $V=V_1 \oplus V_2$ and $X=(X_1,X_2)$, with $X_i$ a random vector in $V_i$, then $m_X = m_{X_1}+m_{X_2}$ and the variance form $\var(X)$ splits accordingly into four parts:
\begin{align*}
\var(X_1) &: V_1^* \times V_1^* \to \R, & & & \cov(X_1,X_2) &: V_1^* \times V_2^* \to \R, \\ \var(X_2) &: V_2^* \times V_2^* \to \R & &\text{and} & \cov(X_2,X_1) &: V_2^* \times V_1^* \to \R.
\end{align*}
These bilinear forms are associated, as above, to the following operators:
\begin{align*}
\Lambda_{11} &: V_1^* \to V_1, & & & \Lambda_{12} &: V_2^* \to V_1, \\ \Lambda_{22} &: V_2^* \to V_2 & &\text{and} & \Lambda_{21} &: V_1^* \to V_2.
\end{align*}
Since $\var(X)$ is symmetric, $\cov(X_1,X_2)(\xi,\eta)=\cov(X_2,X_1)(\eta,\xi)$ for any $\xi$ and $\eta$.

\begin{dfn}
We say that $\cov(X_1,X_2)$ is the \emph{covariance} of $X_1$ and $X_2$, and that $\Lambda_{12}$ is their \emph{covariance operator}.
\end{dfn}

As above, $\cov(X_1,X_2)$ is naturally an element of $V_1 \otimes V_2$.
\begin{lem}
\label{lemma covariance as a tensor}
Let $X_1$ and $X_2$ be random vectors in $V_1$ and $V_2$ respectively, then we have: \[\cov(X_1,X_2)= \esp{(X_1-m_{X_1}) \otimes (X_2-m_{X_2})}.\]
Moreover, for any $\eta \in V_2^*$, \qquad $\Lambda_{12}(\eta) = \esp{(X_1 - m_{X_1}) \otimes \eta(X_2 - m_{X_2})}$.
\end{lem}

Let $L:V\to V'$ be a linear map between finite-dimensional vector spaces and $X$ be a random vector in $V$. Then $L(X)$ is a random vector in $V'$ with $\loi{L(X)} = L_*(\loi{X})$. An immediate consequence of \eqref{equation definition expectation}, \eqref{equation definition covariance} and \eqref{equation definition covariance operator} is that:
\begin{align}
\label{equation pushforward mean}
m_{L(X)} &= m_X \circ L^*,\\
\label{equation pushforward variance form}
\var(L(X)) &= \var(X)(L^* \cdot ,L^* \cdot),\\
\intertext{and}
\label{equation pushforward variance operator}
\Lambda_{L(X)} &= L \Lambda_X L^*,
\end{align}
where $L^*: (V')^* \to V^*$ is defined by $L^* : \xi \mapsto (\xi \circ L)$.

If $X$ is a random vector in a Euclidean space $(V,\prsc{\cdot}{\cdot})$, we can see $\var(X)$ as a bilinear symmetric form on $V$, and $\Lambda_X$ as a self-adjoint operator on $V$. Then, by~\eqref{equation variance as a tensor} and~\eqref{equation variance operator as a tensor} :
\begin{align*}
\var(X) &= \esp{(X-m_X)^* \otimes (X-m_X)^*},\\
\Lambda_X &= \esp{(X-m_X) \otimes (X-m_X)^*},
\end{align*}
where for any $v \in V$, we set $v^* = \prsc{v}{\cdot} \in V^*$.

If $V=V_1 \oplus V_2$, we can see $\Lambda_{12}$ as a linear operator from $V_2$ to $V_1$ and by Lemma~\ref{lemma covariance as a tensor}:
\begin{align*}
\cov(X_1,X_2)&= \esp{(X_1-m_{X_1})^* \otimes (X_2-m_{X_2})^*},\\
\Lambda_{12} &= \esp{(X_1-m_{X_1}) \otimes (X_2-m_{X_2})^*},
\end{align*}

\begin{lem}
\label{lemma random scalar product}
Let $X$ be a random vector in a Euclidean space $V$, then we have:
\begin{equation}
\label{equation random scalar product var}
\forall v \in V, \forall w \in V, \quad \esp{\prsc{v}{X-m_X}\prsc{w}{X-m_X}}=\prsc{v}{\Lambda_X w}.
\end{equation}
\end{lem}

\begin{proof}
Let $v$ and $w\in V$, then we have:
\begin{align*}
\esp{\prsc{v}{X-m_X}\prsc{w}{X-m_X}} &= \var(X)(v,w) & &\text{as a bilinear form on }V,\\
&= \prsc{v}{\Lambda_X w} & &\text{where } \Lambda_X : V\to V.\qedhere
\end{align*}
\end{proof}

\subsection{Gaussian vectors}
\label{subsection gaussian vectors}

The following material can be found either in \cite[section~1.2]{AW2009} or \cite[section~1.2]{AT2007}. We present it in a coordinate-free fashion, in the spirit of \cite{Nic2014}.

Let $m \in \R$ and $\sigma \geq 0$, then the \emph{Gaussian} (or \emph{normal}) distribution on $\R$ with expectation $m$ and variance $\sigma^2$ is the distribution whose characteristic function is $\xi \mapsto \exp\left(im\xi - \frac{1}{2}\sigma^2\xi^2\right)$. If $\sigma = 0$, this is the Dirac measure centered at $m$, otherwise it has a density with respect to the Lebesgue measure, given by $x \mapsto \frac{1}{\sigma\sqrt{2\pi}}\exp\left(-\frac{(x-m)^2}{2\sigma^2}\right)$.

Let $V$ be a real vector space of dimension $n$, then a random vector $X$ in $V$ is said to be \emph{Gaussian}, or \emph{normally distributed}, if for any $\xi \in V^*$, $\xi(X)$ is a Gaussian variable in $\R$. Recall that  a Gaussian vector has finite moments of all orders and that its distribution is totally determined by its expectation and variance. We denote by $\mathcal{N}(m,\Lambda)$ the Gaussian distribution with expectation $m$ and variance operator $\Lambda$ and by $X \sim \mathcal{N}(m,\Lambda)$ the fact that $X$ is distributed according to $\mathcal{N}(m,\Lambda)$. From~\eqref{equation pushforward mean} and \eqref{equation pushforward variance operator} we deduce the following.

\begin{lem}
\label{lemma pushforward gaussian variable}
Let $L:V \to V'$ be a linear map between finite-dimensional vector spaces and $X\sim \mathcal{N}(m,\Lambda)$ in $V$. Then $L(X) \sim \mathcal{N}(Lm,L\Lambda L^*)$ in $V'$.
\end{lem}

If $V = V_1 \oplus V_2$ and $X=(X_1,X_2) \sim \mathcal{N}(m,\Lambda)$ then, with the notations of section~\ref{subsection basics of random vectors}, Lemma~\ref{lemma pushforward gaussian variable} shows that $X_1 \sim \mathcal{N}(m_{X_1},\Lambda_{11})$ and $X_2 \sim \mathcal{N}(m_{X_2},\Lambda_{22})$. Besides, $X_1$ and $X_2$ are independent if and only if $\cov(X_1,X_2)=0$, or equivalently $\Lambda_{12}=0$.

\begin{prop}[Regression formula]
\label{proposition regression formula}
Let $X=(X_1,X_2)$ be a Gaussian vector in $V_1\oplus V_2$. If $\var(X_1)$ is non-degenerate then $X_2$ has the same distribution as
\begin{equation*}
m_{X_2} + \Lambda_{21}(\Lambda_{11})^{-1}\left(X_1-m_{X_1}\right) + Y
\end{equation*}
where $Y$ is a centered Gaussian vector in $V_2$ with variance operator $\Lambda_{22}-\Lambda_{21}(\Lambda_{11})^{-1}\Lambda_{12}$, independent of $X_1$.
\end{prop}

\noindent
This is shown in \cite[prop.~1.2]{AW2009}. From this, we deduce that the distribution of $X_2$ given $X_1=x_1$ is Gaussian in $V_2$ with expectation $m_{X_2}+\Lambda_{21}(\Lambda_{11})^{-1}(x_1-m_{X_1})$ and variance operator $\Lambda_{22}-\Lambda_{21}(\Lambda_{11})^{-1}\Lambda_{12}$. We use this in the case where $X$ is centered and $x_1=0$.

\begin{cor}
\label{corollary conditional expectation}
Let $X=(X_1,X_2)$ be a centered Gaussian vector in $V_1\oplus V_2$ and assume that $\var(X_1)$ is non-degenerate. Then the distribution of $X_2$ given $X_1 = 0$ is a centered Gaussian in $V_2$ with variance operator $\Lambda_{22}- \Lambda_{21} (\Lambda_{11})^{-1} \Lambda_{12}$.
\end{cor}

In what follows, we assume that $V$ is a Euclidean space. Recall that in this case, we can see the variance operator of a random vector as an endomorphism of $V$. We will say that $\mathcal{N}(0,\Id)$ is the \emph{standard normal distribution} on $V$, where $\Id$ denotes the identity map on $V$.

If $\var(X)$ is non-degenerate, then $\loi{X}$ has the following density with respect to the Lebesgue measure on $V$:
\begin{equation*}
x \mapsto \frac{1}{(2\pi)^\frac{n}{2}\sqrt{\det(\Lambda_X)}}\exp\left(-\frac{1}{2}\prsc{(\Lambda_X)^{-1} (x-m)}{x-m}\right).
\end{equation*}
If $\var(X)$ is singular, $\loi{X}$ is supported on $\ker(\Lambda_X)^\perp$.

\begin{lem}
\label{lemma convergence in distribution}
Let $(X_k)_{k \in \N}$ be a sequence of random vectors in $V$ such that $X_k \sim \mathcal{N}(m_k,\Lambda_k)$ for all $k\in \N$. We assume that $m_k \xrightarrow[k \to +\infty]{} m$ and $\Lambda_k \xrightarrow[k \to +\infty]{} \Lambda$. Then $X_k$ converges in distribution to $\mathcal{N}(m,\Lambda)$.
\end{lem}

\begin{proof}
Under the hypothesis of the lemma, the characteristic function of $X_k$ converges pointwise to the characteristic function of a Gaussian vector $X\sim \mathcal{N}(m,\Lambda)$. Then, Lévy's continuity theorem gives the result.
\end{proof}

We conclude this appendix by computing two Gaussian expectations.
\begin{lem}
\label{lemma expectation norm Gaussian}
Let $X \sim \mathcal{N}(0,\Id)$ with values in a Euclidean space of dimension $n$ and let $k \in \Z$ such that $k > -n$. Then,
\begin{equation*}
\esp{\Norm{X}^k} = (2\pi)^\frac{k}{2} \frac{\vol{\S^{n-1}}}{\vol{\S^{n+k-1}}}.
\end{equation*}
\end{lem}

\begin{proof}
\begin{align*}
\esp{\Norm{X}^k} &= \frac{1}{(2\pi)^\frac{n}{2}}\int_{V} \Norm{x}^k e^{-\frac{1}{2}\Norm{x}^2} \dx x = \frac{\vol{\S^{n-1}}}{(2\pi)^\frac{n}{2}}\int_0^{+\infty} r^{k+n-1}e^{-\frac{1}{2}r^2} \dx r \\
&= \frac{\vol{\S^{n-1}}}{(2\pi)^\frac{n}{2}} \int_0^{+\infty} (2t)^{\frac{k+n}{2}-1}e^{-t} \dx t = \frac{(2\pi)^\frac{k}{2}\vol{\S^{n-1}}}{2\pi^\frac{k+n}{2}}\  \Gamma\!\left(\frac{k+n}{2}\right).
\qedhere
\end{align*}
\end{proof}
\noindent
Assuming we proved Proposition~\ref{proposition same distribution} (more precisely its Corollary~\ref{corollary distribution orthogonal det}), we have the following.
\begin{lem}
\label{lemma expectation odet of L}
Let $V$ and $V'$ be two Euclidean spaces of dimension $n$ and $r$ respectively, with $1 \leq r \leq n$. Let $L \sim \mathcal{N}(0,\Id)$ in $V' \otimes V^*$. Then:
\[\esp{\odet{L}} = (2\pi)^\frac{r}{2} \frac{\vol{\S^{n-r}}}{\vol{\S^n}}.\]
\end{lem}

\begin{proof}
By Corollary~\ref{corollary distribution orthogonal det}, $\odet{L}$ is distributed as $\Norm{X_{n-r+1}}\cdots\Norm{X_n}$, with $X_p$ a standard Gaussian in $\R^p$ for all $p$ and $X_{n-r+1},\dots,X_n$ independent. Then, using Lemma~\ref{lemma expectation norm Gaussian},
\begin{equation*}
\esp{\odet{L}} = \esp{\Norm{X_{n-r+1}}\cdots\Norm{X_n}}=\prod_{p=n-r+1}^n \esp{\Norm{X_p}}= (2\pi)^\frac{r}{2} \frac{\vol{\S^{n-r}}}{\vol{\S^n}}.\qedhere
\end{equation*}
\end{proof}

\section{Proof of Proposition \ref{proposition same distribution}}
\label{section a useful result}

This appendix is devoted to the proof of Proposition~\ref{proposition same distribution} which is a reformulation of~\cite[prop.~3.12]{Buer2006}. Let $V$ and $V'$ be two Euclidean spaces of dimension $n$ and $r$ respectively, with $1 \leq r \leq n$. The space $V' \otimes V^*$ of linear maps from $V$ to $V'$ comes with a natural scalar product induced by those on $V$ and $V'$. The set of linear maps of rank less than $r$ is an algebraic submanifold of $V'\otimes V^*$ of codimension at least $1$, hence it has measure $0$ for any non-singular Gaussian measure. Let $L$ be a standard Gaussian vector in $V' \otimes V^*$, then the rank of $L$ is $r$ almost surely. Hence $L^\dagger$ is well-defined almost surely (recall Definition~\ref{definition pseudo-inverse}).

We introduce some further notations. Let $\mathcal{B} \subset V'\otimes V^*$ denote the set of maps of rank~$r$. We set $\mathcal{F}=\{(L,U)\in \mathcal{B} \times V \mid U \in \ker(L)^\perp \}$ and $\mathcal{S}=\{(L,U)\in \mathcal{B} \times V \mid U \in \S(\ker(L)^\perp) \}$. Here and in the sequel, $\S(\cdot)$ stands for the unit sphere of the concerned space. Given $L \in \mathcal{B}$ and $U \in V$, we denote by $\tilde{U}$ the orthogonal projection of $U$ onto $\ker(L)^\perp$. Then we set:
\begin{align}
\label{equation theta'}
\rho &= \lVert\tilde{U}\rVert, & \theta &= \frac{\tilde{U}}{\lVert\tilde{U}\rVert} & &\text{and} & \theta' &= \frac{(L^\dagger)^*\theta}{\Norm{(L^\dagger)^*\theta}}.
\end{align}
Note that $L^*\theta' = \displaystyle\frac{(L^\dagger L)^*\theta}{\Norm{(L^\dagger)^*\theta}} = \frac{\theta}{\Norm{(L^\dagger)^*\theta}}$, hence $\Norm{L^*\theta'}=\displaystyle \frac{1}{\Norm{(L^\dagger)^*\theta}}$ and finally:
\begin{align}
\label{equation L dagger *}
\theta &=  \frac{L^*\theta'}{\Norm{L^*\theta'}} & &\text{and} & L^{\dagger *}U &= L^{\dagger *} \tilde{U} = \rho L^{\dagger *} \theta = \frac{\rho}{\Norm{L^*\theta'}} \theta'.
\end{align}

We choose orthonormal bases $(e_1,\dots,e_n)$ and $(e'_1,\dots,e'_r)$, of $V$ and $V'$ respectively, such that $e_r=\theta$, $e'_r=\theta'$ and $(e_1,\dots,e_r)$ is a basis of $\ker(L)^\perp$. Then,
\begin{equation*}
\forall i \in \{1, \dots,n\}, \quad \prsc{Le_i}{\theta'} = \prsc{e_i}{L^*\theta'} = \Norm{L^*\theta'}\prsc{e_i}{\theta}.
\end{equation*}
Thus the matrix of $L$ in these bases has the form:
\begin{equation}
\label{equation matrix L}
\left( \begin{array}{c|c|c}
A & \begin{smallmatrix} * \vspace{-1mm} \\ \vdots \\ * \vspace{1mm} \end{smallmatrix} & 0 \\ \hline \begin{smallmatrix} 0 & \cdots & 0 \end{smallmatrix} & \Norm{L^*\theta'} & \begin{smallmatrix} 0 & \cdots & 0 \end{smallmatrix}
\end{array}\right),
\end{equation}
and $\odet{L}=\norm{\det(A)}\Norm{L^*\theta'}$.

Let $\pi_\theta$ and $\pi_{\theta'}$ denote the orthogonal projections along $\R\cdot\theta$ in $V$ and along $\R\cdot \theta'$ in $V'$ respectively. We define $L':V \to (\R\cdot\theta')^\perp$ by $L'= \pi_{\theta'} \circ L \circ \pi_\theta$. Then $\norm{\det(A)}=\odet{L'}$, and $L'$ does not depend on our choice of bases. Finally, we have:
\begin{equation}
\label{equation L'}
\left(\odet{L},L^{\dagger *}U\right) = \left(\odet{L'}\Norm{L^*\theta'},\frac{\rho \theta'}{\Norm{L^*\theta'}}\right).
\end{equation}
To prove Proposition~\ref{proposition same distribution}, we will show that  $\odet{L'}$, $\Norm{L^*\theta'}$, $\rho$ and $\theta'$ are independent and identify their distributions.

If $L$ and $U$ are independent standard Gaussians, then almost surely $L \in \mathcal{B}$ and we can consider $(L,U)$ as a random element of $\mathcal{B}\times V$. Then $(L,\tilde{U})$ is a random element of $\mathcal{F}$ and its distribution is characterized by:
\begin{align*}
\esp{\phi(L,\tilde{U})} &= \int_{L\in \mathcal{B}} \left( \int_{U \in V}\phi(L,\tilde{U}) \dx\nu_n(U)\right)\dx\nu_{nr}(L)\\
&= \int_{L \in \mathcal{B}} \left( \int_{\tilde{U} \in \ker(L)^\perp} \phi(L,\tilde{U}) \dx\nu_r(\tilde{U})\right) \dx\nu_{nr}(L),
\end{align*}
for any bounded continuous function $\phi:\mathcal{F}\to \R$. Recall that $\dx\nu_N$ stands for the standard Gaussian measure in dimension $N$. We get the distribution of $(L,\theta,\rho) \in \mathcal{S}\times \R_+$ by a polar change of variables in the innermost integral: for any bounded continuous $\phi : \mathcal{S}\times \R_+ \to \R$,
\begin{equation*}
\esp{\phi(L,\theta,\rho)} = \int_{L \in \mathcal{B}}  \int_{\theta \in \S(\ker(L)^\perp)} \int_{\rho =0}^{+\infty} \phi(L,\theta,\rho)\rho^{r-1} e^{-\frac{\rho^2}{2}} \frac{\dx \rho}{(2\pi)^\frac{r}{2}} \dx \theta \dx\nu_{nr}(L),
\end{equation*}
where $\dx \rho$ is the Lebesgue measure on $\R$ and $\dx \theta$ is the Euclidean measure on the sphere $\S(\ker(L)^\perp)$.

This distribution is a product measure on $\mathcal{S}\times \R_+$, thus $(L,\theta)$ and $\rho$ are independent variables. Since $\left(\odet{L'},\theta',\Norm{L^*\theta'}\right)$ only depends on $(L,\theta)$, this triple is independent of $\rho$. Besides, $\rho$ is distributed as the norm of a standard Gaussian vector in $\R^r$ since its density with respect to the Lebesgue measure is $\rho \mapsto \vol{\S^{r-1}}(2\pi)^{-\frac{r}{2}} \rho^{r-1}e^{-\frac{\rho^2}{2}}$ on $\R_+$ and vanishes elsewhere. Finally, the distribution of $(L,\theta)$ satisfies:
\begin{equation}
\label{equation distribution on S}
\esp{\phi(L,\theta)}= \int_{L \in \mathcal{B}} \int_{\theta \in \S(\ker(L)^\perp)} \phi(L,\theta) \frac{\dx \theta}{\vol{\S^{r-1}}} \dx\nu_{nr}(L),
\end{equation}
for any bounded and continuous $\phi : \mathcal{S} \to \R$.

We will now compute the distribution of $(L,\theta')$ in $\mathcal{B}\times \S(V')$.

\begin{lem}
\label{lemma distribution of L theta'}
For any bounded and continuous $\phi : \mathcal{B}\times \S(V') \to \R$,
\begin{equation}
\label{equation distribution of L theta'}
\esp{\phi(L,\theta')}= \int_{\theta' \in \S(V')} \int_{L \in \mathcal{B}} \phi(L,\theta') \frac{\odet{L'}}{\Norm{L^*\theta'}^{r-1}} e^{-\frac{\Norm{L}^2}{2}}\frac{\dx \theta'}{\vol{\S^{r-1}}} \frac{\dx L}{(2\pi)^\frac{nr}{2}},
\end{equation}
where $\dx \theta'$ is the Euclidean measure on $\S(V')$, $\dx L$ is the Lebesgue measure on $V'\otimes V^*$ and $L'$ is defined as in \eqref{equation L'}.
\end{lem}

\begin{proof}
Fixing some $\phi$, we see from~\eqref{equation distribution on S} and~\eqref{equation theta'} that:
\[\esp{\phi(L,\theta')}=\int_{L \in \mathcal{B}} \int_{\theta \in \S(\ker(L)^\perp)} \phi\left(L,\frac{(L^\dagger)^*\theta}{\Norm{(L^\dagger)^*\theta}}\right) e^{-\frac{\Norm{L}^2}{2}} \frac{\dx \theta}{\vol{\S^{r-1}}} \frac{\dx L}{(2\pi)^\frac{nr}{2}}.\]
Then we make the change of variables $\theta' = \psi(\theta)=\displaystyle\frac{(L^\dagger)^*\theta}{\Norm{(L^\dagger)^*\theta}}$ in the innermost integral, with $L$ fixed. Recalling~\eqref{equation L dagger *}, we have $\psi^{-1}:\theta' \mapsto \displaystyle\frac{L^*\theta'}{\Norm{L^*\theta'}}$ from $\S(V')$ to $\S(\ker(L)^\perp)$. Now, the differential of $\psi^{-1}$ at $\theta' \in \S(V')$ satisfies:
\[\forall v \in (\R\cdot\theta')^\perp,\quad d_{\theta'}(\psi^{-1})\cdot v = \frac{1}{\Norm{L^*\theta'}}\left(L^*v - \prsc{L^*v}{\frac{L^*\theta'}{\Norm{L^*\theta'}}}\frac{L^*\theta'}{\Norm{L^*\theta'}}\right)=\frac{\pi_\theta(L^*v)}{\Norm{L^*\theta'}}.\]
As above, we choose an orthonormal basis $\left(e'_1,\dots,e'_{r-1}\right)$ of $\left(\R\cdot\theta'\right)^\perp$ and an orthonormal basis $(e_1,\dots,e_{r-1})$ of $\left(\R\cdot \theta \oplus \ker(L)\right)^\perp$. In these coordinates we have:
\[\norm{\det(d_{\theta'}(\psi^{-1}))} = \frac{\norm{\det\left(\pi_\theta \circ L^*_{/(\theta')^\perp}\right)}}{\Norm{L^*\theta'}^{r-1}} = \frac{\norm{\det(A^*)}}{\Norm{L^*\theta'}^{r-1}} = \frac{\norm{\det(A)}}{\Norm{L^*\theta'}^{r-1}} = \frac{\odet{L'}}{\Norm{L^*\theta'}^{r-1}},\]
where $A$ is as in \eqref{equation matrix L} and $L'$ as in \eqref{equation L'}. This proves~\eqref{equation distribution of L theta'}.
\end{proof}

We can now compute the joint distribution of $\left(\odet{L'},\theta',\Norm{L^*\theta'}\right)$ from the one of $(L,\theta')$. We fix $\theta' \in \S(V')$ and an orthonormal basis $\left(e'_1,\dots e'_r\right)$ of $V'$ such that $e'_r=\theta'$. The choice of $L \in V'\otimes V^*$ is equivalent to the choice of the $r$ independent standard Gaussian vectors $L^*e'_1,\dots,L^*e'_r$ in $V$. For simplicity, we set $L_i = L^*e'_i$. Note that if we choose a basis for $V$ as well, these are the rows of the matrix of $L$. We can rewrite \eqref{equation distribution of L theta'} as:
\[\esp{\phi(L,\theta')}= \int_{\theta' \in \S(V')}\int_{L_1,\dots,L_{r-1} \in V}\int_{L_r \in V}\phi(L,\theta') \frac{\odet{L'}}{\Norm{L_r}^{r-1}} \frac{e^{-\frac{1}{2}\sum \Norm{L_i}^2}\dx \theta'}{\vol{\S^{r-1}}} \frac{\dx L_1 \cdots \dx L_r}{(2\pi)^\frac{nr}{2}},\]
where $\dx L_i$ denotes the Lebesgue measure in the $i$-th copy of $V$. We set $\alpha_r = \frac{L_r}{\Norm{L_r}}$ and $\rho_r=\Norm{L_r}$. Here, $L'=\pi_{\theta'}\circ L \circ \pi_{\alpha_r}$ depends on $\alpha_r$ and $\theta'$ but not on $\rho_r$. Making a polar change of variables, the above integral equals:
\[\int_{\substack{\theta' \in \S(V')\\ \alpha_r \in \S(V)\\L_1,\dots,L_{r-1}\in V}}\!\int_{\rho_r=0}^{+\infty}\phi(L,\theta')\frac{\left(\rho_r\right)^{n-r}e^{-\frac{\rho_r^2}{2}}\dx \rho_r}{(2\pi)^\frac{n-r+1}{2}} \frac{e^{-\frac{1}{2}\sum_{i=1}^{r-1} \Norm{L_i}^2}\odet{L'}}{\vol{\S^{r-1}}} \frac{\dx \theta' \dx \alpha_r \dx L_1 \cdots \dx L_{r-1}}{(2\pi)^\frac{(n+1)(r-1)}{2}}.\]

Then, $\rho_r = \Norm{L^*\theta'}$ is independent of $(\theta',\alpha_r,L_1,\dots,L_{r-1})$, hence of $\left(\theta',\odet{L'}\right)$. Moreover $\rho_r$ is distributed as the norm of a standard Gaussian in $\R^{n-r+1}$, since it has the same density. Finally, $\left(\theta',\alpha_r,L_1,\dots,L_{r-1}\right)$ has the density:
\[\left(\theta',\alpha_r,L_1,\dots,L_{r-1}\right) \mapsto \frac{e^{-\frac{1}{2}\sum_{i=1}^{r-1} \Norm{L_i}^2}}{(2\pi)^\frac{(n+1)(r-1)}{2}} \frac{\odet{L'}}{\vol{\S^{r-1}}\vol{\S^{n-r}}}\]
with respect to $\dx \theta' \otimes \dx \alpha_r \otimes \dx L_1 \otimes \cdots \otimes \dx L_{r-1}$.

For $i \in \{1,\dots,r-1\}$ we denote by $L_i^\perp$ the orthogonal projection of $L_i$ onto the orthogonal of the subspace spanned by $(\alpha_r,L_1,\dots,L_{i-1})$.

\begin{lem}
\label{lemma expression of odet(L')}
For any $L \in \mathcal{B}$, $\odet{L'} = \Norm{L_1^\perp} \cdots \Norm{L_{r-1}^\perp}$.
\end{lem}

\begin{proof}
If one of the $L_i^\perp$ is zero, then the vectors $\alpha_r,L_1,\dots,L_{r-1}$ are linearly dependent and $L$ is singular. Since we assumed $L \in \mathcal{B}$ this is not the case and $\left(\frac{L_1^\perp}{\Norm{L_1^\perp}},\dots,\frac{L_{r-1}^\perp}{\Norm{L_{r-1}^\perp}} \right)$ is an orthonormal basis of $\ker(L')^\perp$. Writing the matrix of the restriction of $L'$ to $\ker(L')^\perp$ in this basis and $\left(e'_1,\dots,e'_{r-1}\right)$, we see that it is lower triangular with diagonal coefficients $\Norm{L_1^\perp}$,\dots,$\Norm{L_{r-1}^\perp}$. This proves the lemma.
\end{proof}

Let $\phi$ be a continuous bounded function from $\S (V')\times \R_+$ to $\R$. We have:
\begin{multline*}
\esp{\phi\left(\theta',\odet{L'}\right)}\\
\begin{aligned}
&=\int \phi\left(\theta',\odet{L'}\right)\frac{e^{-\frac{1}{2}\sum_{i=1}^{r-1} \Norm{L_i}^2}\odet{L'}}{(2\pi)^\frac{(n+1)(r-1)}{2}} \frac{\dx \theta' \dx \alpha_r \dx L_1 \dots \dx L_{r-1}}{\vol{\S^{r-1}}\vol{\S^{n-r}}}\\
&= \int_{\alpha_r,\theta'}\int_{L_1^\perp}\hspace{-1mm}\dots \int_{L_{r-1}^\perp}\! \phi\!\left(\theta',\prod_{i=1}^{r-1} \Norm{L_i^\perp}\right) \frac{e^{-\frac{1}{2}\!\sum_{i=1}^{r-1} \Norm{L_i^\perp}^2}\prod_{i=1}^{r-1} \Norm{L_i^\perp}}{(2\pi)^\frac{(2n+2-r)(r-1)}{4}} \frac{\dx L_{r-1}^\perp \dots \dx L_1^\perp \dx \theta' \dx \alpha_r}{\vol{\S^{r-1}}\vol{\S^{n-r}}}.
\end{aligned}
\end{multline*}
Then we make polar changes of variables: for each $i$ we set $\rho_i=\Norm{L_i^\perp}$ and $\alpha_i = \displaystyle \frac{L_i^\perp}{\Norm{L_i^\perp}}$. Note that, when $L_1,\dots,L_{i-1}$ are fixed, $L_i^\perp$ is a vector in a space of dimension $n-i$. We have:
\begin{multline*}
\esp{\phi\left(\theta',\odet{L'}\right)}\\
\begin{aligned}
&=\int \phi\left(\theta',\prod_{i=1}^{r-1} \rho_i\right) \frac{e^{-\frac{1}{2}\sum_{i=1}^{r-1} \rho_i^2}\prod_{i=1}^{r-1} \left(\rho_i\right)^{n-i}}{(2\pi)^\frac{(2n+2-r)(r-1)}{4}} \frac{\dx \rho_1 \dots \dx \rho_{r-1} \dx \alpha_1 \dots \dx \alpha_r \dx \theta'}{\vol{\S^{r-1}}\vol{\S^{n-r}}}\\
&= \int_{\rho_1,\dots,\rho_{r-1},\theta'} \phi\left(\theta',\prod_{i=1}^{r-1} \rho_i\right) \prod_{i=1}^{r-1} \left(\vol{\S^{n-i}} \frac{e^{-\frac{\rho_i^2}{2}}\left(\rho_i\right)^{n-i}}{(2\pi)^\frac{n+1-i}{2}}\right) \dx \rho_1 \dots \dx \rho_{r-1}\frac{\dx \theta'}{\vol{\S^{r-1}}}.
\end{aligned}
\end{multline*}
This shows that $\theta',\rho_1,\dots,\rho_{r-1}$ are independent variables, that $\theta'$ is uniformly distributed in $\S(V')$ and that, for all $i\in \{1,\dots,r-1\}$, $\rho_i$ is distributed as the norm of a standard Gaussian vector in $\R^{n+1-i}$. Finally, this shows that $\odet{L'}$ is distributed as $\displaystyle\prod_{i=1}^{r-1} \rho_i$.

Putting all we have done so far together, we see that $\left(\odet{L},L^{\dagger *}U\right)$ is distributed as $\left(\Norm{X_n}\cdots \Norm{X_{n-r+1}},\displaystyle\frac{\rho \theta'}{\Norm{X_{n-r+1}}} \right)$, where $X_p$ is a standard Gaussian vector in $\R^p$ for all $p$. Moreover, $\theta'$ is uniformly distributed in $\S(V')$, $\rho$ is distributed as the norm of a standard Gaussian vector in $\R^r$, and all these variables are globally independent. Finally $U'=\rho \theta'$ is a standard Gaussian in $V'$, independent of $X_n,\dots,X_{n-r+1}$ so we have proved Proposition~\ref{proposition same distribution}. An immediate corollary of this is the following.

\begin{cor}
\label{corollary distribution orthogonal det}
Let $V$ and $V'$ be two Euclidean spaces of dimension $n$ and $r$ respectively, with $1 \leq r \leq n$. Let $L \sim \mathcal{N}(0,\Id)$ in $V' \otimes V^*$. Then $\odet{L}$ is distributed as $\Norm{X_{n-r+1}}\cdots \Norm{X_n}$, where for all $p \in \{n-r+1,\dots,n\}$, $X_p \sim \mathcal{N}(0,\Id)$ in~$\R^p$ and these vectors are globally independent.
\end{cor}

\section{Proof of the Kac--Rice formula}
\label{section proof of Kac-Rice}

In this appendix, we give a proof of the Kac--Rice formula using Federer's coarea formula. This was already done by Bleher, Shiffman and Zelditch in \cite[thm.~4.2]{BSZ2001}. See also \cite[chap.~6]{AW2009}.

\subsection{The coarea formula}
\label{subsection coarea formula}

We start by stating the coarea formula in the case of a smooth map between smooth Riemannian manifolds. A proof in this special case can be found in \cite[Appendix]{How1993} (see \cite[thm.~3.2.12]{Fed1996} for the general case).

Let $\pi : \tilde{M} \to M$ be a smooth map between smooth Riemannian manifolds of respective dimensions $m$ and $n$. We assume that $m\geq n$. Let $\rmes{\tilde{M}}$ (resp.~$\rmes{M}$) denote the Riemannian measure on $\tilde{M}$ (resp.~$M$) induced by its metric. By Sard's theorem, for almost every $y \in M$, $\pi^{-1}(y)$ is a smooth submanifold of dimension $(m-n)$ of $\tilde{M}$. For such $y \in M$, we denote by $\rmes{y}$ the Riemannian measure on $\pi^{-1}(y)$ induced by the metric of $\tilde{M}$. When $m=n$, the dimension $\pi^{-1}(y)$ is $0$ and $\rmes{y}$ is just $\displaystyle\sum_{x \in \pi^{-1}(y)} \delta_x$, where $\delta_x$ is the Dirac measure~at~$x$.

\begin{thm}[Coarea formula, Federer]
\label{theorem coarea formula}
Let $\pi:\tilde{M}\to M$ be a smooth map between smooth Riemannian manifolds of dimension $m$ and $n$ respectively, with $m \geq n$. Let $\phi : \tilde{M} \to \R$ be a Borel measurable function. Then:
\begin{equation*}
\int_{x \in \tilde{M}} \phi(x) \odet{d_x\pi}\rmes{\tilde{M}} = \int_{y \in M}\left( \int_{x \in \pi^{-1}(y)} \phi(x)\rmes{y}\right) \rmes{M},
\end{equation*}
whenever one of these integrals is well-defined.
\end{thm}
\noindent
Note that the innermost integral on the right-hand side is only defined almost everywhere.

\subsection{The double-fibration trick}
\label{subsection double fibration trick}

We now describe the double-fibration trick, which consists in applying the coarea formula twice, for different fibrations. Let $M_1$ and $M_2$ be two smooth Riemannian manifolds of dimension $n_1$ and $n_2$ respectively. Let $F:M_1 \times M_2 \to \R^r$ be a smooth submersion, and let $\Sigma = F^{-1}(0)$. We equip $\Sigma$ with the restriction of the product metric on $M_1\times M_2$ and denote by $\rmes{M_1}$, $\rmes{M_2}$ and $\rmes{\Sigma}$ the Riemannian measures on the corresponding manifolds. Finally, let $\pi_1 : \Sigma \to M_1$ and $\pi_2 : \Sigma \to M_2$ be the projections from $\Sigma$ to each factor. Assuming that $r \leq \min(n_1,n_2)$, we have $\dim(\Sigma) = n_1+n_2-r \geq \max(n_1,n_2)$. Thus we can apply the coarea formula both to $\pi_1$ and $\pi_2$.

Let $\phi:\Sigma \to \R$ be a Borel measurable function, then:
\begin{multline}
\label{equation double fibration formula 1}
\int_{y_1 \in M_1} \left( \int_{x \in \pi_1^{-1}(y_1)} \phi(x) \rmes{y_1} \right) \rmes{M_1} = \int_{x \in \Sigma} \phi(x) \odet{d_x\pi_1} \rmes{\Sigma}\\
= \int_{y_2 \in M_2} \left( \int_{x \in \pi_2^{-1}(y_2)} \phi(x) \frac{\odet{d_x\pi_1}}{\odet{d_x\pi_2}}\rmes{y_2} \right) \rmes{M_2},
\end{multline}
whenever one of these integrals is well-defined. Note that if $\odet{d_x\pi_2}$ vanishes then $\pi_2(x)$ is a critical value of $\pi_2$, and the set of such critical values has measure $0$ in $M_2$.

We would like the integrand on the right-hand side to depend on $F$ rather than on $\pi_1$ and $\pi_2$. Let $\partial_1F$ and $\partial_2F$ denote the partial differentials of $F$ with respect to the first and second variable respectively. For any $x=(x_1,x_2) \in \Sigma$,
\begin{equation*}
T_x\Sigma=\left\{(v_1,v_2) \in T_{x_1}M_1 \times T_{x_2}M_2 \mid \partial_1F(x) \cdot v_1 + \partial_2F(x) \cdot v_2 = 0\right\}.
\end{equation*}

\begin{lem}
\label{lemma exchanging orthogonal determinant}
Let $x \in \Sigma$, then $\odet{d_x\pi_2}=0$ if and only if $\odet{\partial_1F(x)}=0$. Moreover,
\begin{equation}
\label{equation exchanging orthogonal determinant}
\odet{d_x\pi_1}\odet{\partial_1F(x)}=\odet{d_x\pi_2}\odet{\partial_2F(x)}.
\end{equation}
\end{lem}

\begin{proof}
First note that $d_xF = \partial_1F(x) \circ d_x\pi_1 + \partial_2F(x) \circ d_x\pi_2 = 0$ on $T_x\Sigma$. This shows that:
\begin{align}
\label{equation kernel partial derivatives}
\ker(d_x\pi_1)&=\{0\}\times \ker(\partial_2F(x))& &\text{and} & \ker(d_x\pi_2)&= \ker(\partial_1F(x))\times \{0\}.
\end{align}
The space $T_x\Sigma$ splits as the following orthogonal direct sum:
\begin{equation}
\label{equation orthogonal splitting}
T_x\Sigma = \ker(d_x\pi_1) \oplus \ker(d_x\pi_2) \oplus G,
\end{equation}
where $G$ is the orthogonal complement of $\ker(d_x\pi_1) \oplus \ker(d_x\pi_2)$ in $T_x\Sigma$.

Then, $\odet{d_x\pi_2} = 0$ if and only if $d_x\pi_2$ is not onto. Recalling that $\dim(M_2)=n_2$ and $\dim(\Sigma)=n_1+n_2-r$, this is equivalent to $\dim(\ker(d_x\pi_2))> n_1-r$. In the same way, $\odet{\partial_1F(x)}=0$ if and only if $\dim(\ker(\partial_1F(x)))>n_1-r$. But the kernels of $d_x\pi_2$ and $\partial_1F(x)$ have the same dimension by \eqref{equation kernel partial derivatives}, so that $\odet{d_x\pi_2}=0$ if and only if $\odet{\partial_1F(x)}=0$. A similar argument shows that $\odet{d_x\pi_1}=0$ if and only if $\odet{\partial_2F(x)}=0$. Thus the lemma is true if any of the four maps in \eqref{equation exchanging orthogonal determinant} is singular.

From now on, we assume that these maps are all surjective. In this case, we have:
\begin{align*}
\dim(\ker(\partial_2F(x)))&= \dim(\ker(d_x\pi_1))=n_2-r,\\
\dim(\ker(\partial_1F(x))) &= \dim(\ker(d_x\pi_2))=n_1-r,\\
\intertext{and}
\dim(G) &= r.
\end{align*}
We choose an orthonormal basis of $T_xM_1$ adapted to $\ker(\partial_1F(x))\oplus\ker(\partial_1F(x))^\perp$ and an orthonormal basis of $T_xM_2$ adapted to  $\ker(\partial_2F(x))\oplus\ker(\partial_2F(x))^\perp$. From these, we deduce orthonormal bases of
\begin{align*}
\ker(d_x\pi_1)&=\{0\}\times \ker(\partial_2F(x)) & &\text{and} & \ker(d_x\pi_2)&= \ker(\partial_1F(x))\times \{0\}.
\end{align*}
Finally we complete the resulting basis of $\ker(d_x\pi_1) \oplus \ker(d_x\pi_2)$ in an orthonormal basis of $T_x\Sigma$ adapted to the splitting~\eqref{equation orthogonal splitting}. In these bases the matrix of $d_x\pi_1$ has the form $\begin{pmatrix} 0 & I_{n_1-r} & A_1 \\ 0 & 0 & B_1 \end{pmatrix}$ where $I_{n_1-r}$ stands for the identity matrix of size $n_1-r$. Similarly, the matrix of $d_x\pi_2$ has the form $\begin{pmatrix} I_{n_2-r} & 0 & A_2 \\ 0 & 0 & B_2 \end{pmatrix}$, and the matrices of $\partial_1F(x)$ and $\partial_2F(x)$ have the form $\begin{pmatrix} 0 & C_1 \end{pmatrix}$ and $\begin{pmatrix} 0 & C_2 \end{pmatrix}$ respectively. Thus $B_1$, $B_2$, $C_1$ and $C_2$ are square matrices satisfying the following relations:
\begin{align*}
\odet{d_x\pi_1} &= \norm{\det(B_1)}, & \odet{\partial_1F(x)} &= \norm{\det(C_1)}, \\ \odet{d_x\pi_2} &= \norm{\det(B_2)}, & \odet{\partial_2F(x)} &= \norm{\det(C_2)}.
\end{align*}
Besides the relation $\partial_1F(x) \circ d_x\pi_1 + \partial_2F(x) \circ d_x\pi_2 = 0$ means that $C_1B_1=-C_2B_2$, hence $\norm{\det(C_1)}\norm{\det(B_1)}=\norm{\det(C_2)}\norm{\det(B_2)}$. This proves~\eqref{equation exchanging orthogonal determinant}.
\end{proof}

An immediate consequence of \eqref{equation double fibration formula 1} and \eqref{equation exchanging orthogonal determinant} is the following.

\begin{prop}
\label{proposition double fibration formula}
Let $M_1$ and $M_2$ be two smooth Riemannian manifolds of dimension $n_1$ and $n_2$ respectively. Let $F:M_1 \times M_2 \to \R^r$ be a smooth submersion, and let $\Sigma = F^{-1}(0)$. Let $\phi : \Sigma \to \R$ be a Borel measurable function. Then:
\begin{equation*}
\int_{y_1 \in M_1}\hspace{-1mm} \left( \int_{\pi_1^{-1}(y_1)} \hspace{-1mm}\phi(x) \rmes{y_1} \right)\hspace{-1mm} \rmes{M_1} = \int_{y_2 \in M_2}\hspace{-1mm} \left( \int_{\pi_2^{-1}(y_2)} \hspace{-1mm}\phi(x) \frac{\odet{\partial_2F(x)}}{\odet{\partial_1F(x)}}\rmes{y_2} \right)\hspace{-1mm} \rmes{M_2},
\end{equation*}
whenever one of these integrals is well-defined.
\end{prop}

\subsection{Proof of Theorem~\ref{theorem Kac-Rice}}
\label{subsection proof of Kac-Rice}

Finally, we prove Theorem~\ref{theorem Kac-Rice}. Let $M$ be a closed Riemannian manifold of dimension $n$ and $V$ be a subspace of $\mathcal{C}^\infty(M,\R^r)$ of dimension $N$ (recall that $1 \leq r \leq n$). We assume that $V$ is $0$-ample, so that $F:(f,x) \mapsto f(x)$ is a smooth submersion from $V\times M$ to $\R^r$ and
\[\Sigma = F^{-1}(0) = \{(f,x) \in V \times M \mid f(x) = 0\}\]
is a submanifold of codimension $r$ of $V\times M$. Let $\dx f$ denote the Lebesgue measure on $V$ or on a subspace of $V$. Let $\phi:\Sigma \to \R$ be a Borel measurable function, by Proposition~\ref{proposition double fibration formula},
\begin{multline*}
\esp{\int_{x \in Z_f} \phi(f,x) \rmes{f}}= \frac{1}{(2\pi)^\frac{N}{2}} \int_{f \in V} \left(\int_{x\in Z_f} \phi(f,x) e^{-\frac{\Norm{f}^2}{2}} \rmes{f} \right) \dx f\\
= \frac{1}{(2\pi)^\frac{N}{2}} \int_{x \in M} \left(\int_{f \in \ker(j_x^0)} \phi(f,x) e^{-\frac{\Norm{f}^2}{2}} \frac{\odet{\partial_2F(x)}}{\odet{\partial_1F(x)}} \dx f \right)\rmes{M}.
\end{multline*}

Recall that for all $x \in M$, $j_x^0 : f \mapsto f(x)$ is onto, since $V$ is $0$-ample. In particular, $\ker(j_x^0)$ has codimension $r$ and $V$ splits as $\ker(j_x^0) \oplus \ker(j_x^0)^\perp$. We recognize the innermost integral to be a conditional expectation given $f(x)=0$ (see Corollary~\ref{corollary conditional expectation}). Thus,
\begin{equation}
\label{eq kac}
\esp{\int_{x \in Z_f} \phi(f,x) \rmes{f}}= \frac{1}{(2\pi)^\frac{r}{2}} \int_{x \in M} \espcond{\phi(f,x) \frac{\odet{\partial_2F(x)}}{\odet{\partial_1F(x)}}}{f(x)=0}\rmes{M}.
\end{equation}

By equation~\eqref{equation partial differentials of F}, $\odet{\partial_2F(x)} = \odet{d_xf}$ and
\begin{equation}
\label{kac2}
\odet{\partial_1F(x)}= \odet{j_x^0}=\sqrt{\det\left(j_x^0 (j_x^0)^*\right)}.
\end{equation}
Since $f \sim \mathcal{N}(0,\Id)$, equation~\eqref{equation pushforward variance operator} shows that $j_x^0j_x^{0*}$ is the variance operator of $f(x) = j_x^0(f)$. Then, by~\eqref{equation var t}, $\det\left(j_x^0 (j_x^0)^*\right) =\det(\var(f(x)))=\det(E(x,x))$. In particular, this quantity does not depend on $f \in V$. Equations~\eqref{eq kac}, \eqref{kac2} and this last equality prove Theorem~\ref{theorem Kac-Rice}.

% Bibliography
\bibliographystyle{amsplain}
\bibliography{ExpectedVolumeAndEulerCharacteristic}

\end{document}